\documentclass[reqno,11pt]{amsart}
\usepackage{amsfonts,amsmath,amssymb}
\usepackage{graphicx}

\usepackage{tikz}

\usepackage{hyperref}

\usepackage{amsmath,amsthm,amssymb,amscd,mathrsfs,mathtools,color,soul}
\usepackage[shortlabels]{enumitem}

\mathtoolsset{showonlyrefs}

\usepackage[margin=1.30in]{geometry}
\numberwithin{equation}{section}

\def\T{\mathbb{T}}
\def\Z{\mathbb{Z}}

\def\div{\mathrm{div}}

\def\Z{\mathbb{Z}}

\def\div{{\,\mathrm{div}\,}}

\def\wt{\widetilde}
\def\bar{\overline}
\def\R{{\mathbb R}}

\def\jk{\langle {\bf k} \rangle}

\newtheorem{theorem}{Theorem}[section]
\newtheorem{lemma}[theorem]{Lemma}

\newtheorem{proposition}[theorem]{Proposition}
\newtheorem{definition}[theorem]{Definition}

\author{Yu Deng}
\address{University of Chicago}
\email{yudeng@uchicago.edu}

\author{Alexandru D. Ionescu}
\address{Princeton University}
\email{aionescu@math.princeton.edu}

\author{Fabio Pusateri}
\address{University of Toronto}
\email{fabiop@math.toronto.edu}

\begin{document}

\title[Wave kinetic equations for water waves]{Local 
well-posedness of wave kinetic equations associated with water waves}

\begin{abstract}
We prove two conceptually different local well-posedness results
for wave kinetic equations associated with water-wave systems.
For pure gravity waves with 2D interfaces, we establish
a classical well-posedness theory in weighted Sobolev spaces:
regularity is propagated, and solutions depend continuously on the
initial data in the same topology.

For gravity-capillary waves with one-dimensional interfaces,
the situation is fundamentally different. 
We prove well-posedness in weighted $L^\infty$ spaces
(in the sense of existence, uniqueness, and continuity in a weaker topology),
but we also show that, in general, any additional regularity of the initial data is instantly lost.
Moreover, continuity in time and continuous dependence
on the initial data fail in the weighted $L^\infty$ norm,
even for smooth compactly supported initial data. 
This instability is caused by degenerate three-wave resonances,
which produce singular collision rates and arbitrarily short
relaxation time scales. The distinction between the two theories reflects a fundamental
difference in their resonant geometry.
\end{abstract}

\maketitle

\setcounter{tocdepth}{1}

\tableofcontents

\section{Introduction}

Wave kinetic equations describe the slow evolution of the distribution
of energy among the frequencies of weakly nonlinear random waves.
Their collision operators account for the exchange of energy through
resonant interactions.
In this paper we study two such equations associated with water waves:
the gravity-capillary equation for one-dimensional interfaces and the
pure gravity equation for two-dimensional interfaces.
Although both equations admit a local well-posedness theory, the
meaning of this theory is quite different in the two cases.

For the 2D gravity equation, we prove a classical form of well-posedness, showing that Sobolev regularity is propagated and the
solution depends continuously on the initial data in the same topology.
On the other hand, for the 1D gravity-capillary equation, we prove that bounded solutions exist uniquely,
but additional smoothness of the initial data is lost instantaneously.

To state our theorems precisely we introduce first the water-wave system from which these
kinetic equations arise.
We consider an incompressible, irrotational fluid of infinite depth,
whose free surface is the graph of a function $h(t,x)$.
The fluid domain is
\[
\Omega_t
 := \big\{(x,y)\in\T_R^{d-1}\times\R:\ y<h(t,x)\big\},
\qquad d=2,3,
\]
where $\T_R:=\R/(2\pi R\mathbb Z)$.
We denote by $\phi(t,x)$ the trace of the velocity potential on the
free surface. In the Zakharov formulation, the surface variables $(h,\phi)$ satisfy the system 
(see \cite{SulemBook,LannesBook})
\begin{equation}\label{WWE1}
\left\{
\begin{array}{l}
\partial_t h = G(h) \phi,
\\
\partial_t \phi = -g h  + \sigma \div \Big[ \dfrac{\nabla h}{ (1+|\nabla h|^2)^{1/2} } \Big]
  - \dfrac{1}{2} {|\nabla \phi|}^2 + \dfrac{{\left( G(h)\phi + \nabla h \cdot \nabla \phi \right)}^2}{2(1+{|\nabla h|}^2)},
\end{array}
\right.
\end{equation}
where $h, \phi : I\times \T_R^{d-1}\rightarrow \R$, $d=2$ or $3$,
\begin{equation}
\label{defG0}
G(h) := \sqrt{1+{|\nabla h|}^2} \mathcal{N}(h),
\end{equation}
and $\mathcal{N}(h)$ is the Dirichlet-Neumann map associated to the domain $\Omega_t$.
The Hamiltonian of the system \eqref{WWE1} is the physical energy
\begin{align}
\mathcal{H}(h,\phi) = \frac{g}{2}\int_{\T_R^{d-1}} h^2 + 
  \sigma \int_{\T_R^{d-1}} \sqrt{1+|\nabla h|^2} 
  + \frac{1}{2}\int_{\T_R^{d-1}}  \, \phi  G(h)\phi.
\end{align}

In dimension $d=2$ and assuming that the gravity constant and the surface tension coefficient are normalized to $g=1$ and $\sigma=1$, the dispersion relation is $\Lambda(\xi):=\sqrt{|\xi|+|\xi|^3}$, and the associated Wave Kinetic Equation (WKE) is 
\begin{equation}\label{kingeneral2.5}
\partial_t\psi(t,\xi)=\mathcal{Q}[\psi(t),\psi(t)](\xi),
\end{equation}
where $\psi:I\times\R\to [0,\infty)$ is the average energy density, and $\mathcal{Q}$ is the quadratic collision operator associated to gravity-capillary water waves,
\begin{align}\label{kingeneral2}
\begin{split}
&\mathcal{Q}[\psi,\psi]:=\mathcal{Q}_1[\psi,\psi]+\mathcal{Q}_2[\psi,\psi]+\mathcal{Q}_3[\psi,\psi]+\mathcal{Q}_4[\psi,\psi]+\mathcal{Q}_5[\psi,\psi],\\
&\mathcal{Q}_1[\psi,\psi](\xi):=\frac{1}{2\pi} \int_{\R} \frac{|\xi||\xi-\rho||\rho|}{8}\delta_0(\Lambda(\xi)-\Lambda(\xi-\rho)-\Lambda(\rho))
  \psi(\rho)\psi(\xi-\rho) \, d\rho,\\
&\mathcal{Q}_2[\psi,\psi](\xi):=\frac{1}{2\pi} \int_{\R}\frac{|\xi|^2|\rho|(1+|\xi|^2)}{4(1+|\xi+\rho|^2)}
  \delta_0(\Lambda(\xi+\rho)-\Lambda(\xi)-\Lambda(\rho))\psi(\rho)\psi(\xi+\rho) \, d\rho,\\
&\mathcal{Q}_3[\psi,\psi](\xi):=-\frac{1}{2\pi} \int_{\R} \frac{\sqrt{|\xi|}|\rho||\xi-\rho|^{3/2}\sqrt{1+|\xi-\rho|^2}}{4\sqrt{1+|\xi|^2}}
\delta_0(\Lambda(\xi)-\Lambda(\xi-\rho)-\Lambda(\rho)) \psi(\rho)\psi(\xi) \, d\rho,\\
& \mathcal{Q}_4[\psi,\psi](\xi):=-\frac{1}{2\pi} \int_{\R} \frac{|\xi|^{3/2}|\rho|\sqrt{|\xi+\rho|}\sqrt{1+|\xi|^2}}{4\sqrt{1+|\xi+\rho|^2}}
\delta_0(\Lambda(\xi+\rho)-\Lambda(\xi)-\Lambda(\rho))\psi(\rho)\psi(\xi) \, d\rho,\\
& \mathcal{Q}_5[\psi,\psi](\xi):=\frac{1}{2\pi} \int_{\R} \frac{|\xi|^{3/2}|\rho|^{3/2}\sqrt{1+|\xi|^2}\sqrt{1+|\rho|^2}}{4(1+|\xi+\rho|^2)}\delta_0(\Lambda(\xi+\rho)-\Lambda(\rho)-\Lambda(\xi))\psi(\xi+\rho)\psi(\xi) \, d\rho.
\end{split}
\end{align}
In Appendix \ref{secderivation} we will give a formal derivation of \eqref{kingeneral2}
starting from \eqref{WWE1} and following the standard procedure based on 
the iteration of Duhamel's formula and averaging with respect to random initial conditions.

For two-dimensional interfaces ($d=3$) with $g=1$ and $\sigma=0$,
three-wave resonances are absent, and the leading kinetic equation
is the cubic four-wave model
\begin{align}\label{kingeneralg}
\begin{split}
\partial_t f(t,{\bf k}) = \int_{\R^6} \mathcal{M}({\bf k},{\bf k_1},{\bf k_2},{\bf k_3}) 
  \, \delta\big( \Phi({\bf k},{\bf k_1},{\bf k_2},{\bf k_3})\big)
 \delta\big(\Sigma({\bf k},{\bf k_1},{\bf k_2},{\bf k_3})\big) 
 \\
 \times \big[ f_1f_2f_3 + ff_2f_3 - ff_1f_3 - ff_1f_2  \big] \, d{\bf k_1}d{\bf k_2}d\bf{k_3}.
\end{split}
\end{align}
Here $f_j=f(t,{\bf k}_j)$,
$\Phi=\sqrt{|{\bf k}|}+\sqrt{|{\bf k}_1|}
-\sqrt{|{\bf k}_2|}-\sqrt{|{\bf k}_3|}$,
and $\Sigma={\bf k}+{\bf k}_1-{\bf k}_2-{\bf k}_3$.
The relevant properties of the interaction symbol $\mathcal M$
are listed in \eqref{2dMsymm}--\eqref{2dMbound}.

\subsection{Main result in $1$d}
To state our first main result, we define the Banach spaces
\begin{equation}\label{bana1}
Z_A^\alpha:=\big\{f\in L^\infty(\R):\, f\text{ is continuous on }A,\, \|f\|_{Z_A^\alpha}:=\|\langle\xi\rangle^\alpha f(\xi)\|_{L^\infty(\R)}<\infty\big\},
\end{equation}
where $\alpha\geq 0$ is an exponent and $A\subseteq\R$ is an open set. For any interval $I\subseteq\R$ we define
\begin{equation}\label{bana2}
\begin{split}
Z_A^\alpha(I):=\big\{g\in L^\infty(I\times A):\, &g\text{ is continuous on }I\times A,\\
&\|g\|_{Z_A^\alpha(I)}:=\|\langle\xi\rangle^\alpha g(t,\xi)\|_{L^\infty(I\times A)}<\infty\big\}.
\end{split}
\end{equation}
For any $R>0$ let 
\begin{equation}\label{bana3}
\begin{split}
B_R^+(Z_A^\alpha)&:=\{f\in Z_A^\alpha:\, f\geq 0\text{ in }A\text{ and }\|f\|_{Z_A^\alpha}<R\},\\
B_R^+(Z_A^\alpha(I))&:=\{g\in Z_A^\alpha(I):\, g\geq 0\text{ in }I\times A\text{ and }\|g\|_{Z_A^\alpha(I)}<R\},
\end{split}
\end{equation}
denote the balls of radius $R$ in the spaces $Z_A^\alpha$ and $Z_A^\alpha(I)$, restricted to non-negative functions.

Our main result concerns the local well-posedness theory 
of the kinetic equation \eqref{kingeneral2.5}--\eqref{kingeneral2}. 

\begin{theorem}\label{MainThm1} 
Assume $\alpha\geq 1$ is a fixed exponent, $\overline{E}_0$ 
is the exceptional set of iterated degenerate points defined in \eqref{overE} below, and let $A_0:=\R\setminus\overline{E}_0$. Then the following hold:

(i) (Local existence and uniqueness of solutions)  For any $R>0$ there is $T=T(R,\alpha)\gtrsim_\alpha R^{-1}$ such that for initial data $\psi_0\in B_R^+(Z_{A_0}^\alpha)$ there is a unique solution $\psi\in B_{C_\alpha R}^+(Z_{A_0}^\alpha([0,T]))$ of the wave kinetic equation \eqref{kingeneral2.5}--\eqref{kingeneral2} with $\psi(0)=\psi_0$, where $C_\alpha$ depends only on $\alpha$.

(ii) (Restricted continuity) For any $\epsilon\in(0,0.1)$, $M\geq1$, 
and any $\psi_0,\psi'_0\in B_R^+(Z_{A_0}^\alpha)$ we have
\begin{equation}\label{lip1}
\sup_{t\in[0,T],\,M^{-1}\leq |x|\leq M,\,x\in\R\setminus\overline{E}_\epsilon}\langle x\rangle^{\alpha}|\psi'(t,x)-\psi(t,x)|\lesssim_{\epsilon,M,R,\alpha}\|\psi'_0-\psi_0\|_{Z_{A_0}^{\alpha}}
\end{equation}
where $\overline{E}_\epsilon$ is the thickened set of iterated degenerate points defined in \eqref{overE}.
\end{theorem} 

The weaker topology in Theorem \ref{MainThm1}(ii) is essential, even for smooth, compactly supported initial data. In fact, we have the following precise theorem:

\begin{theorem}\label{MainThm2} (Instability) Assume $\alpha\geq1$ and $R>0$ are fixed. Then there is an initial datum $\psi_0\in B_R^+(Z_{A_0}^\alpha)$, given by the restriction of a nonnegative function in $C_c^\infty(\R)$, such that
\begin{equation}\label{lip2}
\limsup_{t\to 0,\,x\to\gamma_1,\,x\in A_0}|\psi(t,x)-\psi_0(x)|\gtrsim_\alpha R,
\end{equation}
where $\gamma_1=\sqrt2$. In particular, the map $t\mapsto\psi(t)$ is not continuous from $[0,T]$ to $Z_{A_0}^\alpha$.

Moreover, there is a sequence of initial data $\psi_0^n\in B_R^+(Z_{A_0}^\alpha)$, also given by restrictions of nonnegative functions in $C_c^\infty(\R)$, such that
\begin{equation}\label{lip3}
\begin{gathered}
\lim_{n\to\infty}\|\psi_0^n-\psi_0\|_{Z_{A_0}^\alpha}=0,\\
\limsup_{t\to 0,\,x\to-\gamma_1,\,x\in A_0}
|\psi^n(t,x)-\psi(t,x)|\gtrsim_\alpha R
\qquad\text{for every }n\geq1,
\end{gathered}
\end{equation}
where $\psi,\psi^n\in B_{C_\alpha R}^+(Z_{A_0}^\alpha([0,T]))$ are the solutions corresponding to $\psi_0,\psi_0^n$. In particular, the solution map is not continuous at $\psi_0$ from $B_R^+(Z_{A_0}^\alpha)$ to $B_{C_\alpha R}^+(Z_{A_0}^\alpha([0,T]))$.
\end{theorem}

\subsection{Main result in $2$d} We present now our main result on the well-posedness of \eqref{kingeneralg}. Let $L^{2,s}$ be the standard space $L^2(\jk^{2s} d{\bf k})$ on $\R^2$,
and let us fix $N_1 \geq 30$.
We introduce the weighted Sobolev space 
\begin{equation}\label{defZ0}
Z_0 
:= \big\{f\in L^2(\R^2) \, : \, 
  {\|f\|}_{Z_0} := \sum_{0 \leq |\ell| \leq 4} 
 \big\| |{\bf k}|^{|\ell|} \nabla_{\bf k}^\ell f({\bf k}) 
  \big\|_{L^{2,N_1-2|\ell|}}<\infty\big\}.
\end{equation}
Note that in the norm in \eqref{defZ0} we apply the zero-homogeneous 
operator $|{\bf k}| \nabla_{\bf k}$, and that
for each application of it we discount two weights.
In the proof of our Theorem \ref{MainThmg}, instead of the norm \eqref{defZ0}, 
we will use the equivalent norm $Z$ defined in \eqref{2dZnorm} 
which employs the more convenient (static) scaling and rotation `vector-fields' 
in \eqref{2dZfields} instead of $|{\bf k}| \nabla_{\bf k}$.

For any interval $I$ and radius $R >0$ we define  $Z_0(I)$ and $B^+_R(Z_0)$, $B^+_R(Z_0(I))$ as follows:
\begin{equation}\label{bana2g}
\begin{split}
Z_0(I) := \big\{ g\in L^\infty(I;L^2(\R^2)) \, :\, &\|g\|_{Z_0(I)}
  := \sup_{t\in I} {\| g(t,\cdot)\|}_{Z_0}<\infty\big\},
\end{split}
\end{equation}
and
\begin{equation}\label{bana3g}
\begin{split}
& B_R^+(Z_0) := \{f\in Z_0: \, f\geq 0 \text{ in } \R^2 \text{ and } \|f\|_{Z_0}<R\},
\\
& B_R^+(Z_0(I)) := \{ g \in Z_0(I):\, g\geq 0 \text{ in } I\times \R^2 \text{ and }\|g\|_{Z_0(I)}<R\}.
\end{split}
\end{equation}

\begin{theorem}\label{MainThmg} 
Consider the equation \eqref{kingeneralg} with
a symbol $\mathcal{M}$ satisfying \eqref{2dMsymm}-\eqref{2dMbound},
and let $Z_{\rho}$ with $\rho \in (0,4]$ be the spaces defined in \eqref{2dZnorm-},
that is, the slightly weaker versions of the space $Z_0$ with $\rho$ fewer ${\bf k}$ weights.
Then, the following hold:

(i) (Local existence and uniqueness of solutions) 
For any $R>0$ there is $T = T(R) \gtrsim R^{-2}$ such that for initial data $f_0\in B_R^+(Z_0)$ 
there exists a unique solution 
\begin{align}
f \in B_{CR}^+(Z_0([0,T])) \, \cap \, 
  C([0,T];Z_0)\, \cap \mathrm{Lip}([0,T]; Z_2)
\end{align}
of the wave kinetic equation \eqref{kingeneralg} with $f(0)=f_0$.

(ii) (Continuity of the flow)
On every bounded set $B_0 \subset B^{+}_R(Z_0)$, $R>0$, the solution map
$f_0 \in B_0 \longrightarrow f \in Z_0([0,T])$ is continuous. 
Moreover, the solution map is uniformly continuous when we equip the domain and target with the $Z_\rho$ ($0<\rho\leq4$) topology
and it is Lipschitz on $Z_2$:
for any two solutions $f,g \in B_{CR}^+(Z_0([0,T]))$ we have
\begin{align}
\sup_{t\in[0,T]} {\| f(t) - g(t) \|}_{Z_2} \lesssim {\| f(0) - g(0) \|}_{Z_2}.
\end{align}
\end{theorem}

\subsection{Water waves and wave turbulence}
The local well-posedness theory for the water-wave equations is by now well developed, 
both with and without surface tension; we refer, among many other works, to 
\cite{CL,Lannes,Wu2,CS2,ShZ3,ABZ2}. Global or long-time results for
small and sufficiently regular perturbations of a flat interface have
also been proved in several two- and three-dimensional settings; 
see, for example, \cite{ADa,DIPP,GMS2,IT,IoPu2,IoPu4,Wu3DWW} on the whole space 
as well as the periodic results \cite{BertiDelort,BeFeFr,IoPu5,ZhengWW}.
We refer to \cite{IoPuRev,LannesBook} for broader surveys of the mathematical theory.

Wave turbulence theory provides a statistical description of large ensembles of weakly 
nonlinear dispersive waves. Its central object is a kinetic equation governing 
the slow evolution of the expected energy density in frequency space. 
An early three-wave kinetic model appeared in Peierls' work on phonon interactions \cite{Pei}. 
For surface gravity waves, the kinetic equation was derived in pioneering work 
by Hasselmann \cite{Has1,Has2}. The Hamiltonian formulation of the water-wave
equations due to Zakharov \cite{Zak0} and the subsequent development of 
the theory of Kolmogorov--Zakharov spectra played a fundamental role in the subject. 
We refer to \cite{Naz} for a general introduction to weak wave turbulence and its applications.

For the gravity-wave kinetic equation, Pan and Wu \cite{PanWu}
proved local existence of $L^1$ strong solutions with weighted $L^2\cap L^\infty$ bounds.
Here we establish the full Sobolev well-posedness theory for this equation
and a sharply contrasting weighted $L^\infty$ theory for the
one-dimensional gravity-capillary model.
Following \cite{PanWu}, Wang \cite{WangCapillary} also established 
an existence result for the pure capillary equation.

\subsubsection{Wave kinetic equations \normalfont{(WKE)}}
For a general wave kinetic equation, the degree of the nonlinearity 
in the collision operator is determined by the first nontrivial resonant interaction. 
When nontrivial quadratic resonances are present, the leading effective dynamics are described
by a three-wave kinetic equation, quadratic in the spectral density.
This is the situation for the one-dimensional gravity-capillary dispersion relation
$\Lambda(\xi)=\sqrt{|\xi|+|\xi|^3}$ considered in \eqref{kingeneral2}. 
By contrast, quadratic resonances are absent for pure gravity waves, which have 
dispersion relation $\sqrt{|\xi|}$, so that quadratic terms can be
eliminated at the level of the underlying dispersive equation. 
The leading kinetic dynamics are then generated by four-wave resonances
and give the cubic equation \eqref{kingeneralg}, usually referred to 
as the Hasselmann equation \cite{Has1,Has2}.

The derivation of a wave kinetic equation from an underlying
dispersive model formally involves a weak-nonlinearity limit together with a large-box limit, 
and an averaging or closure assumption on the random Fourier modes. 
Rigorous results in this direction were first obtained for semilinear models. 
For the cubic nonlinear Schr\"odinger equation, the onset of the kinetic description was
analyzed in \cite{BGHS2}, and a complete derivation at the kinetic time scale 
was obtained in a series of papers by Deng and Hani \cite{DeHa1,DeHa2,DeHa4}. 
Related derivations and extensions appear in \cite{ColGer1,ColGer2,StafTran,XMa}; see also
\cite{AmpColGer} for an inhomogeneous quadratic model. 
In one space dimension, Vassilev \cite{Va} proved propagation of randomness 
on long time-scales for the Majda--McLaughlin--Tabak model \cite{MMT}.

For the full water-wave equations, a rigorous derivation is expected to be considerably 
more delicate because the system is quasilinear. Well-posedness of the target equation is a key ingredient, and this is what we prove here, but a direct iteration of Duhamel's formula produces an
unavoidable loss of derivatives, and a path to the rigorous derivation of the WKE 
has not been established yet. 
For 1D interface gravity waves, a strategy combining deterministic quasilinear energy estimates,
normal-form transformations, and probabilistic estimates was developed in \cite{DIP,DIP2}
to prove propagation of randomness on long-time scales.

In Appendix \ref{secderivation} we will give a formal derivation of \eqref{kingeneral2.5}-\eqref{kingeneral2}
for the purpose of identifying the precise collision operator associated with the
gravity-capillary system. A complete justification of this kinetic limit is a separate problem.

\subsubsection{Well-posedness of WKE}
An important line of research in mathematical wave turbulence concerns 
the deterministic analysis of the kinetic equations themselves,
and this is our focus here. In \cite{GIT} Germain, Ionescu and Tran established an almost critical local well-posedness theory
for a class of spatially homogeneous four-wave equations with radial dispersion relations. 
Nguyen and Tran \cite{NgTr} constructed global radial solutions for the three-wave capillary equation
with viscosity.
For inhomogeneous equations, global well-posedness near vacuum and
hierarchy methods were developed in \cite{Amp,AmpMPT}; see also the derivation in \cite{AmpColGer}.
Well-posedness and ill-posedness thresholds in weighted spaces
have recently been investigated in \cite{AmpLegIll,AmpLegLr}, 
and well-posedness for some semilinear MMT models was established in \cite{GLZ-MMT}.

The collision kernels for the WKE derived from the full water-wave models present additional
difficulties that are not covered by the works cited above. 
More precisely, the equations studied in this paper exhibit two further features:

\setlength{\leftmargini}{1.5em}
\begin{itemize}

\item First, in the one-dimensional gravity-capillary equation, the resonant
manifold has a degeneracy at a distinguished frequency \(\gamma_1=\sqrt{2}\). 
The effective collision kernel is unbounded there in the natural weighted \(L^\infty\) topology 
and may instantaneously create discontinuities at a family of frequencies
generated by the resonant relations. 

\smallskip
\item 
For both the one-dimensional gravity-capillary equation \eqref{kingeneral2.5}-\eqref{kingeneral2} 
and the two-dimensional pure-gravity equation \eqref{kingeneralg}, a key difficulty
comes from the  growth of the interaction symbol, due to the quasilinear
nature of the original PDE.
This leads to an unavoidable loss of weights when one tries to propagate momentum regularity, similar to the loss of derivatives for the PDE, which rules out the application of fixed-point arguments.
\end{itemize}

To overcome the two types of difficulties described above 
we implement two different arguments based on energy estimates,
both exploiting the dissipative structure of the collision operator, as we now proceed to describe.

\subsection{Main difficulties and ideas of the proof}\label{ssecstrategy}
The collision operators in the WKE are supported on resonant manifolds, 
as imposed by the Dirac distributions expressing conservation of momentum and linear energy. 
They contain gain and loss terms of comparable size,
and estimating these terms separately would generally discard 
the cancellations responsible for non-negativity of the solution and dissipativity. 
A basic principle in the proofs of our main well-posedness results 
is to preserve the gain--loss structure and to use $L^p$ energy estimates 
for the full collision operator to exploit certain dissipative structures 
and overcome the unboundedness of the associated symbols.

Observe that resolving the energy Dirac distribution produces an inverse Jacobian involving a
derivative of the corresponding nonlinear phase. 
Thus, the effective collision kernel may become singular at degenerate points of the resonant manifold. 
This is the situation for \eqref{kingeneral2.5}-\eqref{kingeneral2}.
In this case we are able to exploit the dissipative 
structure to compensate for the singular Jacobian near a degenerate triad. 
In the two-dimensional problem instead
we use the dissipative structure highlighted in Pan and Wu \cite{PanWu},
and carry out energy estimates with the addition of derivatives (in the momentum variable ${\bf k}$)
to compensate for the growth of the four-wave interaction.


\subsubsection{Gravity-capillary waves in $1d$}
The first difficulty is the geometry of the three-wave resonances.
For positive frequencies, consider the phase
\begin{equation}\label{introPhase1d}
 \Phi(\xi,\rho)
 :=\Lambda(\xi)-\Lambda(\xi-\rho)-\Lambda(\rho).
\end{equation}
At $\xi = \gamma_1:=\sqrt{2}, \rho = \gamma_1/2,$
one has
\begin{equation}\label{introDeg1d}
 \Phi\left(\gamma_1,\frac{\gamma_1}{2}\right)= 0 = \partial_\rho\Phi
 \left(\gamma_1,\frac{\gamma_1}{2}\right), \qquad 
 \partial_\rho^2\Phi
 \left(\gamma_1,\frac{\gamma_1}{2}\right)\neq0,
\end{equation}
that is, a degeneracy at which two simple resonant roots coalesce. 
After resolving $\delta_0(\Phi(\xi,\rho))$, the corresponding inverse Jacobian
satisfies, schematically,
\begin{equation}\label{introJac1d}
 \frac{1}
 {|\partial_\rho\Phi(\xi,\rho_\pm(\xi))|}
 \approx \frac{1}{|\xi-\gamma_1|^{1/2}}
\end{equation}
on the side of $\gamma_1$ on which the two resonant branches $\xi \mapsto \rho_{\pm}(\xi)$ exist.
The singular factor \eqref{introJac1d} explains why the collision operator is not
locally Lipschitz in an ordinary weighted $L^\infty$ space
and suggests that it can produce fast initial growth near $\gamma_1$. 
A schematic ODE capturing this mechanism is
\begin{align}\label{introODE}
 \partial_t u(t,\xi) = a(\xi)\big(b-u(t,\xi)\big), \qquad a(\xi)\approx |\xi-\gamma_1|^{-1/2},
\end{align}
where the frozen coefficient $b$ represents the contributions of the other modes entering a resonant triad.
The solution of \eqref{introODE} is
\begin{align}
u(t,\xi) = b + \big(u_0(\xi)-b\big)e^{-t a(\xi)}.
\end{align}
For every fixed $\xi \neq \gamma_1$, this converges to $u_0(\xi)$ as $t \to 0$. 
On the other hand, if $t_n \to 0$ and $\xi_n \to\gamma_1$ are chosen so that
$t_n a(\xi_n) \to + \infty$, then $u(t_n,\xi_n)\longrightarrow b$.
So unless the initial data agree with the effective equilibrium value $b$, 
joint continuity at $(t,\xi)=(0,\gamma_1)$ is lost. 
This is the basic mechanism behind Theorem \ref{MainThm2}. 

The singularity at $\gamma_1$ is transported to other frequencies 
through the nonlinear interactions, by the resonant maps obtained by solving $\Phi = 0$. 
Iterating these maps generates the exceptional set $\overline E_0$ defined in \eqref{overE}. 
This motivates the spaces $Z_{A_0}^{\alpha}$, in which continuity is required
where the resonant parametrizations remain regular, while
discontinuities are allowed at frequencies reached by the iterated degeneracies. 
Away from the thickened set $\overline E_\epsilon$, the relevant Jacobians are uniformly
nondegenerate and the resonant maps have uniformly controlled derivatives, 
so that one can prove the restricted stability estimate \eqref{lip1}.

For the proof of Theorem \ref{MainThm1}, 
we first regularize the collision kernel near the degenerate frequencies and derive estimates in
weighted $L^p$ spaces for finite $p$. 
We do this by multiplying the equation by
$\langle\xi\rangle^{\alpha p}\psi^{p-1}$ and then changing variables
to place the different gain and loss terms on the same resonant manifold. 
Symmetrizing the resulting expression reveals a leading contribution with a favorable sign. 
The remaining terms are controlled by the frequency weights and by the precise bounds for
the resonant maps and their Jacobians.
An important point is that the estimates are uniform both in the regularization parameter 
and as $p\to\infty$. 
We then proceed by first restricting the equation to
resonant-closed sets (see Definition \ref{defResClosed})
away from the degenerate frequencies and adding damping;
using again the precise structure of the equation we show that damping can be removed,
and finally remove the regularization exhausting the resonant-closed sets to include
all non-degenerate frequencies.
The limiting $L^\infty$ estimate gives the required $Z_{A_0}^{\alpha}$ bound.

Uniqueness and restricted stability are obtained by organizing
the difference equation along resonant chains. After excluding
$\overline E_\epsilon$, the coefficients still grow along each chain,
but slowly enough that iteration of the integral inequalities
gives a convergent series and yields \eqref{lip1}. For Theorem \ref{MainThm2}, we choose one smooth, compactly supported datum.
Its positive-frequency component produces fast relaxation near $\gamma_1$,
while its negative-frequency component has stationary resonant triples near
$-\gamma_1$. Adding small smooth perturbations near $-\gamma_1/2$ triggers
fast relaxation in these stationary triples and yields the order-one
separation in \eqref{lip3}.

\subsubsection{Gravity waves in $2d$}

For the four-wave equation \eqref{kingeneralg}, the resonance
manifold has higher dimension and does not exhibit the singularity described above. 
The main obstruction is instead the growth and nonlocality of the interaction kernel. 
Let us write \eqref{kingeneralg} as
\begin{align}\label{intro2ddtf}
\partial_t f(t,{\bf k}) = \mathcal{Q}(f)({\bf k}), \qquad f \geq 0.
\end{align}
First, following the structural idea of Pan and Wu \cite{PanWu}, 
we linearize the cubic collision operator:
given a non-negative background $g$, we define a (linear) operator 
$\mathcal{Q}_g[\cdot]$ such that $\partial_t f = \mathcal{Q}_f [f].$
Using the permutation symmetries of the collision kernel (see \eqref{2dMsymm})
we decompose
\begin{equation}\label{intro2dDecomp}
 \mathcal{Q}_g = \mathcal{L}_g + \mathcal{B}_g,
\end{equation}
where $\mathcal{L}_g$ is dissipative in $L^p$, while $\mathcal{B}_g$ is bounded. 
The dissipative part contains the terms in which the largest momenta are coupled
through differences of the form
$h({\bf k}')-h({\bf k})$, where ${\bf k}'$ is the integration variable, and ${\bf k}$ the output,
and keeping this difference intact is essential. 

When trying to prove a priori energy estimates based on the structure \eqref{intro2dDecomp},
one notices that weighted $L^p$ spaces do not suffice, since the inputs with low momenta
need to be estimated in $L^\infty$.
Therefore, we add derivatives to our functional framework and aim
to prove a priori estimates in a suitable weighted Sobolev space.
Since regular Cartesian derivatives can fall on the interaction symbol and on the
parametrization of the resonant manifold altering the structure of the equation
and leading to loss of weights,
we instead use the static scaling and rotation vector-fields (see \eqref{2dZfields}). 
These fields are very well adapted to the equation 
thanks to the homogeneity and rotational invariance of the momentum and energy constraints,
and of the equation.
Commuting these vector-fields through the collision
operator preserves the basic resonant structure, up to acceptable remainders. 
Note that iterated scaling and rotation fields are equivalent to the operators
$|{\bf k}|^{|\ell|}\nabla_{\bf k}^{\ell}$ appearing in the space \eqref{defZ0}.

The frequency analysis is organized so that the factors carrying the
lowest momenta are placed in $L^\infty$, while the factor on which
the energy estimate is performed remains in weighted $L^2$.
The $L^\infty$ bounds are obtained from the $L^2$-based bounds via a Sobolev--Klainerman-type embedding.
Moreover, each commutation may cost up to two powers of the frequency weight; 
we account for this through the decreasing sequence $N_1-2|\ell|$ appearing in \eqref{defZ0},
so that all the terms that do not count as dissipative in \eqref{intro2dDecomp} can be controlled by
weighted Sobolev embeddings. 

The energy inequality that we obtain in Section \ref{secgLWP} has the schematic form
\begin{equation}\label{intro2dEnergy}
 \frac{d}{dt}\|f(t)\|_{Z}^{2}
 \lesssim \|f(t)\|_{Z}^{4},
\end{equation}
where $Z$ is the space defined in \eqref{2dZnorm} equivalent to $Z_0$.
\eqref{intro2dEnergy} is clearly consistent with the lifespan $T\gtrsim R^{-2}$ in
Theorem \ref{MainThmg}. Solutions are then constructed through the linearized iteration
$\partial_t f_{n+1} =\mathcal{Q}_{f_n}[f_{n+1}], f_{n+1}(0)=f_0.$
The dissipative component controls the potentially unbounded part of
the generator while the bounded component is treated perturbatively.

The difference estimate corresponding to the above iteration
closes with two fewer frequency weights and gives the Lipschitz bound in a weaker topology
but, of course, it cannot be applied in the strongest topology.
To recover continuity in the top topology, and obtain our full well-posedness results, 
we then use a Bona--Smith approximation argument \cite{BS75}.
This allows us to prove that solutions are continuous in time with values in $Z$ 
and also yields continuity (without establishing uniform continuity) of the flow map from 
the space of data into the space of solutions.

\subsection{Structure of the paper}
The first part of the paper is dedicated to the gravity-capillary wave kinetic equation 
\eqref{kingeneral2.5}-\eqref{kingeneral2}.
We begin by analyzing the resonant manifold in Section \ref{ResonMan}, defining the exceptional 
set $\overline{E}_0$, and introducing the concept of `resonant-closed' sets;
see Definition \ref{defResClosed}, \eqref{ken1}, \eqref{overE} and Lemma \ref{lemkenE}.
In Section \ref{sec:energy} we define our main weighted $L^p$ energy (see \eqref{en1})
and prove a priori energy estimates 
on resonant-closed sets that do not intersect the exceptional set $\overline{E}_0$;
these estimates are uniform in $p$, as well as in the thickening parameter $\epsilon$, giving us a priori control on the $Z_{A_0}^\alpha$ norm.
In Section \ref{secprth} we prove Theorems \ref{MainThm1} and \ref{MainThm2}.

The second part of the paper is dedicated to the kinetic equation \eqref{kingeneralg} 
associated to the gravity water waves system for $2$d interfaces.
In Section \ref{secgsetup} we first analyze the model, its linearization and
its decomposition into dissipative and bounded parts. We then introduce 
our functional framework based on vector-fields and establish all the necessary trilinear estimates
on the collision kernel.
Then, in Section \ref{secgLWP} we use these estimates to obtain a priori energy bounds 
in our weighted $Z$ space. Finally, we use a regularization and approximation scheme
to build local-in-time solutions and prove the full well-posedness result in Theorem \ref{MainThmg}.

Finally, in Appendix \ref{secderivation} we give a formal derivation of the wave kinetic equation
\eqref{kingeneral2.5}-\eqref{kingeneral2}, as this did not seem to be present in the literature.

\subsection{Acknowledgements} 
Y. D. was supported by NSF DMS-2531437 and a Simons Collaboration Grant on Wave Turbulence.
A. D. I. was supported in part by NSF-UEFISCDI grant DMS-2407694
and a Simons Collaboration Grant on Wave Turbulence. 
F. P. was supported in part by
NSERC grants RGPIN-2018-06487 and RGPIN-2025-06419, and a Simons Collaboration Grant on Wave Turbulence.

\subsection{AI use disclosure} 
This paper was written by the authors and was not generated by artificial
intelligence. OpenAI's ChatGPT (Sol 5.6 and Astra 6) and Anthropic's Claude (Opus 5.5) assisted with language
editing, organization, bibliographic checks, and checking calculations.
All mathematical content is the authors' own; they reviewed all AI-assisted
revisions and take full responsibility for the paper.


\bigskip
\section{Analysis of the gravity-capillary kinetic wave equation in 1D}

\subsection{The resonant manifold}\label{ResonMan}

In this section we prove several properties of  solutions of the resonant equation. Let
\begin{equation}\label{riv0.5}
\lambda(x):=\sqrt{x+x^3},\qquad \lambda:\R_+\to\R_+.
\end{equation}
Notice that $\lambda$ is a $C^\infty$ function on $(0,\infty)$ and
\begin{equation}\label{riv0.6}
\lambda'(x)=\frac{1+3x^2}{2(x+x^3)^{1/2}},\qquad \lambda''(x)=\frac{3x^4+6x^2-1}{4(x+x^3)^{3/2}}.
\end{equation}

We start with a simple lemma:

\begin{lemma}\label{RIV0}
(i) Assume that $y\leq x\in (0,\infty)$ are solutions of the equation
\begin{equation}\label{riv1}
\lambda(x)+\lambda(y)=\lambda(x+y).
\end{equation}
Let $\gamma_1:=\sqrt{2}$. Then $y\leq \gamma_1/2\leq x$ and $x+y\geq\gamma_1$. Moreover, letting $S:=x^2+y^2$ and $P:=2xy$, we have
\begin{equation}\label{riv2}
P\in(8/9,1],\qquad S=\frac{8-7P^2}{9P-8}.
\end{equation}

(ii) If $x\geq\gamma_1/2$ then there is a unique number $p(x)\in [\gamma_1,\infty)$ such that $p(x)>x$ and
\begin{equation}\label{riv3}
\lambda(p(x)-x)+\lambda(x)=\lambda(p(x)).
\end{equation}

(iii) The map $p_+(x):=p(x)$ is an increasing $C^{\infty}$ bijection $p_+: [\gamma_1/2,\infty)\to [\gamma_1,\infty)$. Similarly, the map $p_-(x):=p(x)-x$ is a decreasing $C^{\infty}$ bijection $p_-: [\gamma_1/2,\infty)\to (0,\gamma_1/2]$. Moreover
\begin{equation}\label{riv3.5}
\lambda'(p_-(x))\leq\lambda'(x)\leq\lambda'(p_+(x))\qquad\text{ for any }x\in[\gamma_1/2,\infty).
\end{equation}
\end{lemma}

\begin{proof} (i) Both sides of \eqref{riv1} are positive, so by squaring we get $2\lambda(x)\lambda(y)=3xy(x+y)$. By squaring again we see that \eqref{riv1} is equivalent to the identity
\begin{equation}\label{eq:R2}
 4(1+x^{2})(1+y^{2})=9 xy\,(x+y)^{2}.
\end{equation}
In terms of $S$ and $P$ this becomes
\[
  4+4S+P^2=9P(S+P)/2\iff
  8+8S=9PS+7P^{2}.
\]
In particular $P> 8/9$ and the identity \eqref{riv2} follows. Since $S\geq P$, this identity also shows that $P\leq 1$, as claimed.

Notice that
\[
  (x+y)^{2}-2=S+P-2=\frac{8-7P^2}{9P-8}+P-2=\frac{2(P-1)(P-12)}{9P-8}\geq 0
\]
which shows that $x+y\ge\sqrt2=\gamma_{1}$. Since $y\le x$, it follows that $x\geq\gamma_1/2$. Finally, $xy=P/2\le 1/2$, so $y\le\gamma_{1}/2$, as claimed.

(ii) Fix $x\ge\gamma_{1}/2$ and define the function $H:[0,\infty)\to\R$ by
\[
  H(y):=\lambda(x)+\lambda(y)-\lambda(x+y).
\]
Notice that $H$ is continuous on $[0,\infty)$, $H(0)=0$, and $H'(y)=\lambda'(y)-\lambda'(x+y)$ for $y>0$. In particular $H'(y)>0$ for $y>0$ small, so $H(y)>0$ for $y\in(0,\delta_0)$, $\delta_0\ll 1$. Since $\lim\limits_{y\to\infty}H(y)=-\infty$ it follows that there is at least one number $y_0\in(0,\infty)$ with the property that $H(y_0)=0$. Moreover $y_0\in (0,\gamma_1/2]$ and $x+y_0\in [\gamma_1,\infty)$, due to part (i).

We show now that at every point $y_0\in (0,\gamma_1/2]$ for which $H(y_0)=0$ we have
\begin{equation}\label{slope}
H'(y_0)=\lambda'(y_0)-\lambda'(x+y_0)<0.
\end{equation}
Indeed, if $y_0\leq\gamma_1/2$ then we use the resonance relation $\lambda(x+y_0)=\lambda(x)+\lambda(y_0)$ and $\lambda(x)=\frac{3xy_0(x+y_0)}{2\lambda(y_0)}$ to see that
\begin{equation*}
\begin{split}
  \lambda'(y_0)<\lambda'(x+y_0)
  &\iff (1+3y_0^{2})\big(\lambda(x)+\lambda(y_0)\big)<\big(1+3(x+y_0)^{2}\big)\lambda(y_0)\\
  &\iff (1+3y_0^{2})\lambda(x)<3x(x+2y_0)\lambda(y_0)\\
  &\iff (1+3y_0^{2})(x+y_0)<2(x+2y_0)(1+y_0^{2}).
\end{split}  
\end{equation*}
This is equivalent to $x(1-y_0^{2})+3y_0+y_0^3>0$, which is true because $y_0<1$, thus proving \eqref{slope}.

We show now that $H$ has only one $0$ in $(0,\gamma_1/2]$. Indeed, if $H$ had two zeros $y_{1}<y_{2}$, then
$H'(y_{1})<0$, so $H$ is strictly negative just to the right of $y_{1}$. To reach the next zero
$y_{2}$ it would have to return to $0$ from below, forcing $H'(y_{2})\ge0$, thus contradicting
\eqref{slope}. 

Hence the zero $y_0\in(0,\gamma_1/2]$ is unique, \(y_0<\gamma_1/2\) if \(x>\gamma_1/2\), and \(y_0=\gamma_1/2\)
when \(x=\gamma_1/2\). We set
$p(x):=x+y_0\in[\gamma_1,\infty)$.

(iii) Let $F(x,p):=\lambda(p-x)+\lambda(x)-\lambda(p)$, which is a $C^{\infty}$ function on the open set
$\{(x,p)\in(0,\infty)^2:\,x<p\}$. At a resonant point $(x,p(x))$ we have
\[
  \partial_{p}F\big(x,p(x)\big)=\lambda'(p(x)-x)-\lambda'(p(x))=H'(p(x)-x)<0
\]
by \eqref{slope}. The implicit function theorem shows that $p$ is smooth and 
\begin{equation}\label{pprime}
  p'(x)=-\frac{\partial_{x}F(x,p(x))}{\partial_{p}F(x,p(x))}
  =\frac{\lambda'(x)-\lambda'\big(p(x)-x\big)}{\lambda'\big(p(x)\big)-\lambda'\big(p(x)-x\big)} .
\end{equation}

Clearly $p(\gamma_1/2)=\gamma_1$ and $p(x)>x$ (so $p(x)\to\infty$ as $x\to\infty$). To show that $p$ is a strictly increasing bijection it remains to prove that $p'(x)>0$ for all $x\in(\gamma_1/2,\infty)$. We use \eqref{pprime}, so it suffices to prove that 
\begin{equation}\label{slope2}
\lambda'(x)>\lambda'(y_0),
\end{equation}
where, as before, $y_0\in(0,\gamma_1/2)$ is the unique solution of the resonant equation $H(y_0)=0$. As in (ii) this is equivalent to the inequality $(1+3y_0^{2})\lambda(x)<\big(1+3x^{2}\big)\lambda(y_0)$, which is equivalent to the inequality $g>0$, where
\[
  g:=2(1+3x^2)(1+y_0^2)-3x(1+3y_0^2)(x+y_0)
   =2+2y_0^2+3x^2-3xy_0-3x^2y_0^2-9xy_0^3 .
\]
Using $P=2xy_0$, $S=x^2+y_0^2$ (so $xy_0=P/2$, $x^2y_0^2=P^2/4$, $xy_0^3=\tfrac{P}{2}y_0^2$) and
isolating the non-symmetric part via $x^2-y_0^2$,
\[
  g=2-\frac{3P}{2}-\frac{3P^2}{4}
   +\frac{(10-9P)\,S+(2+9P)\,(x^2-y_0^2)}{4}.
\]
Since $2+9P>0$ and $x^2-y_0^2>0$, we can drop the last term and use \eqref{riv2} to see that
\[
  g>\ 2-\frac{3P}{2}-\frac{3P^2}{4}+\frac{(10-9P)(8-7P^2)}{4(9P-8)}
   =\frac{(P-1)(P-2)(9P+2)}{9P-8}.
\]
Therefore $g>0$ (due to \eqref{riv2}) so $\lambda'(x)>\lambda'(y_0)$, so $p'(x)>0$ for every $x>\gamma_1/2$ (due to \eqref{pprime}). Thus $p$ is a smooth map on $[\gamma_1/2,\infty)$, $p'>0$ on $(\gamma_1/2,\infty)$, $p(\gamma_1/2)=\gamma_1$, and $p(x)>x$, so the map $p_+(x)=p(x)$ is a strictly
increasing $C^\infty$ bijection from $[\gamma_1/2,\infty)$ to $[\gamma_1,\infty)$.

Since $p_-(x)=p(x)-x$, we have $p_-'(x)=p'(x)-1$. For $x\ge\gamma_1/2$ we have
$x^2\ge1/2$, hence $3x^4+6x^2-1>0$, so $\lambda''(x)>0$ by \eqref{riv0.6}. Thus $\lambda'$ is strictly increasing on $[\gamma_1/2,\infty)$. Since
$\gamma_1/2\le x<p(x)$, this yields $\lambda'(x)<\lambda'(p(x))$, so $p'(x)<1$ for all $x\geq \gamma_1/2$ by \eqref{pprime}. Hence $p_-$ is strictly decreasing and
$C^\infty$, $p_-(\gamma_1/2)=\gamma_1/2$, and $p_-(x)=\frac{P}{2x}\le\frac{1}{2x}\to0$
as $x\to\infty$. Therefore $p_-$ is a strictly decreasing $C^\infty$ bijection from
$[\gamma_1/2,\infty)$ to $(0,\gamma_1/2]$.

The inequality \eqref{riv3.5} follows from \eqref{slope2}, since $\lambda'$ is strictly increasing on $[\gamma_1/2,\infty)$. This completes the proof of (iii). 
\end{proof}

We define now an important concept for the local well-posedness theory.

\begin{definition}\label{defResClosed}
(i) A set $E\subseteq(0,\infty)$ will be called resonant-closed if for any $\zeta\in[\gamma_1/2,\infty)$
the set $\{\zeta, p_-(\zeta), p_+(\zeta)\}$ is either contained in $E$ or is disjoint from $E$.

(ii) For any $\zeta\in[\gamma_1/2,\gamma_1)$ we define the countable set 
\begin{equation}\label{ken1}
E_\zeta:=\big\{\zeta_0,\zeta_1,\zeta_2\ldots,\zeta_{-1},\zeta_{-2}\ldots:\,\zeta_0:=\zeta,\,\zeta_{k+1}:=p_{+}(\zeta_k),\,\zeta_{-k-1}:=p_-(\zeta_k),\,k\in\{0,1,\ldots\}\big\}.
\end{equation}

(iii) We also define the sets $\overline{E}_0$ and $\overline{E}_\epsilon$, $\epsilon\in[0,\gamma_1/2)$, by
\begin{equation}\label{overE}
\begin{split}
&\overline{E}_0:=\big\{\xi\in\R:\,\xi=0\,\,\text{ or }\,\,\xi\in E_{\gamma_1/2}\,\,\text{ or }\,\,-\xi\in E_{\gamma_1/2}\big\},\\
&\overline{E}_\epsilon:=\big\{\xi\in\R:\,\xi=0\,\,\text{ or }\,\,\xi\in\cup_{\zeta\in[\gamma_1/2,\gamma_1/2+\epsilon]}E_{\zeta}\,\,\text{ or }\,\,-\xi\in \cup_{\zeta\in[\gamma_1/2,\gamma_1/2+\epsilon]}E_{\zeta}\big\}.
\end{split}
\end{equation}
\end{definition}

\begin{figure}[htbp]
\centering
\begin{tikzpicture}[
  x=1.48cm,
  y=1cm,
  orbitlabel/.style={
    font=\scriptsize,
    align=center,
    inner sep=1pt,
    text=black
  },
  guide/.style={gray!55,densely dashed,thin}
]
\def\ytop{1.0}
\def\ybottom{0}

\newcommand{\upperorbitpoint}[3]{%
  \fill[#3] (#1,\ytop) circle (2pt);
  \draw[guide] (#1,\ytop) -- ++(0,.20)
    node[orbitlabel,above=1pt] {#2};
}

\newcommand{\lowerorbitpoint}[3]{%
  \fill[#3] (#1,\ybottom) circle (2pt);
  \draw[guide] (#1,\ybottom) -- ++(0,-.20)
    node[orbitlabel,below=1pt] {#2};
}

\draw[thick,->] (0,\ytop) -- (7.85,\ytop);
\draw[thick,->] (0,\ybottom) -- (7.85,\ybottom);

\node[anchor=east] at (-.12,\ytop) {$E_{\gamma_1/2}$};
\node[anchor=east] at (-.12,\ybottom) {$E_{\zeta}$};

\node[orbitlabel,below=2pt] at (0,\ytop) {$0$};
\node[orbitlabel,below=2pt] at (0,\ybottom) {$0$};

\node[orbitlabel,above=2pt] at (.25,\ytop) {$\cdots$};
\node[orbitlabel,above=2pt] at (7.58,\ytop) {$\cdots$};
\node[orbitlabel,above=2pt] at (.25,\ybottom) {$\cdots$};
\node[orbitlabel,above=2pt] at (7.58,\ybottom) {$\cdots$};

\upperorbitpoint{.65}
  {$\gamma_{-4}$\\[-2pt]$0.2315$}
  {blue!70!black}

\upperorbitpoint{1.45}
  {$\gamma_{-3}$\\[-2pt]$0.2702$}
  {blue!70!black}

\upperorbitpoint{2.25}
  {$\gamma_{-2}$\\[-2pt]$0.3415$}
  {blue!70!black}

\upperorbitpoint{3.35}
  {$\gamma_{-1}=\gamma_0$\\[-2pt]
   $=\gamma_1/2\approx0.7071$}
  {red!75!black}

\upperorbitpoint{4.65}
  {$\gamma_1$\\[-2pt]$1.4142$}
  {blue!70!black}

\upperorbitpoint{5.55}
  {$\gamma_2$\\[-2pt]$1.7557$}
  {blue!70!black}

\upperorbitpoint{6.40}
  {$\gamma_3$\\[-2pt]$2.0259$}
  {blue!70!black}

\upperorbitpoint{7.22}
  {$\gamma_4$\\[-2pt]$2.2574$}
  {blue!70!black}

\lowerorbitpoint{.65}
  {$\zeta_{-4}$\\[-2pt]$0.2306$}
  {teal!70!black}

\lowerorbitpoint{1.45}
  {$\zeta_{-3}$\\[-2pt]$0.2688$}
  {teal!70!black}

\lowerorbitpoint{2.25}
  {$\zeta_{-2}$\\[-2pt]$0.3385$}
  {teal!70!black}

\lowerorbitpoint{3.05}
  {$\zeta_{-1}$\\[-2pt]$0.6186$}
  {teal!70!black}

\lowerorbitpoint{3.85}
  {$\zeta_0=\zeta$\\[-2pt]$0.8071$}
  {red!75!black}

\lowerorbitpoint{4.65}
  {$\zeta_1$\\[-2pt]$1.4257$}
  {teal!70!black}

\lowerorbitpoint{5.55}
  {$\zeta_2$\\[-2pt]$1.7642$}
  {teal!70!black}

\lowerorbitpoint{6.40}
  {$\zeta_3$\\[-2pt]$2.0330$}
  {teal!70!black}

\lowerorbitpoint{7.22}
  {$\zeta_4$\\[-2pt]$2.2637$}
  {teal!70!black}

\end{tikzpicture}

\caption{The two resonant orbits $E_{\gamma_1/2}$ and $E_\zeta$,
where $\gamma_1/2=\sqrt{2}/2$ and
$\zeta=\sqrt{2}/2+0.1$.
The negative-index points accumulate at $0$, while the
positive-index points tend to infinity.
Horizontal spacing is schematic; the displayed decimal values
are the actual point locations.}
\label{fig:resonant-orbits}
\end{figure}
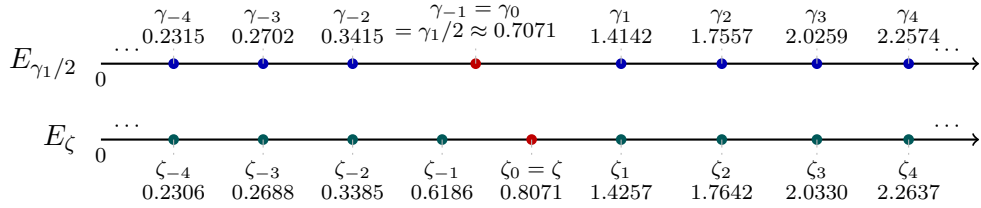

\begin{lemma}\label{lemkenE}
(i) For any $\zeta\in[\gamma_1/2,\gamma_1)$ the set $E_\zeta$ is the smallest resonant-closed set that contains $\zeta$. Moreover, $E_{\zeta_1}\cap E_{\zeta_2}=\emptyset$ if $\zeta_1\neq\zeta_2\in[\gamma_1/2,\gamma_1)$. 

(ii) If $B\subseteq [\gamma_1/2,\gamma_1)$ is a compact set then the set 
\begin{equation}\label{ken10}
E(B):=\cup_{\zeta\in B}E_\zeta
\end{equation}
is a closed subset of $(0,\infty)$. Moreover, if $B=[\gamma_1/2,\gamma_1)$ then $E(B)=(0,\infty)$.
\end{lemma}

\begin{proof}[Proof of Lemma \ref{lemkenE}]
Set $\gamma_0:=\frac{\gamma_1}{2}$.  For $q\in[\gamma_0,\gamma_1)$ it will be convenient to write
\begin{equation}\label{eq:orbit-coordinates}
 x_0(q):=q,\qquad x_{k+1}(q):=p_+(x_k(q)),\qquad
 y_k(q):=p_-(x_k(q)),\qquad k\geq0.
\end{equation}
Thus
\[
 E_q=\{x_k(q),y_k(q):k\geq0\},
\]
where $x_0(\gamma_0)=y_0(\gamma_0)=\gamma_0$ is counted only once.  

(i) We first prove that the sets $E_q$, $q\in[\gamma_0,\gamma_1)$, form a partition of
$(0,\infty)$.  Suppose first that $\xi\geq \gamma_0$.  If $\xi\in[\gamma_0,\gamma_1)$, then
$\xi=x_0(\xi)$.  If $\xi\geq \gamma_1$, then apply $p_+^{-1}$ repeatedly for as long
as the resulting point belongs to $[\gamma_1,\infty)$.  Since the map $p_+^{-1}:[\gamma_1,\infty)\to[\gamma_0,\infty)$ is continuous, strictly increasing, and satisfies $p_+^{-1}(x)<x$, the resulting sequence is strictly decreasing and this procedure terminates
after finitely many steps.  Therefore there is a unique
$q\in[\gamma_0,\gamma_1)$ and a unique $k\geq0$ such that $\xi=x_k(q)$. The uniqueness follows from the injectivity of $p_+$.

If instead $0<\xi<\gamma_0$, we use the fact that
$p_-:[\gamma_0,\infty)\to(0,\gamma_0]$ is a bijection and set
$\eta:=p_-^{-1}(\xi)$.  Applying the preceding argument to $\eta$, we obtain a
unique $q\in[\gamma_0,\gamma_1)$ and $k\geq0$ such that $\eta=x_k(q)$, and hence
\begin{equation}\label{eq:positive-partition-low}
 \xi=p_-(x_k(q))=y_k(q).
\end{equation}
Therefore
\begin{equation}\label{eq:orbit-partition}
 (0,\infty)=\bigcup_{q\in[\gamma_0,\gamma_1)}E_q,\qquad E_{q_1}\cap E_{q_2}=\emptyset\text{ if }q_1\neq q_2\in[\gamma_0,\gamma_1).
\end{equation}

Assume now that $r\geq \gamma_0$, and let $q\in[\gamma_0,\gamma_1)$ and $k\geq0$
be the unique parameters such that $r=x_k(q)$.  Clearly, the triple
$\{r,p_-(r),p_+(r)\}$ is contained in $E_q$.  The uniqueness in
\eqref{eq:orbit-partition} shows that a resonant triple which intersects
$E_q$ must be this type of triple for the same orbit parameter $q$.
Therefore every $E_q$ is resonant-closed.

On the other hand, if a resonant-closed set $F$ contains $q=x_0(q)$, then it contains the
complete triple $\{x_0(q),y_0(q),x_1(q)\}$.  Applying the same argument
successively to $x_1(q),x_2(q),\ldots$, we see that $F$ contains every
$x_k(q)$ and every $y_k(q)$.  Thus $E_q\subseteq F$, which proves that
$E_q$ is the smallest resonant-closed set containing $q$.  

(ii) Since $B\subseteq [\gamma_0,\gamma_1)$ is compact it is bounded away
from $\gamma_1$. For $q\in[\gamma_0,\gamma_1)$ and $k\geq0$, set
\[
 P_k(q):=2x_k(q)y_k(q).
\]
Lemma \ref{RIV0}(i) gives $8/9<P_k(q)\leq1$.  Since
$x_{k+1}(q)=x_k(q)+y_k(q)$, we have
\begin{equation}\label{eq:orbit-growth-square}
 x_{k+1}(q)^2-x_k(q)^2=P_k(q)+y_k(q)^2.
\end{equation}
Moreover, $0<y_k(q)\leq \gamma_0$ and
\begin{equation}\label{eq:orbit-low-high-comparison}
 \frac{4}{9x_k(q)}<y_k(q)\leq\frac{1}{2x_k(q)}.
\end{equation}
By induction over $k$ it follows that
\begin{equation}\label{eq:orbit-growth}
 q^2+\frac{8k}{9}\leq x_k(q)^2\leq q^2+\frac{3k}{2},
 \qquad k\geq0,
\end{equation}
uniformly in $q\in[\gamma_0,\gamma_1)$. In particular, $x_k(q)\to\infty$ and $y_k(q)\to0$ uniformly in
$q\in B$.

For every fixed $k$, the maps $q\mapsto x_k(q)$ and
$q\mapsto y_k(q)$ are continuous.  Therefore all the sets $x_k(B)$ and $y_k(B)$, $k\geq 0$,  are
compact. The uniform limits just proved show that any compact subset of $(0,\infty)$
intersects only finitely many of these sets.  Hence
\[
 E(B)=\bigcup_{k\geq0}\big(x_k(B)\cup y_k(B)\big)
\]
is closed in $(0,\infty)$, which completes the proof of the lemma.
\end{proof}

\subsection{A priori $L^p$ energy estimates}\label{sec:energy}

Assume that $\epsilon\in (0,\gamma_1/4)$ and $A\subseteq (0,\infty)\setminus \overline{E}_\epsilon$ is a resonant-closed open set. Let $\alpha\geq 1$, and fix a non-decreasing $C^1$ weight
$\omega=\omega_\alpha:(0,\infty)\to[1,\infty)$ such that
\begin{equation}\label{en0}
\omega(\xi)=1\quad \hbox{if } \xi\leq 2\gamma_1,\qquad
\omega(\xi)=\langle\xi\rangle^\alpha\quad \hbox{if } \xi\geq 3\gamma_1.
\end{equation}
Assume that $p\geq 50$ is an even integer, and $\psi:[0,T]\times A\to[0,\infty)$ is a non-negative continuous solution of the mollified kinetic equation
\begin{equation}\label{en0.1}
\partial_t\psi(t,\xi) = \mathcal{Q}[\psi(t),\psi(t)](\xi)-\delta\langle\xi\rangle^4\psi(t,\xi),
\end{equation}
where $\delta\in[0,1]$ and $\mathcal{Q}_1+\mathcal{Q}_2+\mathcal{Q}_3+\mathcal{Q}_4+\mathcal{Q}_5$ is the collision operator
defined in \eqref{kingeneral2}. We also note that, for $\xi>0$, every nontrivial resonance
contributing to $\mathcal Q_1,\ldots,\mathcal Q_5$ has $\rho>0$
(and $0<\rho<\xi$ for $\mathcal Q_1$ and $\mathcal Q_3$), so the equation decouples into independent evolutions
on $(0,\infty)$ and $(-\infty,0)$.

\begin{proposition}\label{EnergyProp}
With the notation above, we assume that 
\begin{equation}\label{en0.2}
\sup_{(t,x)\in[0,T]\times A}\langle x\rangle^\alpha|\psi(t,x)|<\infty
\end{equation}
and define the weighted $L^p$ energy functional
\begin{equation}\label{en1}
E_p(t)=E_{p,\omega,A}(t):=\int_A\psi(t,\xi)^{p}\,\omega(\xi)^{p-8}\,d\xi.
\end{equation}
Then there is a constant \(C_{\alpha}\) (independent of $p, \delta, \epsilon$ and of the set $A$) such that for any $t\in[0,T]$
\begin{equation}\label{tora45}
[E_p(t)]^{1/p}\leq [E_p(0)]^{1/p}\exp\Big(C_{\alpha}\int_0^t {\|\psi(s)\|}_{L^\infty(A)}\,ds\Big).
\end{equation}
\end{proposition}

\begin{proof} 
{\bf{Step 1.}} We begin by integrating out explicitly the resonance condition in 
the mollified wave kinetic equation \eqref{en0.1}. Recall the formula
\begin{align}
\label{torev0}
\int_{\R} F(\xi,\rho) \, \delta_0(G(\xi,\rho)) \, d\rho =  \sum_{k=1}^KF(\xi,q_k(\xi)) 
  \frac{1}{|\partial_\rho G(\xi,q_k(\xi))|}
\end{align}
valid for any continuous $F$ and $C^1$ function $G$, provided that the function $G(\xi,\cdot )$ has non-degenerate zeros at $q_1(\xi),\ldots,q_K(\xi)$. Then, adopting the same notation 
as in Lemma \ref{RIV0} for the maps $p_+$ and $p_-$, 
the equation \eqref{en0.1} can be written as
\begin{equation}\label{tora1}
\partial_t\psi 
  = T_1+T_{21}+T_{22}+T_{31}+T_{32}+T_{41}+T_{42}+T_{51}+T_{52}-\delta\langle\xi\rangle^4\psi,
\end{equation}
where, for any $\xi\in (0,\infty)\setminus \{\gamma_1/2,\gamma_1\}$, and using the definitions \eqref{kingeneral2},
\begin{equation}\label{tora2}
\begin{split}
\mathcal{Q}_1[\psi,\psi](\xi) = T_1(\xi) & :=\mathbf{1}_{(\gamma_1,\infty)}(\xi)\mathfrak{n}_1(p_+^{-1}(\xi),\xi-p_+^{-1}(\xi))\psi(p_+^{-1}(\xi))\psi(\xi-p_+^{-1}(\xi)),\\
\mathfrak{n}_1(x,y) & :=\frac{1}{8\pi}\frac{xy(x+y)}{\lambda'(x)-\lambda'(y)},
\end{split}
\end{equation}
\begin{equation}\label{tora3}
\begin{split}
\mathcal{Q}_2[\psi,\psi](\xi) & = T_{21}(\xi) + T_{22}(\xi),\\
T_{21}(\xi) & := \mathbf{1}_{(\gamma_1/2,\infty)}(\xi)\mathfrak{n}_2(p_+(\xi),p_-(\xi))\psi(p_+(\xi))\psi(p_-(\xi)),\\
T_{22}(\xi) & := \mathbf{1}_{(0,\gamma_1/2)}(\xi)\mathfrak{n}_2(p_+(p_-^{-1}(\xi)),p_-^{-1}(\xi))\psi(p_-^{-1}(\xi))\psi(p_+(p_-^{-1}(\xi))),\\
\mathfrak{n}_2(x,y) & :=\frac{1}{8\pi}\frac{y(x-y)^2(1+(x-y)^2)}{(1+x^2)(\lambda'(x)-\lambda'(y))},
\end{split}
\end{equation}
\begin{equation}\label{tora4}
\begin{split}
\mathcal{Q}_3[\psi,\psi](\xi) & = T_{31}(\xi) + T_{32}(\xi),\\
T_{31}(\xi) & :=\mathbf{1}_{(\gamma_1,\infty)}(\xi)
\mathfrak{n}_3\big(\xi,p_+^{-1}(\xi)\big)\psi(\xi)\psi(p_+^{-1}(\xi)),\\
T_{32}(\xi) &  :=\mathbf{1}_{(\gamma_1,\infty)}(\xi)
\mathfrak{n}_3\big(\xi,\xi-p_+^{-1}(\xi)\big)\psi(\xi)\psi(\xi-p_+^{-1}(\xi)),\\
\mathfrak n_3(x,y) & :=\frac{-1}{8\pi}\frac{\sqrt{x}y(x-y)^{3/2}\sqrt{1+(x-y)^2}}{\sqrt{1+x^2}|\lambda'(y)-\lambda'(x-y)|},
\end{split}
\end{equation}
\begin{equation}\label{tora5}
\begin{split}
\mathcal{Q}_4[\psi,\psi](\xi) & = T_{41}(\xi) + T_{42}(\xi),\\
T_{41}(\xi) & :=\mathbf{1}_{(\gamma_1/2,\infty)}(\xi)\mathfrak{n}_4(\xi,p_-(\xi))\psi(\xi)\psi(p_-(\xi)),\\
T_{42}(\xi) & :=\mathbf{1}_{(0,\gamma_1/2)}(\xi)\mathfrak{n}_4(\xi,p_-^{-1}(\xi))\psi(\xi)\psi(p_-^{-1}(\xi)),\\
\mathfrak{n}_4(x,y) & :=\frac{-1}{8\pi}\frac{x^{3/2}y\sqrt{x+y}\sqrt{1+x^2}}{\sqrt{1+(x+y)^2}(\lambda'(x+y)-\lambda'(y))},
\end{split}
\end{equation}
\begin{equation}\label{tora6}
\begin{split}
\mathcal{Q}_5[\psi,\psi](\xi) & = T_{51}(\xi) + T_{52}(\xi),\\
T_{51}(\xi) & :=\mathbf{1}_{(\gamma_1/2,\infty)}(\xi)\mathfrak{n}_5(\xi+p_-(\xi),\xi)\psi(\xi)\psi(\xi+p_-(\xi)),\\
T_{52}(\xi) & :=\mathbf{1}_{(0,\gamma_1/2)}(\xi)\mathfrak{n}_5(\xi+p_-^{-1}(\xi),\xi)
  \psi(\xi)\psi(\xi+p_-^{-1}(\xi)),\\
\mathfrak{n}_5(x,y) & :=\frac{1}{8\pi}\frac{y^{3/2}(x-y)^{3/2}\sqrt{1+y^2}\sqrt{1+(x-y)^2}}{(1+x^2)(\lambda'(x)-\lambda'(x-y))}.
\end{split}
\end{equation}
Then, using the formulas \eqref{tora1}-\eqref{tora6} above
we can calculate 
\begin{equation}\label{tora7}
\partial_tE_p =  p \big( U_1+U_{21}+U_{22}+U_{31}+U_{32}+U_{41}+U_{42}+U_{51}+U_{52}+U_6 \big),
\end{equation}
where, with $A_> := A\cap(\gamma_1/2,\infty)$,
the contribution corresponding to \eqref{tora2} is given by
\begin{align}\label{tora8.0}
U_1 &:=\int_{\gamma_1}^\infty \psi(\xi)^{p-1}\omega(\xi)^{p-8}\cdot \mathfrak{n}_1(p_+^{-1}(\xi),\xi-p_+^{-1}(\xi))\psi(p_+^{-1}(\xi))\psi(\xi-p_+^{-1}(\xi))\,\mathbf{1}_A(\xi)d\xi,
\end{align}
the contributions corresponding to \eqref{tora3} are
\begin{align}
\label{tora9.0}
U_{21} & :=\int_{A_>} \omega(\zeta)^{p-8}\mathfrak n_2(p_+(\zeta),p_-(\zeta))\cdot\psi(\zeta)^{p-1}\psi(p_+(\zeta))\psi(p_-(\zeta))\,d\zeta,
\\
\label{tora10.0}
U_{22} &:=\int_0^{\gamma_1/2}\psi(\xi)^{p-1}\omega(\xi)^{p-8}\cdot\mathfrak n_2(p_+(p_-^{-1}(\xi)),p_-^{-1}(\xi))\psi(p_-^{-1}(\xi))\psi(p_+(p_-^{-1}(\xi)))\,\mathbf{1}_A(\xi)d\xi,
\end{align}
the contributions corresponding to \eqref{tora4} are
\begin{align}
\label{tora11.0}
U_{31} &:=\int_{\gamma_1}^\infty \psi(\xi)^{p-1}\omega(\xi)^{p-8}\cdot\mathfrak n_3(\xi,p_+^{-1}(\xi))\psi(\xi)\psi(p_+^{-1}(\xi))\,\mathbf{1}_A(\xi)d\xi,
\\
\label{tora12.0}
U_{32} &:=\int_{\gamma_1}^\infty \psi(\xi)^{p-1}\omega(\xi)^{p-8}\cdot\mathfrak n_3(\xi,\xi-p_+^{-1}(\xi))\psi(\xi)\psi(\xi-p_+^{-1}(\xi))\,\mathbf{1}_A(\xi)d\xi,
\end{align}
the contributions corresponding to \eqref{tora5} are
\begin{align}
\label{tora13.0}
U_{41} &:=\int_{A_>} \omega(\zeta)^{p-8}\mathfrak n_4(\zeta,p_-(\zeta))\cdot\psi(\zeta)^{p}\psi(p_-(\zeta))\,d\zeta,
\\
\label{tora14.0}
U_{42} &:=\int_0^{\gamma_1/2} \psi(\xi)^{p-1}\omega(\xi)^{p-8}\cdot\mathfrak n_4(\xi,p_-^{-1}(\xi))\psi(\xi)\psi(p_-^{-1}(\xi))\,\mathbf{1}_A(\xi)d\xi,
\end{align}
the contributions corresponding to \eqref{tora6} are
\begin{align}\label{tora15.0}
U_{51} &:=\int_{A_>} \omega(\zeta)^{p-8}\mathfrak n_5(p_+(\zeta),\zeta)\cdot\psi(\zeta)^{p}\psi(p_+(\zeta))\,d\zeta,
\\
\label{tora16.0}
U_{52} &:=\int_0^{\gamma_1/2} \psi(\xi)^{p-1}\omega(\xi)^{p-8}\cdot\mathfrak n_5(\xi+p_-^{-1}(\xi),\xi)\psi(\xi)\psi(\xi+p_-^{-1}(\xi))\,\mathbf{1}_A(\xi)d\xi,
\end{align}
and the contribution corresponding to the last term on the right-hand side of \eqref{tora1} is
\begin{align}
\label{tora16.5.0}
U_{6} & := - \delta \int_A \psi(\xi)^{p}\omega(\xi)^{p-8}\langle\xi\rangle^{4}\,d\xi.
\end{align}
Then, we change variables so that all integrals are defined 
over the common set $A_>$. 
Using \eqref{pprime} to calculate the Jacobians of these changes of variables, we obtain
\begin{align}\label{tora8}
U_1 & = \int_{A_>}  \omega(p_+(\zeta))^{p-8}\mathfrak{n}_1(\zeta,p_-(\zeta))\frac{\lambda'(\zeta)-\lambda'(p_-(\zeta))}{\lambda'(p_+(\zeta))-\lambda'(p_-(\zeta))}\cdot \psi(\zeta)\psi(p_-(\zeta))\psi(p_+(\zeta))^{p-1}\,d\zeta,
\\
\label{tora9}
U_{21} & = \int_{A_>}\omega(\zeta)^{p-8}\mathfrak 
n_2(p_+(\zeta),p_-(\zeta))\cdot\psi(\zeta)^{p-1}\psi(p_+(\zeta))\psi(p_-(\zeta))\,d\zeta,
\\
\label{tora10}
U_{22} & = \int_{A_>} \omega(p_-(\zeta))^{p-8}\mathfrak n_2(p_+(\zeta),\zeta)\frac{\lambda'(p_+(\zeta))-\lambda'(\zeta)}{\lambda'(p_+(\zeta))-\lambda'(p_-(\zeta))}\cdot\psi(p_-(\zeta))^{p-1}\psi(\zeta)\psi(p_+(\zeta))\,d\zeta,
\\
\label{tora11}
U_{31} &= \int_{A_>} \omega(p_+(\zeta))^{p-8}\mathfrak n_3(p_+(\zeta), \zeta)\frac{\lambda'(\zeta)-\lambda'(p_-(\zeta))}{\lambda'(p_+(\zeta))-\lambda'(p_-(\zeta))}\cdot\psi(p_+(\zeta))^{p}\psi(\zeta)\,d\zeta .
\\ 
\label{tora12}
U_{32} & = \int_{A_>} \omega(p_+(\zeta))^{p-8}\mathfrak n_3(p_+(\zeta),p_-(\zeta))\frac{\lambda'(\zeta)-\lambda'(p_-(\zeta))}{\lambda'(p_+(\zeta))-\lambda'(p_-(\zeta))}\cdot\psi(p_+(\zeta))^{p}\psi(p_-(\zeta))\,d\zeta,
\\
\label{tora13}
U_{41} & =\int_{A_>} \omega(\zeta)^{p-8}\mathfrak 
n_4(\zeta,p_-(\zeta))\cdot\psi(\zeta)^{p}\psi(p_-(\zeta))\,d\zeta,
\\
\label{tora14}
U_{42} & =\int_{A_>} \omega(p_-(\zeta))^{p-8}\mathfrak n_4(p_-(\zeta),\zeta)\frac{\lambda'(p_+(\zeta))-\lambda'(\zeta)}{\lambda'(p_+(\zeta))-\lambda'(p_-(\zeta))}\cdot\psi(p_-(\zeta))^{p}\psi(\zeta)\,d\zeta,
\\
\label{tora15}
U_{51} & = \int_{A_>} \omega(\zeta)^{p-8}\mathfrak n_5(p_+(\zeta),\zeta)  
\cdot\psi(\zeta)^{p}\psi(p_+(\zeta))\,d\zeta,
\\
\label{tora16}
U_{52} & = \int_{A_>} \omega(p_-(\zeta))^{p-8}\mathfrak n_5(p_+(\zeta),p_-(\zeta))\frac{\lambda'(p_+(\zeta))-\lambda'(\zeta)}{\lambda'(p_+(\zeta))-\lambda'(p_-(\zeta))}\cdot\psi(p_-(\zeta))^{p}\psi(p_+(\zeta))\,d\zeta.
\end{align}

For $\zeta\in(\gamma_1/2,\infty)$ we set, for simplicity, 
\begin{align}\label{not1}
& \zeta_-:= p_-(\zeta), \qquad \zeta_+:= p_+(\zeta), \qquad (\zeta = \zeta_+ - \zeta_-),
\\
& \mu:=\frac{\lambda(\zeta)}{\lambda(\zeta_+)},\qquad \mu_-:=\frac{\lambda(\zeta_-)}{\lambda(\zeta_+)}, 
  \qquad (\mu + \mu_- = 1),
\\
& K(\zeta):=\frac{1}{8\pi}\frac{\zeta\zeta_-\zeta_+}{\lambda'(\zeta_+)-\lambda'(\zeta_-)}. 
\end{align}

Using the definitions \eqref{tora2}-\eqref{tora6} and the notation \eqref{not1} we write  
explicitly all the symbols in the above integral expressions: from \eqref{tora2} we have
\begin{equation}\label{tora17}
\mathfrak{n}_1(\zeta,p_-(\zeta))\frac{\lambda'(\zeta)-\lambda'(p_-(\zeta))}{\lambda'(p_+(\zeta))-\lambda'(p_-(\zeta))}
  = \frac{1}{8\pi}\frac{\zeta\zeta_-\zeta_+}{\lambda'(\zeta_+)-\lambda'(\zeta_-)} = K(\zeta),
\end{equation}
from \eqref{tora3} we get
\begin{equation}\label{tora18}
\mathfrak{n}_2(p_+(\zeta),p_-(\zeta)) = 
  \frac{1}{8\pi}\frac{\zeta\zeta_-\zeta_+}{\lambda'(\zeta_+)-\lambda'(\zeta_-)}\frac{\lambda(\zeta)^2}{\lambda(\zeta_+)^2} = K(\zeta) \mu^2,
\end{equation}
and
\begin{equation}\label{tora19}
\mathfrak{n}_2(p_+(\zeta),\zeta)\frac{\lambda'(p_+(\zeta))-\lambda'(\zeta)}{\lambda'(p_+(\zeta))-\lambda'(p_-(\zeta))}=\frac{1}{8\pi}\frac{\zeta\zeta_-\zeta_+}{\lambda'(\zeta_+)-\lambda'(\zeta_-)}\frac{\lambda(\zeta_-)^2}{\lambda(\zeta_+)^2} = K(\zeta) \mu_-^2,
\end{equation}
from \eqref{tora4} we have
\begin{equation}\label{tora20}
\mathfrak{n}_3(p_+(\zeta),\zeta)\frac{\lambda'(\zeta)-\lambda'(p_-(\zeta))}{\lambda'(p_+(\zeta))-\lambda'(p_-(\zeta))}=-\frac{1}{8\pi}\frac{\zeta\zeta_-\zeta_+}{\lambda'(\zeta_+)-\lambda'(\zeta_-)}\frac{\lambda(\zeta_-)}{\lambda(\zeta_+)} = - K(\zeta) \mu_-,
\end{equation}
and
\begin{equation}\label{tora21}
\mathfrak{n}_3(p_+(\zeta),p_-(\zeta))\frac{\lambda'(\zeta)-\lambda'(p_-(\zeta))}{
  \lambda'(p_+(\zeta))-\lambda'(p_-(\zeta))}=-\frac{1}{8\pi}\frac{\zeta\zeta_-\zeta_+}{
  \lambda'(\zeta_+)-\lambda'(\zeta_-)}\frac{\lambda(\zeta)}{\lambda(\zeta_+)} = - K(\zeta) \mu,
\end{equation}
from \eqref{tora5} we see that
\begin{equation}\label{tora23}
\mathfrak{n}_4(p_-(\zeta),\zeta)\frac{\lambda'(p_+(\zeta))-\lambda'(\zeta)}{\lambda'(p_+(\zeta))-\lambda'(p_-(\zeta))}=-\frac{1}{8\pi}\frac{\zeta\zeta_-\zeta_+}{\lambda'(\zeta_+)-\lambda'(\zeta_-)}\frac{\lambda(\zeta_-)}{\lambda(\zeta_+)}  = - K(\zeta) \mu_-,
\end{equation}
and
\begin{equation}\label{tora22}
\mathfrak{n}_4(\zeta,p_-(\zeta))=-\frac{1}{8\pi}\frac{\zeta\zeta_-\zeta_+}{\lambda'(\zeta_+)-\lambda'(\zeta_-)}\frac{\lambda(\zeta)}{\lambda(\zeta_+)}  = - K(\zeta) \mu,
\end{equation}
and, finally, from \eqref{tora6} we have
\begin{equation}\label{tora24}
\mathfrak{n}_5(p_+(\zeta),\zeta)=\frac{1}{8\pi}\frac{\zeta\zeta_-\zeta_+}{\lambda'(\zeta_+)-\lambda'(\zeta_-)}\frac{\lambda(\zeta)\lambda(\zeta_-)}{\lambda(\zeta_+)^2}
  = K(\zeta) \mu \mu_-,
\end{equation}
and
\begin{equation}\label{tora25}
\mathfrak{n}_5(p_+(\zeta),p_-(\zeta))\frac{\lambda'(p_+(\zeta))-\lambda'(\zeta)}{\lambda'(p_+(\zeta))-\lambda'(p_-(\zeta))} 
  = \frac{1}{8\pi}\frac{\zeta\zeta_-\zeta_+}{\lambda'(\zeta_+)-\lambda'(\zeta_-)}
  \frac{\lambda(\zeta)\lambda(\zeta_-)}{\lambda(\zeta_+)^2} = K(\zeta) \mu \mu_-.
\end{equation}

\bigskip
{\bf{Step 2.}} For simplicity, let us introduce further short-hand notation
\begin{equation}\label{tora26.5}
\begin{split}
&a:=\psi(\zeta),\qquad a_-:=\psi(\zeta_-),\qquad a_+:=\psi(\zeta_+),
\\
&w:=\omega(\zeta)^{p-8},\qquad w_-:=\omega(\zeta_-)^{p-8},\qquad w_+:=\omega(\zeta_+)^{p-8},
\end{split}
\end{equation}
Since \(\lambda(\zeta_+)=\lambda(\zeta)+\lambda(\zeta_-)\), we have $\mu+\mu_-=1$. 
Using the identity \eqref{tora7} and the rest of the formulas derived in Step 1 above, 
gathering the common factor $K$, we get
\begin{equation}\label{tora27}
\partial_tE_p(t) = p\int_{A_>}K(\zeta)\mathcal B_p(\zeta)\,d\zeta
  - \delta p\int_A\psi(\xi)^{p}\omega(\xi)^{p-8}\langle\xi\rangle^{4}\,d\xi,
\end{equation}
where, grouping together the terms with the same weight, we can write
\begin{equation}\label{tora28}
\begin{split}
\mathcal B_p&=w_+\big[aa_- a_+^{p-1}-\mu_- a a_+^p-\mu a_-a_+^p\big] 
\\
&+w\big[\mu^2a^{p-1}a_+a_--\mu a^p a_-+\mu\mu_- a^pa_+\big] 
\\
&+w_-\big[\mu_-^2a_-^{p-1}aa_+-\mu_- a_-^pa+\mu\mu_- a_-^pa_+\big].
\end{split}
\end{equation}

We now want to estimate the right-hand side of \eqref{tora27} by using
the (negative) sign-definite dissipative structure of some terms to absorb
some of the positive definite contributions that can potentially lead to losses of derivatives
from the fact that $K$ is unbounded for large $\zeta$.
To this end, we begin by noticing that we can factor
\begin{equation}\label{tora29}
\mathcal B_p = \big(aa_--\mu_- aa_+-\mu a_+a_-\big)
  \big(w_+a_+^{p-1}-\mu w a^{p-1}-\mu_- w_-a_-^{p-1}\big).
\end{equation}

\medskip
{\bf{Case 1.}}  Assume first that $\zeta$ is close to $\gamma_1/2$, 
i.e. $\zeta\in(\gamma_1/2,\gamma_\ast]$, with $\gamma_\ast$ such that $p_+(\gamma_\ast)=3\gamma_1/2$. 
In view of the assumption \eqref{en0} we have $w=w_-=w_+=1$ hence
\begin{equation}\label{tora29.0}
\mathcal B_p = \big(aa_--\mu_- aa_+-\mu a_+a_-\big)
  \big(a_+^{p-1} - \mu a^{p-1}-\mu_- a_-^{p-1}\big).
\end{equation}
We show first that
\begin{equation}\label{tora35}
\mathcal B_p\le 2\|\psi\|_{L^\infty}\big(a^p+a_-^p\big),
\end{equation}
where the main feature of this inequality is the absence of the $a_+$ term,
that cannot be included because of the degeneracy of $p_+'$.

If \(\mu_-a+\mu a_- = 0\), since $\mu_-,\mu>0$ and $a,a_-\geq 0$, then $\mathcal B_p = 0$.
On the other hand, if \(\mu_-a+\mu a_- \neq 0\), define
\begin{align}\label{HP}
H:=\frac{aa_-}{\mu_-a+\mu a_-},\qquad
P:=\big(\mu a^{p-1}+\mu_-a_-^{p-1}\big)^{1/(p-1)}, \qquad H \leq P,
\end{align}
where the inequality above follows from applying the weighted 
arithmetic-geometric mean inequality ($\mu+\mu_-=1$) to both factors in the product below:
\[
(\mu_-a+\mu a_-)^{p-1} \cdot \big(\mu a^{p-1}+\mu_-a_-^{p-1}\big)
 \geq  (a^{\mu_-} a_-^{\mu})^{p-1} \cdot \big( a^{(p-1)\mu} a_-^{(p-1)\mu_-} \big) 
 = a^{p-1}a_-^{p-1}.
\]
From \eqref{tora29.0} we see that
\begin{align}
\mathcal B_p=(\mu_-a+\mu a_-)(H-a_+)\big(a_+^{p-1}-P^{p-1}\big), 
\end{align}
which is non-positive unless \(H\le a_+\le P\). 
In this remaining case we estimate
\[
\mathcal B_p\le (\mu_-a+\mu a_-)(P-H)\big(P^{p-1}-H^{p-1}\big)\le (\mu_-a+\mu a_-)P^p
\le 2\|\psi\|_{L^\infty}\big(a^p+a_-^p\big),
\]
which proves \eqref{tora35}. 

Note that $K$ is bounded for bounded frequencies, 
and that $-p_-'$ is bounded from below for  $\zeta \in (\gamma_1/2,\gamma_\ast]$ 
due to $p_- = p_+-\mathrm{id}$ and \eqref{pprime}:
\begin{equation}\label{pprime-}
-p_-'(\zeta) 
  =  \frac{\lambda'\big(p_+(\zeta)\big) - \lambda'(\zeta)}{
  \lambda'\big(p_+(\zeta)\big) - \lambda'\big(p_+(\zeta)-\zeta\big)} \gtrsim 1.
\end{equation}
Therefore, using this and \eqref{tora35}, we have
\begin{equation}\label{tora36}
K(\zeta)\mathcal B_p(\zeta) \leq C_{1,\alpha} 
  \|\psi\|_{L^\infty}\big(w a^p+(-p_-'(\zeta))w_-a_-^p\big),\qquad \zeta\in(\gamma_1/2,\gamma_\ast].
\end{equation}

\medskip
{\bf{Case 2.}} 
It remains to consider the exterior region \(\zeta\ge \gamma_\ast\). In this region the degeneracy at $\gamma_1/2$ is absent. 
More precisely, using Lemma \ref{RIV0} we have
\begin{equation}\label{tora37a}
\begin{gathered}
\zeta_+\approx\langle\zeta\rangle,\qquad \zeta_-\lesssim\langle\zeta\rangle^{-1},\qquad
\mu\lesssim1,\qquad \mu_-\lesssim\langle\zeta\rangle^{-2},\\
K(\zeta)\lesssim \langle\zeta\rangle^{1/2},\qquad
p_+'(\zeta)\approx 1,\qquad -p_-'(\zeta)\approx\langle\zeta\rangle^{-2}.
\end{gathered}
\end{equation}
Notice also that if \(p\ge1\), \(r\ge1\), and \(X,Y\ge0\), then,
\begin{equation}\label{tora30}
\begin{split}
(X-r^{-1}Y)\big(rY^{p-1}-X^{p-1}\big) 
&=-(X-Y)(X^{p-1}-Y^{p-1})+(r-1)XY^{p-1}-(1-r^{-1})YX^{p-1}
\\
& \le (r-1)(X^p+Y^p).
\end{split}
\end{equation}

We first estimate the leading order terms given by
\begin{equation}\label{tora31}
\mathcal B_p^{(0)}:=a_-(a-\mu a_+)(w_+a_+^{p-1}-\mu w a^{p-1})
\end{equation}
that are obtained from \eqref{tora29} by disregarding the $\mu_-$ smoothing factors.
Set
\begin{equation}\label{tora32}
r_p(\zeta):=\left(\frac{w_+}{\mu^p w}\right)^{1/p}=\frac{\omega(\zeta_+)^{1-8/p}}{\mu\omega(\zeta)^{1-8/p}}.
\end{equation}
Clearly $r_p(\zeta)\geq 1$ since $\omega$ is nondecreasing and $\mu\leq1$. 
We then apply \eqref{tora30} with $r=r_p(\zeta)$, $X=a$, and $Y=r_p(\zeta)\mu a_+$, as follows:
\begin{align}
\begin{split}\label{tora33}
\mathcal B_p^{(0)} & = \mu w a_- (X-r^{-1}Y)(rY^{p-1} - X^{p-1})
\\
& \le \|\psi\|_{L^\infty}\big(r_p(\zeta)-1\big)\big(\mu w a^p+\mu w_+a_+^p\big). 
\end{split}
\end{align}
For \(\zeta\geq \gamma_\ast\) we have
\(\zeta_+-\zeta=\zeta_-\lesssim\zeta^{-1}\), \(\mu_- \lesssim \zeta^{-2}\), and
\[
\frac{\omega(\zeta_+)}{\omega(\zeta)}-1
\lesssim_\alpha \frac{\zeta_+-\zeta}{\zeta}
\lesssim_\alpha \zeta^{-2}.
\]
Therefore $r_p(\zeta)-1\lesssim_\alpha\zeta^{-2}$. Using also \eqref{tora33} and \eqref{tora37a} gives
\begin{equation}\label{tora38.3}
K(\zeta)\mathcal B_p^{(0)}\le C_{2,\alpha}\|\psi\|_{L^\infty}\big(w a^p+p_+'(\zeta)w_+a_+^p\big),\qquad \zeta\ge\gamma_\ast.
\end{equation}

We estimate now the lower order remainder. From \eqref{tora28} and \eqref{tora31},
\begin{align}
\nonumber
\mathcal B_p-\mathcal B_p^{(0)} 
  & = \mu_-\big[-w_+aa_+^p+\mu w a^pa_++\mu_- w_-a_-^{p-1}aa_+-w_-a_-^pa+\mu w_-a_-^pa_+\big]
  \\ 
  \label{tora37}
  & \leq \mu_-\big[ \mu w a^p a_+ + \mu_- w_-a_-^{p-1}a a_+ + \mu w_-a_-^pa_+\big]
\end{align}
having discarded the negative terms for the last inequality. 
We claim that, provided \(\alpha\ge1\),
\begin{equation}\label{tora38}
K(\zeta)\big(\mathcal B_p-\mathcal B_p^{(0)}\big)\le C_{3,\alpha}\|\psi\|_{L^\infty}\big(w a^p+p_+'(\zeta)w_+a_+^p+(-p_-'(\zeta))w_-a_-^p\big).
\end{equation}
Indeed the first non-negative term in \eqref{tora37} is obviously 
controlled by $C_{\alpha}\|\psi\|_{L^\infty}w a^p$. 
To bound the remaining two terms we estimate $$K(\zeta)\mu_- \big[\mu_- w_-a_-^{p-1}aa_++\mu w_-a_-^pa_+\big]\le \|\psi\|_{L^\infty}K(\zeta)\mu_- w_-a_+a_-^{p-1}.$$ 
Young's inequality then gives
\[
K(\zeta)\mu_- w_-a_+a_-^{p-1}\lesssim_{\alpha} p_+'(\zeta)w_+a_+^p+(-p_-'(\zeta))w_-a_-^p
\]
provided that
\begin{align}\label{ineqK}
K(\zeta)\mu_- w_-\lesssim_{\alpha}(p_+'(\zeta)w_+)^{1/p}(-p_-'(\zeta)w_-)^{(p-1)/p}.
\end{align}
Using \eqref{tora37a} and \eqref{tora26.5} we see that 
the bound \eqref{ineqK} holds uniformly if \(\alpha\ge1\) and $p\geq 50$.

\medskip
{\bf{Step 3.}} We now integrate the inequalities \eqref{tora36} and \eqref{tora38.3}. 
Since \(A\) is resonant-closed,
\(\zeta\in A_>\) implies \(\zeta_-,\zeta_+\in A\). Therefore
\begin{equation}\label{tora40}
\int_{A_>}w a^p\,d\zeta+\int_{A_>}p_+'(\zeta)w_+a_+^p\,d\zeta+\int_{A_>}(-p_-'(\zeta))w_-a_-^p\,d\zeta\le 3E_p(t).
\end{equation}
Using \eqref{tora27}, \eqref{tora36}, \eqref{tora38.3}, and \eqref{tora38} we conclude that if \(p\ge50\) and $\alpha\geq 1$ then
\begin{equation}\label{tora43}
\partial_tE_p(t)\le C_{\alpha}p\|\psi(t)\|_{L^\infty}E_p(t).
\end{equation}
The constant \(C_{\alpha}\) may depend only on $\alpha$. The desired bounds \eqref{tora45} follow.
\end{proof}

\bigskip
\section{Proofs of Theorems \ref{MainThm1} and \ref{MainThm2}}\label{secprth}

In this section we complete the proofs of Theorems \ref{MainThm1} and \ref{MainThm2}. 

\subsection{Proof of Theorem \ref{MainThm1}}
Notice that the equation decouples into two identical evolutions on $(0,\infty)$ and $(-\infty,0)$, so we work on the positive half-line until the last step of the proof. We divide the proof into several steps.  

\textbf{Step 1.} For $\epsilon\in(0,0.1)$, define
\begin{equation}\label{eq:Aeps-Ieps}
A_\epsilon^+:=(0,\infty)\setminus\overline E_\epsilon,
 \qquad I_\epsilon:=(\gamma_0+\epsilon,\gamma_1).
\end{equation}
A standard fixed-point argument, using the Duhamel formula, 
shows that if $\epsilon,\delta\in(0,0.1)$, $R>0$, and $h_0\in B^+_R(Z^\alpha_{A_\epsilon^+})$ then there is $T_{\delta,\epsilon,R}>0$ and a solution $h\in Z^\alpha_{A_\epsilon^+}([0,T_{\delta,\epsilon,R}])$ of the mollified equation
\begin{equation}\label{sec3eq1}
\partial_th(t,\xi)=\mathcal{Q}[h(t),h(t)](\xi)-\delta\langle\xi\rangle^4h(t,\xi),\qquad h(0,\xi)=h_0(\xi).
\end{equation}
Recalling the formulas \eqref{kingeneral2}, in particular the positivity of the terms $\mathcal{Q}_1[h,h]$ and $\mathcal{Q}_2[h,h]$,  it follows easily that $h$ is non-negative on $[0,T_{\delta,\epsilon,R}]\times A_\epsilon^+$. Using the uniform bounds in Proposition \ref{EnergyProp} and letting $p\to\infty$ it follows that there is $T=T_{\alpha,R}\gtrsim_\alpha R^{-1}$ (which does not depend on $\epsilon,\delta$) and a solution $h_{\delta,\epsilon}\in B^+_{C_\alpha R}(Z^\alpha_{A_\epsilon^+}([0,T]))$ of \eqref{sec3eq1}.
\medskip

\textbf{Step 2.} We would like now to let $\delta\to 0$. This step is non-trivial because it still requires the precise structure of the nonlinearity $\mathcal{Q}$. 
Fix $\psi_0\in B_R^+(Z_{A_0}^\alpha)$, and let $h_{\delta,\epsilon}$
be the solution from Step 1 with initial values $\psi_0$ on
$A_\epsilon^+$. Write $L=C_\alpha R$ for their common weighted bound.
Lemma \ref{lemkenE} and \eqref{overE} show that $A_\epsilon^+$ is open
and resonant-closed, and that
\begin{equation}\label{eq:Aeps-orbits}
 A_\epsilon^+=\mathop{\bigcup}_{q\in I_\epsilon}^{\boldsymbol\cdot}E_q,
 \qquad E_q=\{x_k(q),y_k(q):k\geq0\},
\end{equation}
where the coordinates are those of \eqref{eq:orbit-coordinates} and
$\gamma_0=\gamma_1/2$. We use these coordinates only to compare the
mollified solutions; no truncation of the chains is needed.

The bounds \eqref{eq:orbit-low-high-comparison}--\eqref{eq:orbit-growth}
give, uniformly for $q\in I_\epsilon$,
\begin{equation}\label{eq:chain-size}
 x_k(q)\asymp(k+1)^{1/2},\qquad
 y_k(q)\asymp(k+1)^{-1/2}.
\end{equation}
Together with \eqref{pprime} and \eqref{tora37a}, these imply
\begin{equation}\label{eq:chain-coefficients}
 \begin{split}
 &K(x_k(q))\lesssim_\epsilon(k+1)^{1/4},
 \qquad p_+'(x_k(q))\asymp_\epsilon1,\\
 &-p_-'(x_k(q))\asymp_\epsilon(k+1)^{-1},
 \qquad \mu_-(x_k(q))\lesssim_\epsilon(k+1)^{-1}.
 \end{split}
\end{equation}
For a triple $(\zeta,p_-(\zeta),p_+(\zeta))$, write its solution values
as $(u,v,w)$ and use the coefficients $\mu,\mu_-,K$ from \eqref{not1}.
Set
\begin{equation}\label{eq:triple-vector-field}
 F_+(u,v,w):=uv-\mu_-uw-\mu vw,\qquad
 F_0:=-\mu F_+,\qquad F_-:=-\mu_-F_+.
\end{equation}
The calculations \eqref{tora17}--\eqref{tora25} show that this triple
contributes
\begin{equation}\label{eq:triple-contributions}
 \frac{K(\zeta)}{p_+'(\zeta)}F_+,\qquad
 K(\zeta)F_0,\qquad
 \frac{K(\zeta)}{-p_-'(\zeta)}F_-
\end{equation}
to the collision terms at $p_+(\zeta)$, $\zeta$, and $p_-(\zeta)$,
respectively. Thus the equations at level $k$ involve only levels
$k-1,k,k+1$. On a ball with weighted bound $L$, their collision terms
have weighted Lipschitz constant at most
\begin{equation}\label{eq:weighted-chain-Lipschitz}
 C_{\epsilon,\alpha}L(k+1)^{1/4}.
\end{equation}
Here the weights are $\langle x_k\rangle^\alpha$ and
$\langle y_k\rangle^\alpha$. To verify this bound, use
\eqref{eq:chain-coefficients} and the comparability of the weights at
$x_k$ and $x_{k+1}$. In the equation at $y_k$, the factor
$1/(-p_-'(x_k))$ is canceled by the factor $\mu_-(x_k)$ present in
every term of $F_-$.

For $\delta,\delta'\in(0,0.1)$, define
\begin{equation}\label{eq:chain-difference}
 \begin{split}
 D_k(t):=\sup_{q\in I_\epsilon}\big(&
 \langle x_k(q)\rangle^\alpha
 |h_{\delta,\epsilon}(t,x_k(q))-h_{\delta',\epsilon}(t,x_k(q))|\\
 &+\langle y_k(q)\rangle^\alpha
 |h_{\delta,\epsilon}(t,y_k(q))-h_{\delta',\epsilon}(t,y_k(q))|
 \big).
 \end{split}
\end{equation}
Then $D_k(0)=0$ and $D_k(t)\leq4L$. In subtracting the two equations,
write the damping difference at a frequency $\xi$ as
\[
 -\delta\langle\xi\rangle^4
 (h_{\delta,\epsilon}-h_{\delta',\epsilon})
 -(\delta-\delta')\langle\xi\rangle^4h_{\delta',\epsilon}.
\]
The integrating factor for the first term is bounded by one. Since
$\langle x_k\rangle^4+\langle y_k\rangle^4\lesssim(k+1)^2$,
\eqref{eq:weighted-chain-Lipschitz} gives
\begin{equation}\label{eq:mollified-difference}
 D_k(t)\leq CL|\delta-\delta'|(k+1)^2t
 +B(k+1)^{1/4}\int_0^t
 (D_{k-1}+D_k+D_{k+1})(s)\,ds,
\end{equation}
where $D_{-1}=0$ and $B=C_{\epsilon,\alpha}L$.

Fix $m\geq0$. Iterating \eqref{eq:mollified-difference} $n$ times and
using $D_k\leq4L$ yields, for $k\leq m$,
\begin{equation}\label{eq:mollified-iteration}
 \begin{split}
 D_k(t)\leq{}&
 CL|\delta-\delta'|
 \sum_{j=0}^{n-1}\frac{(3B)^jt^{j+1}}{(j+1)!}
 (m+j+1)^2\prod_{\ell=1}^j(m+\ell)^{1/4}\\
 &+4L\frac{(3Bt)^n}{n!}\prod_{\ell=1}^n(m+\ell)^{1/4}.
 \end{split}
\end{equation}
The factor $3^j$ bounds the number of choices of adjacent levels,
and an empty product equals one. The estimate
\begin{equation}\label{eq:factorial-tail-decay}
 \frac1{n!}\prod_{\ell=1}^n(m+\ell)^{1/4}
 \leq C_m\frac{(n+1)^{m/4}}{(n!)^{3/4}}
\end{equation}
shows that the remainder tends to zero and the series converges,
even with the additional polynomial factor $(m+j+1)^2$. Therefore
\begin{equation}\label{eq:mollified-Cauchy}
 \sup_{t\in[0,T]}\sup_{0\leq k\leq m}D_k(t)
 \leq C_{\epsilon,\alpha,L,T,m}|\delta-\delta'|.
\end{equation}
In particular, the full family $h_{\delta,\epsilon}$ has a limit
$\psi_\epsilon^+$ as $\delta\to0$, uniformly on each fixed finite
collection of levels, including uniformly in $q$ and $t$.
\medskip

\textbf{Step 3: continuity and the limiting equation.}
Every frequency in $A_\epsilon^+\cap[M^{-1},M]$ belongs to one of
the first $m(M)$ levels, where $m(M)\lesssim M^2+1$. Indeed,
\eqref{eq:orbit-growth} gives this for $x_k(q)\leq M$, while
$y_k(q)\geq M^{-1}$ and \eqref{eq:orbit-low-high-comparison} imply
$x_k(q)\leq M/2$. Thus \eqref{eq:mollified-Cauchy} gives uniform
convergence on
$[0,T]\times(A_\epsilon^+\cap[M^{-1},M])$ for every $M\geq1$.
Since the mollified solutions are jointly continuous,
$\psi_\epsilon^+$ is jointly continuous on
$[0,T]\times A_\epsilon^+$. It is nonnegative, has initial values
$\psi_0$, and satisfies
\begin{equation}\label{eq:uniform-standard-bound}
 \sup_{(t,\xi)\in[0,T]\times A_\epsilon^+}
 \langle\xi\rangle^\alpha\psi_\epsilon^+(t,\xi)\leq L.
\end{equation}
At each fixed level, the collision terms involve only finitely many
adjacent coordinates and converge uniformly in time. The damping term
also tends to zero there, since it is bounded in the weighted norm by
$C\delta L(k+1)^2$. Passing to the limit in the integral form of
\eqref{sec3eq1}, we obtain
\[
 \psi_\epsilon^+(t,\xi)=\psi_0(\xi)
 +\int_0^t\mathcal Q[\psi_\epsilon^+(s),\psi_\epsilon^+(s)](\xi)\,ds,
 \qquad \xi\in A_\epsilon^+.
\]
The integrand is continuous, so this is a pointwise solution of the
original equation on $A_\epsilon^+$.
\medskip

\textbf{Step 4: the limit $\epsilon\to0$.}
If $0<\epsilon'<\epsilon<0.1$, then
$A_\epsilon^+\subset A_{\epsilon'}^+$ and $A_\epsilon^+$ is
resonant-closed. For each fixed $\delta>0$, the restriction of
$h_{\delta,\epsilon'}$ to $A_\epsilon^+$ therefore solves
\eqref{sec3eq1} with the same initial values as $h_{\delta,\epsilon}$.
Uniqueness in the fixed-point construction of Step 1, applied
successively on short time intervals, implies
\[
 h_{\delta,\epsilon'}|_{[0,T]\times A_\epsilon^+}
 =h_{\delta,\epsilon}.
\]
Taking $\delta\to0$ gives
\begin{equation}\label{eq:epsilon-compatibility}
 \psi_{\epsilon'}^+|_{[0,T]\times A_\epsilon^+}
 =\psi_\epsilon^+.
\end{equation}
Moreover, Lemma \ref{lemkenE} and \eqref{overE} give
\begin{equation}\label{eq:epsilon-exhaustion}
 \bigcup_{0<\epsilon<0.1}A_\epsilon^+
 =(0,\infty)\setminus E_{\gamma_1/2}=A_0\cap(0,\infty).
\end{equation}
The compatible functions $\psi_\epsilon^+$ consequently define a
solution $\psi^+$ on the right-hand side. Each $A_\epsilon^+$ is
open, so $\psi^+$ is jointly continuous there; it is nonnegative
and satisfies the same weighted bound $L$.

Apply the construction also to the reflected initial datum
$\xi\mapsto\psi_0(-\xi)$, and denote the resulting positive-half-line
solution by $\psi^-$. The decoupling and reflection symmetry of the
equation imply that
\[
 \psi(t,\xi):=
 \begin{cases}
 \psi^+(t,\xi),&\xi\in A_0\cap(0,\infty),\\
 \psi^-(t,-\xi),&\xi\in A_0\cap(-\infty,0)
 \end{cases}
\]
is a solution on $A_0$ with initial values $\psi_0$. It belongs to
$B_{C_\alpha R}^+(Z_{A_0}^\alpha([0,T]))$, after enlarging the fixed
constant $C_\alpha$ if necessary.
\medskip

\textbf{Step 5: uniqueness and restricted continuity.}
Let $\psi,\widetilde\psi$ be two solutions on $A_0$ with a common
weighted bound $L<\infty$, and fix $\epsilon\in(0,0.1)$.
Define $d_k(t)$ as in \eqref{eq:chain-difference}, replacing the
mollified solutions by $\psi,\widetilde\psi$, and put
$d_*:=\sup_{k\geq0}d_k(0)$. The weighted Lipschitz estimate
\eqref{eq:weighted-chain-Lipschitz} now gives
\begin{equation}\label{eq:chain-difference-inequality}
 d_k(t)\leq d_*+
 B(k+1)^{1/4}\int_0^t(d_{k-1}+d_k+d_{k+1})(s)\,ds,
 \qquad B=C_{\epsilon,\alpha}L,\quad d_{-1}=0.
\end{equation}
The same iteration as in \eqref{eq:mollified-iteration}, with
$d_k\leq4L$, gives for $k\leq m$
\[
 d_k(t)\leq
 d_*\sum_{j=0}^{n-1}\frac{(3Bt)^j}{j!}
       \prod_{\ell=1}^j(m+\ell)^{1/4}
 +4L\frac{(3Bt)^n}{n!}\prod_{\ell=1}^n(m+\ell)^{1/4}.
\]
Letting $n\to\infty$ and using \eqref{eq:factorial-tail-decay}, we obtain
\begin{equation}\label{eq:finite-level-stability}
 \sup_{t\in[0,T]}\sup_{0\leq k\leq m}d_k(t)
 \leq C_{\epsilon,\alpha,L,T,m}d_*.
\end{equation}
If the initial data agree, then $d_*=0$, so the solutions agree on
$A_\epsilon^+$. The exhaustion \eqref{eq:epsilon-exhaustion} and
reflection prove uniqueness on all of $A_0$.

Finally, for the solutions in the theorem we have $L\leq C_\alpha R$
and
\[
 d_*\leq2\|\widetilde\psi_0-\psi_0\|_{Z_{A_0}^\alpha}.
\]
Apply \eqref{eq:finite-level-stability} with $m=m(M)$ from Step 3,
and use reflection for the negative half-line. This gives
\eqref{lip1}, with a constant depending only on
$\epsilon,M,R,\alpha$, and completes the proof.

\subsection{Proof of Theorem \ref{MainThm2}}

Choose $\eta>0$ sufficiently small that
\[
 p_-(\gamma_1)<\gamma_0-2\eta,\qquad
 \gamma_0+2\eta<\gamma_1-2\eta,\qquad
 \gamma_1+2\eta<p_+(\gamma_1),
\]
and choose cutoffs
$\chi_{\gamma_j}\in C_c^\infty((\gamma_j-2\eta,\gamma_j+2\eta))$,
$j=0,1$, with $0\leq\chi_{\gamma_j}\leq1$ and
$\chi_{\gamma_j}=1$ on $[\gamma_j-\eta,\gamma_j+\eta]$. Define
\begin{equation}\label{eq:illposed-data}
 \psi_0(\xi):=h\chi_{\gamma_0}(\xi)+h\chi_{\gamma_1}(-\xi),
 \qquad h=c_\alpha R,
\end{equation}
and, for $n\geq1$,
\begin{equation}\label{eq:illposed-approximation}
 \psi_0^n(\xi):=\psi_0(\xi)+2^{-n}h\chi_{\gamma_0}(-\xi).
\end{equation}
Here $c_\alpha>0$ is sufficiently small. These functions are
nonnegative, smooth, and compactly supported, and
\[
 \|\psi_0\|_{Z_{A_0}^\alpha}
 +\sup_{n\geq1}\|\psi_0^n\|_{Z_{A_0}^\alpha}\lesssim_\alpha h,
 \qquad
 \|\psi_0^n-\psi_0\|_{Z_{A_0}^\alpha}
 \lesssim_\alpha 2^{-n}h\longrightarrow0.
\]

We first describe the finite-dimensional reduction used for all these
data. The two half-lines evolve independently, and reflection
identifies their equations. On the positive half-line use the orbit
coordinates \eqref{eq:orbit-coordinates}, with
$q\in(\gamma_0,\gamma_1)$. Since
\[
 x_k(q)\geq p_+(\gamma_1)\quad(k\geq2),\qquad
 y_k(q)\leq p_-(\gamma_1)\quad(k\geq1),
\]
only the three coordinates of the first triple can be initially nonzero.
Setting the values at all other coordinates identically zero is consistent with the
equation: every subsequent triple has at most one nonzero entry.
Writing the remaining values as
\[
 u(t,q):=\psi(t,q),\qquad
 v(t,q):=\psi(t,p_-(q)),\qquad
 w(t,q):=\psi(t,p_+(q)),
\]
and setting $e=p_+'(q)$, $d=-p_-'(q)=1-e$, we obtain
\begin{equation}\label{eq:illposed-triple}
 \partial_t w=\frac{K}{e}F_+(u,v,w),\qquad
 \partial_t u=-K\mu F_+(u,v,w),\qquad
 \partial_t v=-\frac{K\mu_-}{d}F_+(u,v,w).
\end{equation}
Here $K,\mu,\mu_-$ are evaluated at $q$, and we use
\eqref{eq:triple-vector-field}--\eqref{eq:triple-contributions}.
Integration gives the identities
\begin{equation}\label{eq:illposed-invariants}
 u=u_0+\mu e(w_0-w),\qquad
 v=v_0+\frac{\mu_-e}{d}(w_0-w).
\end{equation}
The system preserves nonnegativity, as is seen at each coordinate
face. Moreover, $d$ is bounded below on $[\gamma_0,\gamma_1]$.
Thus \eqref{eq:illposed-invariants} bounds $u,v$ by $Ch$, and
$\partial_t w\leq0$ whenever $w\geq\max(u,v)$, so $w\leq Ch$ as
well. These bounds give global solutions of the finite systems.
They define, on each half-line, a solution supported where
$p_-(\gamma_1)\leq|\xi|\leq p_+(\gamma_1)$, jointly continuous
off $\overline E_0$ by continuous dependence and the smooth orbit
coordinates, with weighted norm $\lesssim_\alpha h$ on $[0,T]$.
For sufficiently small $c_\alpha$, uniqueness in
Theorem \ref{MainThm1} identifies these solutions with $\psi,\psi^n$.

For the first assertion, consider $q=\gamma_0+s$ with $s>0$ small.
On the positive half-line the initial values are
\[
 u_0=v_0=h,\qquad w_0=0.
\]
The resonance identity, \eqref{pprime}, and \eqref{riv0.6} give
\[
 p_+'(\gamma_0)=0,\qquad
 p_+''(\gamma_0)=
 \frac{2\lambda''(\gamma_0)}
 {\lambda'(\gamma_1)-\lambda'(\gamma_0)}>0.
\]
Hence $e\asymp s$, $d\asymp1$, $K\asymp1$, and
$\mu,\mu_-\asymp1$. In this case \eqref{eq:illposed-invariants}
becomes
\[
 u=h-\mu e w,\qquad v=h-\frac{\mu_-e}{d}w.
\]
The scalar equation for $w$ preserves $0\leq w\leq h$: at $w=0$
its right-hand side is positive, and at $w=h$ we have $0\leq u,v\leq h$
and $uv\leq h(\mu_-u+\mu v)$. It follows that
$u,v=h+O(hs)$, and therefore
\[
 \partial_t w=\frac{K}{e}
 \big(h(h-w)+O(h^2s)\big).
\]
An integrating factor yields
\begin{equation}\label{eq:illposed-decay}
 0\leq h-w(t,\gamma_0+s)
 \leq h\exp(-cht/s)+Chs.
\end{equation}
Set
\[
 t_s:=h^{-1}\sqrt{s},\qquad
 \xi_s:=p_+(\gamma_0+s).
\]
Then $t_s\to0$, $\xi_s\to\gamma_1^+$, and $\xi_s\in A_0$.
For small $s$ we have $t_s<T$, and \eqref{eq:illposed-decay} gives
\[
 \psi(t_s,\xi_s)\longrightarrow h,\qquad \psi_0(\xi_s)=0.
\]
This proves \eqref{lip2}.

For the second assertion, reflect the negative half-line to the
positive half-line and use the same coordinates. For $q=\gamma_0+s$
with $s>0$ small, the base solution has initial triple $(0,0,h)$,
which is stationary. In particular, $\psi(t,-\xi_s)=h$.
For the solution $\psi^n$, the initial triple is
$(2^{-n}h,2^{-n}h,h)$. Denote its values by $u^n,v^n,w^n$.
The scalar equation again preserves $0\leq w^n\leq h$, since
$F_+(2^{-n}h,2^{-n}h,h)=2^{-n}h^2(2^{-n}-1)\leq0$.
The invariants give
\[
 u^n=2^{-n}h+\mu e(h-w^n),\qquad
 v^n=2^{-n}h+\frac{\mu_-e}{d}(h-w^n).
\]
Consequently,
\[
 \partial_t w^n=\frac{K}{e}
 \big(2^{-n}h(2^{-n}h-w^n)+O(h^2s)\big),
\]
with an error bound uniform in $n$. The integrating-factor estimate is
\[
 |w^n(t,\gamma_0+s)-2^{-n}h|
 \leq h\exp(-c2^{-n}ht/s)+C2^nhs.
\]
For each fixed $n$, evaluating at $t=t_s$ and letting $s\downarrow0$
gives
\[
 |\psi^n(t_s,-\xi_s)-\psi(t_s,-\xi_s)|
 \longrightarrow h(1-2^{-n})\geq h/2.
\]
This proves \eqref{lip3}, with a lower bound independent of $n$.

Finally, continuity near the nonexceptional points $\xi_s$ gives
the lower bounds in the spatial essential supremum needed for
failure of time continuity. Joint continuity near
$(t_s,-\xi_s)$ gives the lower bounds in the space-time essential
supremum needed for failure of continuity of the solution map.
This completes the proof.


\medskip
\section{Gravity waves kinetic equation in $2d$: set-up and trilinear estimates}\label{secgsetup}
We construct $L^2$-based solutions to the wave kinetic equation associated with the 
$2d$ gravity water waves system.

\medskip
\subsection{The model and linearization}\label{ssecgmodel}
Consider the gravity water waves kinetic model \cite{Zak0}
\begin{align}\label{2ddtf}
\partial_t f(t,{\bf k}) = \mathcal{Q}(f)({\bf k}), \qquad f \geq 0
\end{align}
where
\begin{align}\label{2dQ}
\begin{split}
\mathcal{Q}(f)({\bf k}) := \int_{\R^6} \mathcal{M}({\bf k},{\bf k_1},{\bf k_2},{\bf k_3}) 
  \, \delta\big( \Phi({\bf k},{\bf k_1},{\bf k_2},{\bf k_3})\big)
 \delta\big(\Sigma({\bf k},{\bf k_1},{\bf k_2},{\bf k_3})\big) 
 \\
 \times \big[ f_1f_2f_3 + ff_2f_3 - ff_1f_3 - ff_1f_2  \big] \, d{\bf k_1}d{\bf k_2}d\bf{k_3},
\end{split}
\end{align}
with 
\begin{align}
\Phi({\bf k},{\bf k_1},{\bf k_2},{\bf k_3}) := \omega + \omega_1 - \omega_2 - \omega_3, 
  \qquad \Sigma({\bf k},{\bf k_1},{\bf k_2},{\bf k_3}) := 
  {\bf k} + {\bf k_1} - {\bf k_2} - {\bf k_3},
\end{align}
$\omega = \omega({\bf k}) = \sqrt{|\bf k|}$ and $\omega_j= \omega({\bf k_j})$, 
similarly for $f,f_j$. The symbol $\mathcal{M}$ is explicitly given in \cite{PanWu}
(see Proposition 2.1 there);
here, its exact form is not relevant, and we will only use that it satisfies the following properties:

\setlength{\leftmargini}{1.5em}
\begin{itemize}
 \item[(i)] The natural symmetry
\begin{align}\label{2dMsymm}
\mathcal{M}({\bf k},{\bf k_1},{\bf k_2},{\bf k_3}) = 
  \mathcal{M}({\bf k_1},{\bf k},{\bf k_2},{\bf k_3}) 
  = \mathcal{M}({\bf k},{\bf k_1},{\bf k_3},{\bf k_2}) 
  = \mathcal{M}({\bf k_2},{\bf k_3},{\bf k},{\bf k_1}), 
\end{align}

\smallskip
\item[(ii)] The homogeneity and rotation invariance
\begin{align}\label{2dMhomo}
\begin{split}
\mathcal{M}(\lambda ({\bf k},{\bf k_1},{\bf k_2},{\bf k_3} )) & = 
  |\lambda|^6 \mathcal{M}({\bf k},{\bf k_1},{\bf k_2},{\bf k_3}), 
\\
\mathcal{M}(e^{i\theta} ({\bf k},{\bf k_1},{\bf k_2},{\bf k_3} )) & = 
  \mathcal{M}({\bf k},{\bf k_1},{\bf k_2},{\bf k_3}), 
\end{split}
\end{align}

\smallskip
\item[(iii)] 
The bound
\begin{align}\label{2dMbound}
\begin{split}
0 \leq \mathcal{M}({\bf k},{\bf k_1},{\bf k_2},{\bf k_3}) & \lesssim 
  \big(|{\bf k}|^2 + |{\bf k_1}|^2 + |{\bf k_2}|^2 + |{\bf k_3}|^2\big) 
  \frac{\big(|{\bf k}| |{\bf k_1}| |{\bf k_2}| |{\bf k_3}|\big)^2}{(|{\bf k}|^2+|{\bf k_1}|^2+|{\bf k_2}|^2+|{\bf k_3}|^2)^2}
  \\
  & \approx \max_{i=0,\dots,3}(|{\bf k_i}|^2) 
  \cdot \min_{0\leq j<\ell\leq 3}\big(|{\bf k_j}|^2 + |{\bf k_\ell}|^2 \big)^2,
\end{split}
\end{align}
where the above inequality and the approximate equality hold on the support of the collision kernels,
since the magnitudes of the two largest and of the two smallest frequencies 
are pairwise comparable at a resonance.
\end{itemize}

The bound \eqref{2dMbound} is proven in \cite{PanWu}.
For convenience we will sometimes use ${\bf k_0}$ for ${\bf k}$. Let $0\leq \chi \leq 1$ be smooth compactly supported
in $|x|\leq 1/20$ and equal to $1$ for $|x|\leq 1/40$, define 
\begin{align}\label{cutoff}
\begin{split}
& \chi_{j < \ell} := \chi(|{\bf k_j}|/|{\bf k_\ell}|), \qquad \chi_{j \geq \ell} = 1 - \chi_{j < \ell},
\\
& \widetilde{\chi}_{j,\ell} := 1 - \chi(|{\bf k_j}|/|{\bf k_\ell}|) - \chi(|{\bf k_\ell}|/|{\bf k_j}|)  
  \quad (j\neq \ell),
\end{split}
\end{align}
Similarly to \cite{PanWu}, but using smooth cutoffs for the purpose of differentiation,
we introduce the following polarization:
\begin{align}\label{2dQgh}
\begin{split}
\mathcal{Q}_{a, b} [h] & := \int_{\R^6} \mathcal{M}({\bf k},{\bf k_1},{\bf k_2},{\bf k_3}) 
  \, \delta(\Phi) \delta(\Sigma) 
  \\ & \times \Big[ 2 \chi_{2<3} \, a_1 b_2 \, (h_3 - h) 
  + a_2 b_3 h - 2\chi_{2 \geq 3} a_1 b_2 h + \widetilde{\chi}_{2,3} \, a_1b_2h_3
  \Big]\, d{\bf k_1}d{\bf k_2}d{\bf k_3}.
\end{split}
\end{align}
Note that, according to \eqref{2dQgh} and \eqref{2dQ} we have 
\begin{align}
\mathcal{Q}_{f,f}[f]=\mathcal{Q}(f).
\end{align}
The last term in the linearization in the integral \eqref{2dQgh} 
can be chosen in different ways as well 
since we have $|{\bf k_2}| \approx |{\bf k_3}|$ on the support of $\widetilde{\chi}_{2,3}$
and, in view of $\Phi=0$ we must have $\max(|{\bf k_1}|,|{\bf k}|) \approx \max(|{\bf k_2}|,|{\bf k_3}|)$. 

We then decompose \eqref{2dQgh} into ``dissipative'' and ``bounded'' parts by writing
\begin{align}\label{2dLBsplit}
\begin{split}
\mathcal{Q}_{a,b}[h] = \mathcal{L}_{a,b}[h] + \mathcal{B}_{a,b}[h],
\end{split}
\end{align}
where the dissipative part is 
\begin{align}\label{2dLg}
\mathcal{L}_{a,b} [h] = 2\int_{\R^6} \mathcal{M}({\bf k},{\bf k_1},{\bf k_2},{\bf k_3}) \, \delta(\Phi)
 \delta(\Sigma) \, \chi_{2<3} \chi_{1<0} \, a_1 b_2 \, (h_3 - h) \, d{\bf k_1}d{\bf k_2}d{\bf k_3},
\end{align}
and the bounded part is
\begin{align}\label{2dBg} 
\begin{split}
\mathcal{B}_{a, b} [h] & := \int_{\R^6} 
  \mathcal{M}({\bf k},{\bf k_1},{\bf k_2},{\bf k_3}) \, \delta(\Phi)
 \delta(\Sigma) 
 \\ & \times \big[ 2 \chi_{2<3} \chi_{1\geq 0} \, a_1 b_2 \, (h_3 - h) 
  + a_2 b_3 h - 2\chi_{2\geq 3} a_1 b_2 h + \widetilde{\chi}_{2,3} \, a_1b_2h_3 \big]
  \, d{\bf k_1}d{\bf k_2}d{\bf k_3}.
\end{split}
\end{align}
Note that on the support of \eqref{2dLg} there are two high frequencies ${\bf k}$ and ${\bf k_3}$
that are comparable and the other two are smaller, so that all the `losses of derivatives',
(that is, of weights in ${\bf k}$) are confined there,
while on the support of \eqref{2dBg} there is no `loss of derivatives' on the input function $h$.
Moreover, we observe that the dissipative part loses two `derivatives',
which motivates in part the choice of our main norm.

When the two low-frequency inputs are the same we use the short-hand $\mathcal{T}_g[h] := \mathcal{T}_{g,g}[h]$, $\mathcal{T} \in \{ \mathcal{Q}, \mathcal{L}, \mathcal{B} \}$. 
We will also use the short-hand
\begin{align}
\mathcal{M}_\mathcal{L}({\bf k},{\bf k_1},{\bf k_2},{\bf k_3}) 
 := \mathcal{M}({\bf k},{\bf k_1},{\bf k_2},{\bf k_3}) \, \chi_{2<3} \chi_{1<0} .
\end{align}

\subsection{The dissipative part and functional framework}
We first observe that for any $g \geq 0$, using symmetry,
\begin{align}\label{Lguu}
\begin{split}
\frac{1}{2} \big( \mathcal{L}_g u, u \big)
 & = \int_{\R^2} \int_{\R^6} \mathcal{M}_\mathcal{L}({\bf k},{\bf k_1},{\bf k_2},{\bf k_3}) \, \delta(\Phi)
 \delta(\Sigma)\, g_1g_2 \, (u({\bf k_3}) - u({\bf k})) \, d{\bf k_1}d{\bf k_2}d{\bf k_3} 
  \, u({\bf k}) \, d{\bf k}
 \\
 & = - \frac{1}{2} \int_{\R^8} \mathcal{M}_\mathcal{L}({\bf k},{\bf k_1},{\bf k_2},{\bf k_3}) \, \delta(\Phi)
 \delta(\Sigma) \, g_1g_2 \, \big| u({\bf k_3}) - u({\bf k}) \big|^2 
 \, d{\bf k_1}d{\bf k_2}d{\bf k_3} \, d{\bf k}
 \leq 0.
\end{split}
\end{align}
By the same symmetry we have an ``energy-type estimate'' with weights $\jk^s$:
\begin{align}\label{Lguuweights}
\begin{split}
& \frac{1}{2} \big(  \langle {\bf k} \rangle^s \mathcal{L}_g u, \langle {\bf k} \rangle^s u \big) 
 \\
 & = \int_{\R^2} \int_{\R^6} \mathcal{M}_\mathcal{L}({\bf k},{\bf k_1},{\bf k_2},{\bf k_3}) \, \delta(\Phi)
 \delta(\Sigma) \, g_1g_2 \, (\langle {\bf k_3}\rangle^s u({\bf k_3}) 
  - \jk^s u({\bf k})) \, d{\bf k_1}d{\bf k_2}d{\bf k_3} \, \jk^s u({\bf k}) \, d{\bf k}
 \\
 & + \int_{\R^2} \int_{\R^6} \mathcal{M}_\mathcal{L}({\bf k},{\bf k_1},{\bf k_2},{\bf k_3}) \, \delta(\Phi)
 \delta(\Sigma) \, g_1 g_2
 \, d{\bf k_1}d{\bf k_2} \, \big(\jk^s\langle {\bf k_3}\rangle^{-s} - 1\big) 
 \,  \langle {\bf k_3}\rangle^s u({\bf k_3}) \jk^s u({\bf k}) \, d{\bf k_3}d{\bf k}
 \\
 & \leq \frac{1}{2}
 \int_{\R^2} \int_{\R^6} \mathcal{M}_\mathcal{L}({\bf k},{\bf k_1},{\bf k_2},{\bf k_3}) \, \delta(\Phi)
 \delta(\Sigma) \, g_1 g_2
 \, d{\bf k_1}d{\bf k_2} \, \big(\jk^s\langle {\bf k_3}\rangle^{-s} 
  + \langle {\bf k_3}\rangle^s\jk^{-s} - 2\big) 
 \, \\
 & \qquad \qquad \qquad \qquad \qquad \qquad \qquad \qquad 
 \times \langle {\bf k_3}\rangle^s u({\bf k_3}) \jk^s u({\bf k}) \, d{\bf k_3}d{\bf k};
\end{split}
\end{align}
then we note that,
on the support of the integral, since $|{\bf k}|\approx|{\bf k_3}| 
  \gtrsim |{\bf k_1}|\approx|{\bf k_2}|$
the symbol 
\begin{align}\label{2dMboundmain}
\begin{split}
\big| \mathcal{M}_\mathcal{L}({\bf k},{\bf k_1},{\bf k_2},{\bf k_3})
  \big(\jk^s\langle {\bf k_3}\rangle^{-s} + \langle {\bf k_3}\rangle^s\jk^{-s} - 2\big) \big|
  \\
  \lesssim_s 
  |{\bf k_1}|^2 |{\bf k_2}|^2 \big(|{\bf k_1}|^2 + |{\bf k_2}|^2\big)
\end{split}
\end{align}
does not grow in the high frequencies, 
and vanishes (more than) quadratically at each of the two low frequencies.
One cannot however close a priori energy estimates 
using solely the above observation since the low-frequency inputs need to 
be estimated in $L^\infty_k$ (see Lemma \ref{lemintPW} below, for example);
it is then convenient to introduce a functional framework with derivatives in ${\bf k}$,
based on the space $Z_0$ defined in \eqref{defZ0}.

For the purpose of introducing derivatives in ${\bf k}$ we define the 
`vector-fields' operators
\begin{align}\label{2dZfields}
\begin{split}
\Omega & := \Omega_{\bf k} := k_1\partial_{k_2} - k_2\partial_{k_1} = {\bf k} \wedge \nabla_{\bf k},
  \\ 
S & := S_k :=
  \frac{1}{2}\big({\bf k} \cdot \nabla_{\bf k} + \nabla_{\bf k} \cdot {\bf k}\big) 
  = {\bf k} \cdot \nabla_{\bf k} + 1.
\end{split}
\end{align}
Note that $\Omega$ and $S$ commute, they are anti-self-adjoint, 
so that $\Omega^2$ and $S^2$ are self-adjoint, and
\begin{align}
\widehat{\Omega_{\bf k} f} = \Omega_x \widehat{f},
   \qquad \mathcal{F}_{{\bf k}\mapsto x} \big( S_{\bf k} f \big) = -S_x \mathcal{F}_{{\bf k}\mapsto x}f.
\end{align} 
We will use the commutator notation
\begin{align}
[X, \mathcal{Q}_{a,b}] f = X \big( \mathcal{Q}_{a,b}[f] \big) - \mathcal{Q}_{a,b}[X f].
\end{align}

We fix $N_1 \geq 30$ 
and define our main $Z$-norm as (here $\ell \in \Z_+^2$)
\begin{align}\label{2dZnorm}
{\| f \|}_{Z} := {\| f \|}_{L^{2,N_1}} + \sum_{1\leq |\ell| \leq 4}
  {\big\| Z_{\bf k}^\ell f \big\|}_{L^{2,N_1-2|\ell|}}, \qquad 
  Z_{\bf k}^\ell := S_{\bf k}^{\ell_1} \Omega_{\bf k}^{\ell_2}.
\end{align}
The $Z$-norm controls schematically
\begin{align}
f \in L^{2,N_1}, \quad Zf \in L^{2,N_1-2}, \quad Z^2f \in L^{2,N_1-4},  \quad  Z^3f \in L^{2,N_1-6}, 
  \quad  Z^4f \in L^{2,N_1-8}.
\end{align}
This choice is dictated by the loss of two derivatives in the dissipative operator,
which can only be recovered by the energy argument for the highest number of 
vector-fields, but needs to be allowed for lower order commutators.
The $Z$-norm and the above graded structure 
will work efficiently with the Sobolev type embedding from Lemma \ref{lemSob}.
Note also that $Z$ can be endowed with an equivalent norm 
under which it is a Hilbert space with a natural $L^2$ type inner product.

\medskip
We conclude this subsection with a helpful lemma that will be repeatedly used in the
proof of energy estimates below, that is,
a Sobolev-type embedding which will allow control of the densities in $L^\infty_{\bf k}$:

\begin{lemma}\label{lemSob}
With the above definitions, we have
\begin{align}\label{lemSobconcelliptic}
{\big\| S_{\bf k}^2 f \big\|}_{L^{2,s}} + {\big\| \Omega_{\bf k}^2 f \big\|}_{L^{2,s}} 
  + {\| f \|}_{L^{2,s}} \approx 
  {\big\| |{\bf k}|^2 \Delta_{\bf k} f \big\|}_{L^{2,s}} + {\| f \|}_{L^{2,s}},
\end{align}
and
\begin{align}\label{lemSobconc}
{\big\| |{\bf k}|^2 f \big\|}_{L^{\infty,s}} \lesssim 
  {\big\| S_{\bf k}^2 f \big\|}_{L^{2,s}} + {\big\| \Omega_{\bf k}^2 f \big\|}_{L^{2,s}}
  + {\big\| f \big\|}_{L^{2,s+2}} .
\end{align}
In particular, 
for all $s\leq N_1-2|\ell|-4$, and for $|\ell| = 0,1,2$ we have
\begin{align}\label{lemSobconcZ}
{\big\| |{\bf k}|^2 Z^\ell f \big\|}_{L^{\infty,s}} 
  \lesssim \sum_{|n|\leq 2} {\big\| Z^{n+\ell} f \big\|}_{L^{2,s}}
  + {\big\| Z^\ell f \big\|}_{L^{2,s+2}} \lesssim {\| f \|}_Z.
\end{align}
As a consequence of \eqref{lemSobconcelliptic} we also have
\begin{align}\label{lemSobequiv}
{\| f \|}_Z \approx  {\| f \|}_{Z_0} 
\end{align}
where $Z_0$ is the space defined in \eqref{defZ0}.
\end{lemma}

\begin{proof}
The proof follows from elliptic estimates and Sobolev's embedding.
More precisely, we have 
\begin{align}\label{lemSobpr0}
S_{\bf k}^2 + \Omega_{\bf k}^2 = |{\bf k}|^2 \Delta_{{\bf k}} + 2S_{\bf k}-1.
\end{align}
Then, using also integration by parts we can see that
\begin{align}
{\big\| S_{\bf k} f \big\|}_{L^{2,s}}^2 \lesssim {\big\| S_{\bf k}^2 f \big\|}_{L^{2,s}}
  {\big\| f \big\|}_{L^{2,s}} + {\big\| f \big\|}_{L^{2,s}}^2.
\end{align}
Using also the fact that the commutators of $S_{\bf k}$ and $\Omega_{\bf k}$ with the weights 
are harmless, $[S_{\bf k},\jk^s] = s |{\bf k}|^2\jk^{s-2}$, $[\Omega_{\bf k},\jk]=0$,
to absorb the factor of $S_{\bf k}$ on the right-hand side of \eqref{lemSobpr0}, 
we then obtain \eqref{lemSobconcelliptic}.
The inequality \eqref{lemSobconc} then follows from the embedding $H^2(\R^2) \subset L^\infty(\R^2)$.

For \eqref{lemSobequiv} it suffices to observe that in polar coordinates 
$S=r\partial_r+1$ and $\Omega = \partial_\theta$ and, therefore,
powers of $r\nabla$ are bounded by combinations of $S$ and $\Omega$ and lower order terms,
and vice versa.
\end{proof}

\smallskip
\subsection{Trilinear bounds}
In this section we gather all the necessary trilinear bounds for our proof.
The following lemma is a refined version of Lemma 3.4 from \cite{PanWu}
(which is the statement of the inequalities \eqref{2dintPWconc}).
The refinement \eqref{2dintrefine} is actually needed to close our estimates;
in particular, it will allow us to estimate the commutators of vector-fields and the dissipative
collision integrals in the $Z$ norm; see Lemma \ref{lemintL} below and its applications
in Section \ref{secgLWP}.

\begin{lemma}\label{lemintPW} 
Consider general collision integrals in dimension $d=2$, of the form
\begin{align}\label{2dcollint}
\begin{split}
\mathcal{T}[F]({\bf k}) := \int_{\R^{3d}} 
  \, F({\bf k},{\bf k_1},{\bf k_2},{\bf k_3}) \, 
  \delta\big( \Phi({\bf k},{\bf k_1},{\bf k_2},{\bf k_3})\big)
  \delta\big(\Sigma({\bf k},{\bf k_1},{\bf k_2},{\bf k_3})\big) 
  \, d{\bf k_1}d{\bf k_2}d\bf{k_3}.
\end{split}
\end{align}
Then the following bounds hold true:
\begin{subequations}\label{2dintPWconc}
\begin{align}
\label{2dintPWconc1}
\big| \mathcal{T}[F]({\bf k}) \big| & \lesssim \int_{\R^{d}} \frac{1}{|{\bf k} - {\bf k_3}|}
  {\big\| \langle {\bf k_1} \rangle^{d/2+1} \langle {\bf k_2} \rangle^{d/2+1} 
   F({\bf k},{\bf k_1},{\bf k_2},{\bf k_3}) \big\|}_{L^\infty_{{\bf k_1}{\bf k_2}}(\R^{2d})} d \bf{k_3},
\\
\label{2dintPWconc2}
\big| \mathcal{T}[F]({\bf k}) \big| & \lesssim \int_{\R^{d}} \frac{1}{|{\bf k} - {\bf k_2}|}
  {\big\| \langle {\bf k_1} \rangle^{d/2+1} \langle {\bf k_3} \rangle^{d/2+1} 
   F({\bf k},{\bf k_1},{\bf k_2},{\bf k_3}) \big\|}_{L^\infty_{{\bf k_1}{\bf k_3}}(\R^{2d})} d \bf{k_2},
   \\
\label{2dintPWconc3}
\big| \mathcal{T}[F]({\bf k}) \big| & \lesssim \int_{\R^{d}} \frac{1}{|{\bf k} + {\bf k_1}|}
  {\big\| \langle {\bf k_2} \rangle^{d/2+1} \langle {\bf k_3} \rangle^{d/2+1} 
   F({\bf k},{\bf k_1},{\bf k_2},{\bf k_3}) \big\|}_{L^\infty_{{\bf k_2}{\bf k_3}}(\R^{2d})} d \bf{k_1}.
\end{align}
\end{subequations}
The convention is ${\bf k_1} = {\bf k_2} - {\bf k} + {\bf k_3}$ in \eqref{2dintPWconc1} and \eqref{2dintPWconc2}, and ${\bf k_3} = {\bf k} + {\bf k_1} - {\bf k_2}$ in \eqref{2dintPWconc3}.

Moreover, for $j_1,j \in \Z$ with $j_1 \leq j-5$ 
let us define a high-low localized version of \eqref{2dcollint} by
\begin{align}\label{2dcollintloc}
\begin{split}
\mathcal{T}_{j_1,j}[F](k) & := \int_{\R^{3d}} \varphi_{j_1}({\bf k_1})\varphi_{j}({\bf k_3})
  \varphi_{\sim j_1}({\bf k_2}) \varphi_{\sim j}({\bf k})
  \\
  & \times F({\bf k},{\bf k_1},{\bf k_2},{\bf k_3}) \, 
  \delta\big( \Phi({\bf k},{\bf k_1},{\bf k_2},{\bf k_3})\big)
  \delta\big(\Sigma({\bf k},{\bf k_1},{\bf k_2},{\bf k_3})\big) 
  \, d{\bf k_1}d{\bf k_2}d\bf{k_3},
\end{split}
\end{align}
where $\varphi_j$ is a standard  Littlewood-Paley cutoff and 
$\varphi_{\sim j}$ denotes a similar slightly enlarged cutoff.
Then
\begin{align}
\label{2dintrefine}
\big| \mathcal{T}_{j_1,j}[F]({\bf k}) \big| & \lesssim 2^{(3/2)j_1} \varphi_{\sim j}({\bf k}) 
  \\
  & \qquad \times \int_{\R^{d}} \varphi_{j_1}({\bf k_1}) 
  {\big\| \varphi_{\sim j_1}({\bf k_2}) \varphi_{j}({\bf k_3})
  F({\bf k},{\bf k_1},{\bf k_2},{\bf k_3}) \big\|}_{L^\infty_{{\bf k_2}{\bf k_3}}(\R^{2d})} d \bf{k_1}.
\end{align} 
(It is understood that ${\bf k_3} = {\bf k} + {\bf k_1} - {\bf k_2}$ in the above.)
The same estimate holds by interchanging the roles of ${\bf k_1}$ and ${\bf k_2}$.
\end{lemma}

The refinement \eqref{2dintrefine} is needed in the proof of \eqref{lemintLconc2}
to avoid losing extra weights on the high momentum input $h$; 
the estimate \eqref{lemintLconc2} in turn is used in the course of the energy estimates,
see \eqref{2dpropenpr22}, where it allows us to close the argument in the space $Z$.

\begin{proof}[Proof of Lemma \ref{lemintPW}]
The estimates \eqref{2dintPWconc} are proved in \cite{PanWu}, so we only give details 
for the refinement \eqref{2dintrefine}.
We integrate the $\delta$ in ${\bf k_3}$ and write the operator \eqref{2dcollintloc} as
\begin{align}
\begin{split}
\mathcal{T}_{j_1,j}[F](k) & = \int_{\R^d} 
  \Big( \int_{\R^d} \varphi_{j_1}({\bf k_1})
  \varphi_{j}({\bf k} + {\bf k_1} - {\bf k_2})
  \varphi_{\sim j_1}({\bf k_2}) \varphi_{\sim j}({\bf k})
  \\
  & \times F({\bf k},{\bf k_1},{\bf k_2},{\bf k} + {\bf k_1} - {\bf k_2}) \, 
  \delta\big( \Phi({\bf k},{\bf k_1},{\bf k_2},{\bf k} + {\bf k_1} - {\bf k_2}) \big) 
  \,d{\bf k_2} \Big) \, d{\bf k_1}
  \\
  & =
  \lim_{\epsilon \rightarrow 0^+}
  \int_{\R^d} \Big( \int_{\R^d} \varphi_{j_1}({\bf k_1})
  \varphi_{j}({\bf k} + {\bf k_1} - {\bf k_2})
  \varphi_{\sim j_1}({\bf k_2}) \varphi_{\sim j}({\bf k})
  \\
  & \times F({\bf k},{\bf k_1},{\bf k_2},{\bf k} + {\bf k_1} - {\bf k_2}) \, 
  \epsilon^{-1} \varphi\big( \epsilon^{-1} \Phi({\bf k},{\bf k_1},{\bf k_2},{\bf k} + {\bf k_1} - {\bf k_2}) \big) 
  \,d{\bf k_2} \Big) \, d{\bf k_1},
\end{split}
\end{align}
where $\varphi$ is a non-negative bump function with integral $1$.
To integrate the regularized delta over ${\bf k_2}$ we first calculate that, 
on the support of the integral,
\begin{align}
\big| \nabla_{|{\bf k_2}|} \Phi({\bf k},{\bf k_1},{\bf k_2},{\bf k} + {\bf k_1} - {\bf k_2}) \big|
 = \frac{1}{2} \Big| \frac{1}{\sqrt{|\bf k_2}|} 
 - \frac{{\bf k} + {\bf k_1} - {\bf k_2}}{|{\bf k} + {\bf k_1} - {\bf k_2}|^{3/2}}
 \cdot \frac{{\bf k_2}}{|\bf k_2|} \Big|
 \approx 2^{-j_1/2}.
\end{align}
In polar coordinates ${\bf k_2} = r \theta$, this gives that the support of the set 
$\{ r \, : \, |\Phi| \lesssim \epsilon \}$ for each fixed angle is contained in an interval of size 
about $\epsilon 2^{j_1/2}$.
Then, the $2$-dimensional $d{\bf k_2}$ integral contributes $\epsilon 2^{3j_1/2}$
and, accounting for the $\epsilon^{-1}$ factor from the regularization,
it follows that
\begin{align}
\big| \mathcal{T}_{j_1,j}[F](k) \big| & \lesssim 
  \int_{\R^d}  \varphi_{j_1}({\bf k_1}) \cdot 2^{3j_1/2}
  \varphi_{\sim j}({\bf k})
  \\
  & \times \sup_{{\bf k_2}} 
  \Big| \varphi_{\sim j_1}({\bf k_2}) \varphi_{j}({\bf k} + {\bf k_1} - {\bf k_2}) 
  F({\bf k},{\bf k_1},{\bf k_2},{\bf k} + {\bf k_1} - {\bf k_2}) \Big| \, d{\bf k_1},
\end{align}
which is the desired \eqref{2dintrefine}.
\end{proof}

\smallskip
The next two statements are corollaries of Lemma \ref{lemintPW}
which gather all the trilinear bounds on $\mathcal{L}$ and $\mathcal{B}$ 
that we will use in our energy estimates.

\begin{lemma}[Trilinear bounds for the dissipative part]\label{lemintL}
Consider the operator from \eqref{2dLg}, that is,
\begin{align}\label{2dLg'}
\mathcal{L}_{a,b} [h] = 2\int_{\R^6} \mathcal{M}({\bf k},{\bf k_1},{\bf k_2},{\bf k_3}) \, \delta(\Phi)
 \delta(\Sigma) \, \chi_{2<3} \chi_{1<0} \, a_1 b_2 \, (h_3 - h) \, d{\bf k_1}d{\bf k_2}d{\bf k_3}.
\end{align}
Then, the following hold:

\setlength{\leftmargini}{1.5em}
\begin{itemize}

\smallskip
\item[(i)]  For all $s \in \R$ we have the following estimates
that allow losses of weights:
\begin{subequations}\label{lemintLconc}
\begin{align}\label{lemintLconc1}
\begin{split}
{\big\| \mathcal{L}_{a,b} [h]  \big\|}_{L^{2,s}}
  & \lesssim  {\big\| |{\bf k}|^2a \big\|}_{L^{\infty,3}} 
  \cdot {\big\| |{\bf k}|^2b \big\|}_{L^{\infty,3}} \cdot {\big\| |{\bf k}|^2h \big\|}_{L^{2,s}},
\end{split}
\end{align}
and
\begin{align}\label{lemintLconc2}
\begin{split}
{\big\| \mathcal{L}_{a,b}[h] \big\|}_{L^{2,s}}
  & \lesssim {\big\| |{\bf k}|^2a \big\|}_{L^{2,4}} \cdot {\big\| |{\bf k}|^2 b \big\|}_{L^{\infty,2}} 
  \cdot  {\big\| |{\bf k}|^2 h \big\|}_{L^{\infty,s+2}},
\end{split}
\end{align}
and
\begin{align}\label{lemintLconc3}
\begin{split}
{\big\| \mathcal{L}_{a,b}[h] \big\|}_{L^{2,s}}
  & \lesssim  {\big\| |{\bf k}|^2a \big\|}_{L^{\infty,2}} \cdot {\| |{\bf k}|^2b \|}_{L^{2,4}}
  \cdot  {\big\| |{\bf k}|^2 h \big\|}_{L^{\infty,s+2}}.
\end{split}
\end{align}
\end{subequations}

\smallskip
\item[(ii)] We have the `dissipative' type estimates, for $g\geq 0$, 
\begin{align}\label{lemdippartconcLuu}
\big( \mathcal{L}_g[\langle {\bf k} \rangle^s h] , \langle {\bf k} \rangle^s h \big) & \leq 0,\qquad \Big| \big(\big[\langle {\bf k} \rangle^s, \mathcal{L}_g\big] h, \langle {\bf k} \rangle^s h \big) \Big|\lesssim {\big\| |{\bf k}|^2 g \big\|}_{L^{\infty,5}}^2 {\| h \|}_{L^{2,s}}^2.
\end{align}
In particular
\begin{align}\label{lemdippartconcLuu'}
\begin{split}
\big( \langle {\bf k} \rangle^s \mathcal{L}_g h, \langle {\bf k} \rangle^s h \big)
  & \lesssim {\big\| |{\bf k}|^2 g \big\|}_{L^{\infty,5}}^2 {\| h \|}_{L^{2,s}}^2. 
\end{split}
\end{align}
 
\end{itemize}
\end{lemma}

Note that in the estimates \eqref{lemintLconc1}-\eqref{lemintLconc3} of Lemma \ref{lemintL} 
we are allowing the loss of two weights on the high frequency input $h$. 
The choice of the (vector-fields-weighted) part of the $Z$-norm is consistent with this.
Part (ii) of Lemma \ref{lemintL} handles the loss by exploiting the dissipation,
when all the weights and vector-fields fall on the high frequency input.

\smallskip
\begin{proof}[Proof of Lemma \ref{lemintL}]

{\it Proof of (i)}
It suffices to prove \eqref{lemintLconc1} and \eqref{lemintLconc2}
since the estimate \eqref{lemintLconc3} can be obtained in an analogous fashion 
by exchanging the roles of the low frequency inputs.
On the support of \eqref{2dLg'}, we have $|{\bf k_3}| \approx |{\bf k}|
\gtrsim |{\bf k_1}| \approx |{\bf k_2}|$, so the bound on the multiplier \eqref{2dMbound} gives
\begin{align}\label{intLprM}
\mathcal{M}({\bf k},{\bf k_1},{\bf k_2},{\bf k_3}) 
  & \lesssim 
  |{\bf k_1}|^2 |{\bf k_2}|^2 \cdot \max\big( |{\bf k}|^2, |{\bf k_3}|^2 \big).
\end{align}

\smallskip
\noindent
{\it Proof of \eqref{lemintLconc1}}
Using the symbol bound \eqref{intLprM}, the pointwise estimate \eqref{2dintPWconc1},
and the fact that $|{\bf k}-{\bf k_3}|\leq|{\bf k_1}|+|{\bf k_2}|$ we have
\begin{equation}
\begin{split}
& \big\| \mathcal{L}_{a,b} [h]  \big\|_{L^{2,s}}\lesssim \Big\| \int_{\R^2} \frac{\langle{\bf k}-{\bf k_3}\rangle^{-2}}{|{\bf k}-{\bf k_3}|}
    \Big\| \langle {\bf k_1}\rangle^3  |{\bf k_1}|^2 a_1 \cdot 
    \langle {\bf k_2}\rangle^3  |{\bf k_2}|^2 b_2\\
    &\times \big[ \langle {\bf k} \rangle^s|{\bf k}|^{2} |h({\bf k})| 
    + \langle {\bf k_3} \rangle^s|{\bf k_3}|^{2} |h({\bf k_3})| \big] 
    \Big\|_{L^\infty_{{\bf k_1},{\bf k_2}}}
    \, d {\bf k_3}
    \Big\|_{L^2_{{\bf k}}}\lesssim  {\big\| |{\bf k}|^2a \big\|}_{L^{\infty,3}} 
  \cdot {\big\| |{\bf k}|^2b \big\|}_{L^{\infty,3}} 
  \cdot {\| |{\bf k}|^2h \|}_{L^{2,s}},
\end{split}
\end{equation}
which is consistent with \eqref{lemintLconc1}.

\smallskip
\noindent
{\it Proof of \eqref{lemintLconc2}}
In this case we need to use the refinement \eqref{2dintrefine}
to avoid losing too many factors of the high frequency weight.
Assuming that $|{\bf k_3}| \gg |{\bf k_1}|$
(the case when all frequencies are comparable being easier) 
and using again the bound \eqref{intLprM} we obtain
\begin{align}
\begin{split}
{\big\| \mathcal{L}_{a,b} [h]  \big\|}_{L^{2,s}}
& \lesssim \sum_{j} \Big\| \varphi_{j}({\bf k}) 
    \sum_{j_1<j-5} \int_{\R^2} \varphi_{j_1}({\bf k_1}) 2^{(3/2)j_1}
    \\
    & \quad
    \times {\Big\| |{\bf k_1}|^2 a_1 \cdot \varphi_{\sim j_1}({\bf k_2})|{\bf k_2}|^2 b_2 
    \cdot \big[ \langle {\bf k} \rangle^s|{\bf k}|^{2} |h({\bf k})| 
    + \langle {\bf k_3} \rangle^s |{\bf k_3}|^{2} |h({\bf k_3})| \big] 
    \Big\|}_{L^\infty_{{\bf k_2},{\bf k_3}}} \, d {\bf k_1} \Big\|_{L^2_{{\bf k}}}
\\
& \lesssim {\big\| |{\bf k}|^2a \big\|}_{L^{2,4}} 
  \cdot {\big\| |{\bf k}|^2b \big\|}_{L^{\infty,2}} 
  \cdot {\big\| |{\bf k}|^2 h \big\|}_{L^{\infty,s+2}},
\end{split}
\end{align}
This concludes the proof of the estimates \eqref{lemintLconc}.

\smallskip
\noindent
{\it Proof of (ii)}
Since we have already proven the first inequality in \eqref{lemdippartconcLuu}, 
it suffices to show the second one.
Recall that we have, see \eqref{Lguuweights},
\begin{align}\label{didpartLuupr1}
\begin{split}
& \frac{1}{2} \big( \big[ \langle {\bf k} \rangle^s, \mathcal{L}_g\big] h, 
  \langle {\bf k} \rangle^s h \big) 
 \\
 & = \frac{1}{2}
 \int_{\R^2} \int_{\R^6} \mathcal{M}_\mathcal{L}({\bf k},{\bf k_1},{\bf k_2},{\bf k_3}) \, \delta(\Phi)
 \delta(\Sigma) \, g_1 g_2
 \, d{\bf k_1}d{\bf k_2} \, \big(\jk^s\langle {\bf k_3}\rangle^{-s} 
  + \langle {\bf k_3}\rangle^s\jk^{-s} - 2\big) 
 \, \\
 & \qquad \qquad \qquad \qquad \qquad \qquad \qquad \qquad 
 \times \langle {\bf k_3}\rangle^s h({\bf k_3}) \jk^s h({\bf k}) \, d{\bf k_3}d{\bf k}.
\end{split}
\end{align}
We can then estimate
\begin{align}
\Big| \big(\big[\langle {\bf k} \rangle^s, \mathcal{L}_g\big] h, \langle {\bf k} \rangle^s h \big) \Big|
 & \lesssim {\big\| \mathcal{L}_g' h \big\|}_{L^{2,s-2}} {\| h \|}_{L^{2,s}}.
\end{align}
where $\mathcal{L}_g'$ is defined as (cf. with \eqref{2dLg'} and \eqref{Lguuweights})
\begin{align}
& \mathcal{L}'_{g} [h](\bf k) := \int_{\R^6} \mathcal{M}'({\bf k},{\bf k_1},{\bf k_2},{\bf k_3}) \, \delta(\Phi)
 \delta(\Sigma) \, \chi_{2<3} \chi_{1<0} \, g_1 g_2 \, h_3 \, d{\bf k_1}d{\bf k_2}d{\bf k_3},
\\
& \mathcal{M}'({\bf k},{\bf k_1},{\bf k_2},{\bf k_3}) := 
  \jk^2 \mathcal{M}({\bf k},{\bf k_1},{\bf k_2},{\bf k_3})
  \cdot \big(\jk^s\langle {\bf k_3}\rangle^{-s} 
  + \langle {\bf k_3}\rangle^s\jk^{-s} - 2\big) \langle  {\bf k_3}\rangle^s\jk^{-s}.
\end{align}
On the support of the collision operator, this symbol satisfies
the same upper bound as $\mathcal{M}$ in \eqref{intLprM} up to a multiplicative factor 
of the low frequencies $|{\bf k_1}|^2 + |{\bf k_2}|^2$.
This allows us to apply \eqref{lemintLconc1} (which does not use 
the difference $h-h_3$ in $\mathcal{L}$) to obtain
\begin{align}
\begin{split}
{\big\| \mathcal{L}'_{g} [h]  \big\|}_{L^{2,s-2}}
  & \lesssim  {\big\| |{\bf k}|^2 g \big\|}_{L^{\infty,5}}^2 \cdot {\||{\bf k}|^2 h \|}_{L^{2,s-2}},
\end{split}
\end{align}
which implies the desired conclusion.
\end{proof}

\smallskip
We now analyze the bounded part.

\begin{lemma}[Trilinear bounds for the bounded part]\label{lembdpart}
Consider the operator from \eqref{2dBg}  
\begin{align}\label{2dBg'}
\begin{split}
\mathcal{B}_{a, b} [h] & := \int_{\R^6} 
  \mathcal{M}({\bf k},{\bf k_1},{\bf k_2},{\bf k_3}) \, \delta(\Phi)
 \delta(\Sigma) 
 \\ & \times \big[ 2 \chi_{2<3} \chi_{1\geq 0} \, a_1 b_2 \, (h_3 - h) 
  + a_2 b_3 h - 2\chi_{2\geq 3} a_1 b_2 h + \widetilde{\chi}_{2,3} \, a_1b_2h_3 \big]
  \, d{\bf k_1}d{\bf k_2}d{\bf k_3}.
\end{split}
\end{align}
For all $10 \leq s \leq N_1+2$, and with $s_0 := s/2+6$, we have
\begin{subequations}\label{lembdpartconc}
\begin{align}\label{lembdpartconc1}
\begin{split}
{\big\| \mathcal{B}_{a,b}[h] \big\|}_{L^{2,s}}
  & \lesssim  {\big\| |{\bf k}|^2a \big\|}_{L^{\infty,s_0}} 
  \cdot {\big\| |{\bf k}|^2b \big\|}_{L^{\infty,s_0}} 
  \cdot {\| h \|}_{L^{2,s}},
\end{split}
\end{align}
as well as
\begin{align}\label{lembdpartconc2}
\begin{split}
{\big\| \mathcal{B}_{a,b}[h] \big\|}_{L^{2,s}}
  & \lesssim {\| a \|}_{L^{2,s}} \cdot {\| |{\bf k}|^2 b \|}_{L^{\infty,s_0}} 
  {\| |{\bf k}|^2 h \|}_{L^{\infty,s_0}},
\end{split}
\end{align}
and
\begin{align}\label{lembdpartconc3}
\begin{split}
{\big\| \mathcal{B}_{a,b}[h] \big\|}_{L^{2,s}}
  & \lesssim  {\big\| |{\bf k}|^2a \big\|}_{L^{\infty,s_0}} 
  \cdot {\| b \|}_{L^{2,s}} \cdot {\big\| |{\bf k}|^2 h \big\|}_{L^{\infty,s_0}}.
\end{split}
\end{align}
\end{subequations}
\end{lemma}

\begin{proof}
We will verify the three bounds \eqref{lembdpartconc} separately 
for the three pieces in $\mathcal{B}$ given by
\begin{subequations}\label{bdpartBj}
\begin{align}\label{bdpartB1}
\mathcal{B}^1_{a, b} [h] & := \int_{\R^6} 
  \mathcal{M}({\bf k},{\bf k_1},{\bf k_2},{\bf k_3}) \, \delta(\Phi) \delta(\Sigma) 
 \chi_{2<3} \chi_{1\geq 0} \, a_1 b_2 \, (h_3 - h) \, d{\bf k_1}d{\bf k_2}d{\bf k_3},
\\
\label{bdpartB2}
\mathcal{B}^2_{a, b} [h] & := \int_{\R^6} 
  \mathcal{M}({\bf k},{\bf k_1},{\bf k_2},{\bf k_3}) \, \delta(\Phi)
 \delta(\Sigma) \, a_2 b_3 h \, d{\bf k_1}d{\bf k_2}d{\bf k_3},
\\  
\label{bdpartB3}
\mathcal{B}^3_{a, b} [h] & := \int_{\R^6} 
  \mathcal{M}({\bf k},{\bf k_1},{\bf k_2},{\bf k_3}) \, \delta(\Phi)
 \delta(\Sigma) 
  \chi_{2\geq 3} a_1 b_2 h\, d{\bf k_1}d{\bf k_2}d{\bf k_3}. 
\end{align}
\end{subequations}
We disregard the last contribution in \eqref{2dBg'} since on its support we have
$|{\bf k_2}| \approx |{\bf k_3}| \approx \max(|{\bf k}|,|{\bf k_1}|)$
so that the arguments below can be applied almost verbatim.

For the proof of \eqref{lembdpartconc} 
it will suffice to analyze the various terms in \eqref{bdpartBj} and the size of the 
multipliers on the respective supports to show that 
(1) there is no loss of derivatives and (2) there are always suitable low frequency factors
to pair with the inputs that are estimated in $L^\infty$;
eventually we apply repeatedly the estimates in Lemma \ref{lemintPW}.
We assume that $| {\bf k}|\gg 1$, the complementary case being easier. 

\medskip
\noindent
{\it Proof of \eqref{lembdpartconc1}.}
We consider the operator \eqref{bdpartB1} and begin by noticing that on
its support, when $|{\bf k}| \gg 1$ (the complementary case being easier),
we must have 
\begin{align}\label{bdpartB1supp}
|{\bf k_1}| \approx |{\bf k_3}| \gtrsim |{\bf k}| \approx |{\bf k_2}|; 
\end{align}
in particular the functions $a$ and $b$ can also absorb the weights in ${\bf k}$. 
Using \eqref{2dintPWconc1}, the bound \eqref{2dMbound} on the symbols, 
and the restrictions \eqref{bdpartB1supp}, we can then estimate
\begin{align}
{\big\| \mathcal{B}^1_{a, b} [h] \big\|}_{L^{2,s}} 
  & \lesssim {\left\| \int_{\R^2} \frac{1}{|{\bf k} - {\bf k_3}|}
  {\left\| \langle {\bf k_1} \rangle^{1+s/2} \langle {\bf k_2} \rangle^{3+s/2} 
   \chi_{2<3} \chi_{1\geq 0} \, \right.}\right.}
   \\ 
   & \qquad  \times{\left.{\left. |{\bf k_1} |^2 a_1 |{\bf k_2} |^2 b_2 \, 
   |{\bf k_3} |^2 \big(|h_3| +  |h|\big) \right\|}_{L^\infty_{{\bf k_1}{\bf k_2}}}
   d \bf{k_3} \right\|}_{L^2_{{\bf k}}}
   \\
   & \lesssim  {\big\| |{\bf k}|^2 a \big\|}_{L^{\infty,s/2+1}}
   {\big\| |{\bf k}|^2b \big\|}_{L^{\infty,s/2+3}}
   {\left\| \int_{\R^2} \frac{\langle {\bf k} - {\bf k_3} \rangle^{-2}}{|{\bf k} - {\bf k_3}|}
   \langle {\bf k_3} \rangle^2 |{\bf k_3} |^2 |h_3| \, d \bf{k_3} \right\|}_{L^2_{{\bf k}}} 
   \\
   & \qquad + {\big\| |{\bf k}|^2 a \big\|}_{L^{\infty,s/2+5}}
   {\big\| |{\bf k}|^2b \big\|}_{L^{\infty,s/2+3}}
   {\left\| \int_{\R^2} \frac{ \langle{\bf k_3}\rangle^{-2}}{|{\bf k} - {\bf k_3}|} \, d \bf{k_3} 
   \, h({\bf k})  \right\|}_{L^2_{{\bf k}}}
   \\
   & \lesssim {\big\| |{\bf k}|^2 a \big\|}_{L^{\infty,s/2+5}} 
   {\big\| |{\bf k}|^2b \big\|}_{L^{\infty,s/2+3}}  \cdot {\| h \|}_{L^{2,4}},
\end{align}
which is consistent with \eqref{lembdpartconc1}.

For the term in \eqref{bdpartB2} there are no additional restrictions 
to the support besides the resonant relations. 
We can assume that $| {\bf k_3} | \gtrsim | {\bf k_2} |$, the other case being similar, so that 
\begin{align}\label{bdpartB2supp}
| {\bf k_3} | \gtrsim | {\bf k_2} |, \, | {\bf k} |
\end{align}
on the support of the integral.
Using  \eqref{2dintPWconc1} and distributing weights on $b$, we have
\begin{align}
\bigl\| \mathcal{B}^2_{a,b}[h] \bigr\|_{L^{2,s}}
&\lesssim
\Bigl\| \int_{\R^2}
\frac{\langle {\bf k_3}\rangle^{-3}}{|{\bf k}-{\bf k_3}|}
\Bigl\| \langle {\bf k_1}\rangle^2
        \langle {\bf k_2}\rangle^2|{\bf k_2}|^2 a_2
\cdot \langle {\bf k_3}\rangle^{s/2}|{\bf k_3}|^2 b_3
\cdot \langle {\bf k}\rangle^{s/2+3}|{\bf k}|^2 h
\Bigr\|_{L^\infty_{{\bf k_1},{\bf k_2}}}
\,d{\bf k_3}
\Bigr\|_{L^2_{\bf k}}
\\
&\lesssim
\bigl\| |{\bf k}|^2 a \bigr\|_{L^{\infty,2}}
\bigl\| |{\bf k}|^2 b \bigr\|_{L^{\infty,s/2+4}}
\cdot \bigl\| h \bigr\|_{L^{2,s/2+5}},
\end{align}
which is more than sufficient.

On the support of the term in \eqref{bdpartB3} we have 
\begin{align}\label{bdpartB3supp}
|{\bf k_2}| \gtrsim |{\bf k_3}|, \, |{\bf k}|, \, |{\bf k_1}|.
\end{align}
We can then distribute the ${\bf k}$ weights as above to estimate
\begin{align}
\bigl\| \mathcal{B}^3_{a,b}[h] \bigr\|_{L^{2,s}}
&\lesssim
\Bigl\| \int_{\R^2}\frac{\langle {\bf k_3}\rangle^{-3}}{|{\bf k}-{\bf k_3}|}
\Bigl\| \langle {\bf k_1}\rangle^2
        \langle {\bf k_2}\rangle^2|{\bf k_1}|^2 a_1
\cdot \langle {\bf k_2}\rangle^{s/2}|{\bf k_2}|^2 b_2
\cdot \langle {\bf k}\rangle^{s/2+3}|{\bf k}|^2 h
\Bigr\|_{L^\infty_{{\bf k_1},{\bf k_2}}}
\,d{\bf k_3}
\Bigr\|_{L^2_{\bf k}}
\\
&\lesssim
\bigl\| |{\bf k}|^2 a \bigr\|_{L^{\infty,2}}\bigl\| |{\bf k}|^2 b \bigr\|_{L^{\infty,s/2+4}}
\cdot \bigl\| h \bigr\|_{L^{2,s/2+5}},
\end{align}
therefore concluding the proof of \eqref{lembdpartconc1}.

\smallskip
\noindent
{\it Proof of \eqref{lembdpartconc2}.}
For this estimate we want to bound the first input function $a$ in the high $L^{2,s}$ space
and the rest in $L^{\infty,s_0}$. We can then proceed very similarly to the above but,
instead of using  \eqref{2dintPWconc1}, we invoke
\eqref{2dintPWconc3} for $\mathcal{B}^1_{a,b}[h]$ and $\mathcal{B}^3_{a,b}[h]$, 
and  \eqref{2dintPWconc2} for $\mathcal{B}^2_{a, b}[h]$.
Then, to estimate $\mathcal{B}^1_{a, b}[h]$ it suffices to use
\eqref{bdpartB1supp} to move weights onto $a_1$.
For the term $\mathcal{B}^2_{a, b}[h]$ in \eqref{bdpartB2} we can use the constraint 
\eqref{bdpartB2supp} to distribute (roughly) half of the weights on 
each of $b_3$ and $h$.
Finally, for the term $\mathcal{B}^3_{a, b}[h]$ in \eqref{bdpartB3} we
exploit the constraint \eqref{bdpartB3supp} to distribute again half of the weights on 
each of $b_2$ and $h$. 

\smallskip
\noindent
{\it Proof of \eqref{lembdpartconc3}.}
The proof of the last bound \eqref{lembdpartconc3} is very similar to the proof of 
\eqref{lembdpartconc2} and follows from using \eqref{2dintPWconc1} and \eqref{2dintPWconc2},
respectively, to place $b_3$ and $b_2$, respectively, in $L^{2,s}$, 
and using the constraints \eqref{bdpartB1supp}, \eqref{bdpartB2supp} and \eqref{bdpartB3supp}
for the three integrals \eqref{bdpartB1}-\eqref{bdpartB3}.
\end{proof}

\medskip
\section{A priori energy estimates and well-posedness}\label{secgLWP}

\subsection{Commuting vector-fields with the collision operator}\label{Ssec2dComm}
We denote $Z=Z_{\bf k}$ for either $\Omega$ or $S$ when we can use either of them.
Recall that $\mathcal{M}({\bf k},{\bf k_1},{\bf k_2},{\bf k_3})$ is homogeneous of degree $\alpha=6$,
\begin{align}\label{2dMlambda}
\mathcal{M}\big( \lambda\big({\bf k},{\bf k_1},{\bf k_2},{\bf k_3}) \big) = 
  \lambda^\alpha \mathcal{M}({\bf k},{\bf k_1},{\bf k_2},{\bf k_3})
\end{align}
for any $\lambda>0$, and rotation invariant, that is, for any rotation by an angle $\theta$, $R_\theta$,
\begin{align}\label{2dMtheta}
\mathcal{M}\big( R_\theta\big({\bf k},{\bf k_1},{\bf k_2},{\bf k_3}) \big) 
  = \mathcal{M}({\bf k},{\bf k_1},{\bf k_2},{\bf k_3}). 
\end{align}

The next simple lemma describes the action of $Z$ on the components of the collision kernel:

\begin{lemma}\label{lembasic0}
With the above definitions we have
\begin{align}\label{2dSOmegaM}
\begin{split}
& \sum_{j=0}^3 S_{{\bf k_j}}\mathcal{M}({\bf k},{\bf k_1},{\bf k_2},{\bf k_3}) 
  = (\alpha + 4)\, \mathcal{M}({\bf k},{\bf k_1},{\bf k_2},{\bf k_3}),
\\
& \sum_{j=0}^3 \Omega_{{\bf k_j}}\mathcal{M}({\bf k},{\bf k_1},{\bf k_2},{\bf k_3}) = 0.
\end{split}
\end{align}
and
\begin{align}\label{2dSOmegaPhi}
\begin{split}
& \sum_{j=0}^3 S_{{\bf k_j}} \delta\big( \Phi({\bf k},{\bf k_1},{\bf k_2},{\bf k_3}) \big) 
  = (4-\frac{1}{2})  \delta\big( \Phi({\bf k},{\bf k_1},{\bf k_2},{\bf k_3}) \big), 
\\
& \sum_{j=0}^3 \Omega_{{\bf k_j}}\delta\big( \Phi({\bf k},{\bf k_1},{\bf k_2},{\bf k_3}) \big) = 0.
\end{split}
\end{align}
Moreover (for us $d=2$)
\begin{align}\label{2dSOmegacut}
\begin{split}
& \sum_{j=0}^3 S_{{\bf k_j}} \delta^{(d)} \big( \Sigma({\bf k},{\bf k_1},{\bf k_2},{\bf k_3}) \big) =
 (4-d)\delta^{(d)} \big( \Sigma({\bf k},{\bf k_1},{\bf k_2},{\bf k_3}) \big),
\\
& \sum_{j=0}^3 \Omega_{{\bf k_j}} \delta^{(d)} \big( \Sigma({\bf k},{\bf k_1},{\bf k_2},{\bf k_3}) \big) 
  = 0,
\\
& \sum_{j=0}^3 S_{{\bf k_j}} \rho\big(|{\bf k_{l}}|/|{\bf k_{m}}|\big) =
 4\rho\big(|{\bf k_{l}}|/|{\bf k_{m}}|\big), \quad l,m\in\{0,\dots,3\}
\\
& \sum_{j=0}^3 \Omega_{{\bf k_j}} \rho\big(|{\bf k_{l}}|/|{\bf k_{m}}|\big) = 0, \quad l,m\in\{0,\dots,3\}.
\end{split}
\end{align}
for any smooth $\rho$.
\end{lemma}

\begin{proof}
The identities \eqref{2dSOmegaM} follow by 
calculating the derivatives $\partial_\lambda|_{\lambda=1}$ and 
$\partial_\theta|_{\theta=0}$ in \eqref{2dMlambda} and \eqref{2dMtheta}.
The two identities \eqref{2dSOmegaPhi} follow from
the homogeneity of the dispersion relation $\omega$
\begin{align}
\Phi(\lambda({\bf k},{\bf k_1},{\bf k_2},{\bf k_3})) 
  = \lambda^{1/2}\Phi({\bf k},{\bf k_1},{\bf k_2},{\bf k_3}),
\end{align}
and the last identities are direct calculations.
\end{proof}

Using the above, we can easily commute vector-fields with the collision kernel:

\begin{lemma}\label{lemcomm}
Let $\mathcal{Q}_{a,b} = \mathcal{L}_{a,b} + \mathcal{B}_{a,b}$ 
be as in \eqref{2dQgh}-\eqref{2dLBsplit} with \eqref{2dLg}-\eqref{2dBg},
acting on Schwartz functions. Let $Z\in\{\Omega,S\}$
and define 
\begin{align}
c_S := \alpha + \tfrac{3}{2}, \qquad c_\Omega:=0. 
\end{align}
Then, we have
\begin{align}\label{lemcomm1}
\big[Z, \mathcal{L}_{a,b} \big] = \mathcal{L}_{Za,b} + \mathcal{L}_{a,Zb}
  + c_Z \mathcal{L}_{a,b}, 
\end{align}
and
\begin{align}\label{lemcomm2}
\big[ Z , \mathcal{B}_{a,b} \big] = \mathcal{B}_{Za,b} + \mathcal{B}_{a,Zb}
  + c_Z \mathcal{B}_{a,b}.
\end{align}
In particular, we have
\begin{align}\label{lemcommconc}
\big[Z, \mathcal{Q}_{f,f} \big]f & = \mathcal{Q}_{Zf,f}[f] 
  + \mathcal{Q}_{f,Zf}[f] + c_Z \mathcal{Q}(f).
\end{align} 
\end{lemma}

\begin{proof}
From Lemma \ref{lembasic0} we see that the total homogeneity of 
the kernel in the collision integral is $\alpha - 5/2$.
We then directly apply 
$S_{\bf k} = {\bf k} \cdot \nabla_{\bf k}  + 1$ and integrate by parts to calculate
\begin{align}
S_{\bf k}  \mathcal{L}_{a,b} [h] 
 & = 2\int_{\R^6} \Big(-1-S_{\bf k_1}-S_{\bf k_2}-S_{\bf k_3} + \sum_{j=0}^3 S_{\bf k_j} \Big)
 \big[ \mathcal{M}({\bf k},{\bf k_1},{\bf k_2},{\bf k_3}) \, \delta(\Phi)
 \delta(\Sigma)\chi_{2<3} \chi_{1<0}  \big] 
 \\
 & \qquad \times \, a_1b_2 \, (h_3 - h) \, d{\bf k_1}d{\bf k_2}d{\bf k_3}
 \\
 & + 2\int_{\R^6} \mathcal{M}({\bf k},{\bf k_1},{\bf k_2},{\bf k_3}) \, \delta(\Phi)
 \delta(\Sigma) \, \chi_{2<3} \chi_{1<0} \, a_1b_2 \, (h_3 - S_{\bf k} h) \, d{\bf k_1}d{\bf k_2}d{\bf k_3}
 \\ 
 & \qquad = C_0[h] + C_1[h] + C_2[h] + C_3[h],
\end{align}
with
\begin{align}
C_0[h] & := \big(\alpha + \tfrac{3}{2}\big) \, \mathcal{L}_{a,b} [h],
\\
C_1[h] & := 2\int_{\R^6} \mathcal{M}({\bf k},{\bf k_1},{\bf k_2},{\bf k_3}) \, \delta(\Phi)
 \delta(\Sigma) \, \chi_{2<3} \chi_{1<0} \,  \big( S_{\bf k_1} a_1 \big) b_2 \, (h_3 - h)\, d{\bf k_1}d{\bf k_2}d{\bf k_3},
\\
C_2[h] & := 2\int_{\R^6} \mathcal{M}({\bf k},{\bf k_1},{\bf k_2},{\bf k_3}) \, \delta(\Phi)
 \delta(\Sigma) \, \chi_{2<3} \chi_{1<0} \,  a_1 \big( S_{\bf k_2} b_2  \big) \, (h_3 - h) \, d{\bf k_1}d{\bf k_2}d{\bf k_3},
\\ 
C_3[h] & := 2\int_{\R^6} \mathcal{M}({\bf k},{\bf k_1},{\bf k_2},{\bf k_3}) \, \delta(\Phi)
 \delta(\Sigma) \, \chi_{2<3} \chi_{1<0} \,  a_1 b_2 \, \big( \big(S_{\bf k_3} h_3\big) - S_{\bf k} h\big) 
 \, d{\bf k_1}d{\bf k_2}d{\bf k_3}.
\end{align}
This is \eqref{lemcomm1}.
The proof of \eqref{lemcomm2} is similar and \eqref{lemcommconc} follows from \eqref{2dLBsplit}.
\end{proof}

\medskip
\subsection{Energy estimates}\label{Ssec2dE}
We are now ready to prove our main energy estimate.

\begin{proposition}\label{2dpropen}
Recall the notation \eqref{2dZfields} and \eqref{2dZnorm} and let
\begin{align}\label{2dpropenE}
\mathcal{E}(t) := 
  {\big\| f \big\|}_{L^{2,N_1}}^2 + \sum_{1\leq |\ell| \leq 4} 
    {\big\| Z_{\bf k}^\ell f \big\|}_{L^{2,N_1-2|\ell|}}^2 \approx {\| f \|}_{Z}^2
\end{align}
and
\begin{align}\label{2dpropendiss}
\begin{split}
\mathcal{D}(t) & := - \big(\jk^{N_1} f, \mathcal{L}_f [\jk^{N_1} f] \big)
  - \sum_{1\leq |\ell| \leq 4}\big(\jk^{N_1-2|\ell|} Z^\ell f, \mathcal{L}_f [\jk^{N_1-2|\ell|} 
    Z^\ell f] \big) \geq 0.
\end{split}
\end{align}
Then, we have the following a priori energy estimate for solutions of \eqref{2ddtf}-\eqref{2dQ}:
\begin{align}\label{2dpropenconc}
\frac{d}{dt}\mathcal{E}(t) + \mathcal{D}(t) \lesssim \mathcal{E}(t)^2.
\end{align}
\end{proposition}

\medskip
\begin{proof}
We proceed in a few steps according to the natural strategy for obtaining
$L^2$-based energy estimates with vector-fields.
We first introduce the notation
\begin{align}\label{2dpropenpr1}
\begin{split}
& \mathcal{E}(t) = \mathcal{E}_{N_1}(t) + \sum_{1\leq |\ell| \leq 4} \mathcal{E}_{Z,\ell}(t) 
\\
& \mathcal{E}_{N_1}(t) :=  {\big\| f \big\|}_{L^{2,N_1}}^2,
\qquad \mathcal{E}_{Z,\ell}(t) := {\big\| Z^\ell f \big\|}_{L^{2,N_1-2|\ell|}}^2.
\end{split}
\end{align}

\medskip
\noindent
{\it Step 1: Evolution of the energy and basic weighted bound}.
We calculate, using \eqref{Lguu},
\begin{align}\label{2dpropenpr5}
\begin{split}
\frac{1}{2}\frac{d}{dt} \mathcal{E}_{N_1}(t)
  & = I_{N_1,\mathcal{L}}(t) + I_{N_1,\mathcal{B}}(t), \qquad 
\\
I_{N_1,\mathcal{L}}(t) & := \big(\jk^{N_1} f, \mathcal{L}_f[\jk^{N_1} f] \big)
  + \big(\jk^{N_1} f, \big[ \jk^{N_1}, \mathcal{L}_f\big]f \big),
\\
I_{N_1,\mathcal{B}}(t) & = \big(\jk^{N_1} f, \jk^{N_1} \mathcal{B}_f[f] \big),
\end{split}
\end{align}
and, similarly
\begin{align}\label{2dpropenpr6}
\begin{split}
\frac{1}{2}\frac{d}{dt} \mathcal{E}_{Z,\ell}(t) 
  =
 I_{Z,\ell,\mathcal{L}}(t) + I_{Z,\ell,\mathcal{B}}(t), 
\end{split}
\end{align} 
where
\begin{align}\label{2dpropenpr7}
\begin{split}
I_{Z,\ell,\mathcal{B}}(t) & = 
  \big(\jk^{N_1-2|\ell|} Z^\ell f, \jk^{N_1-2|\ell|} Z^\ell \mathcal{B}_f[f] \big), 
\\
I_{Z,\ell,\mathcal{L}}(t) & = 
  \big(\jk^{N_1-2|\ell|} Z^\ell f, \jk^{N_1-2|\ell|} Z^\ell \mathcal{L}_f[f] \big), \qquad  
  1\leq |\ell| \leq 4,
\end{split}
\end{align} 

In view of Lemma \ref{lemintL}(ii) and \eqref{lembdpartconc1} in Lemma \ref{lembdpart}, 
we see that the terms in \eqref{2dpropenpr5} satisfy
\begin{align}
I_{{N_1},\mathcal{L}}(t) + I_{{N_1},\mathcal{B}}(t) 
  \leq \big(\jk^{N_1} f, \mathcal{L}_f[\jk^{N_1} f] \big)
  + C {\big\| |{\bf k}|^2 f \big\|}_{L^{\infty,s_0}}^2 {\| f \|}_{L^{2,{N_1}}}^2,
\end{align}
where $s_0={N_1}/2+6$.
Combining this with Sobolev's embedding Lemma \ref{lemSob} we obtain, for $N_1$ large enough,
\begin{align}
\begin{split}
I_{{N_1},\mathcal{L}}(t) + I_{{N_1},\mathcal{B}}(t)
  & \leq \big(\jk^{N_1} f, \mathcal{L}_f[\jk^{N_1} f] \big)
  \\
  & + C \Big( \sum_{|n|\leq 2} {\big\| Z^{n} f \big\|}_{L^{2,s_0}}
  + {\big\| |{\bf k}|^2 f \big\|}_{L^{2,s_0}} \Big)^2 {\| f \|}_{L^{2,{N_1}}}^2
  \lesssim {\| f \|}_{Z}^2 {\| f \|}_{L^{2,{N_1}}}^2,
\end{split}
\end{align}
for $N_1$ large enough (so that $s_0 \leq N_1-4$)
consistently with the desired \eqref{2dpropenconc}.

\medskip
\noindent
{\it Step 2: Commuting with vector-fields and the top weighted energy term}.
We now analyze the energy bulk terms in \eqref{2dpropenpr6},
starting with the one involving the dissipation. 
We begin by applying \eqref{lemdippartconcLuu} with $h=Z^\ell f$, $s=N_1-2|\ell|$ to obtain
\begin{align}\label{2dpropenpr10}
\begin{split}
I_{Z,\ell,\mathcal{L}}(t) & = \big(\jk^{N_1-2|\ell|} Z^\ell f, 
  \jk^{N_1-2|\ell|} \big[ Z^\ell, \mathcal{L}_f\big]f \big)
  + \big(\jk^{N_1-2|\ell|} Z^\ell f, 
  \jk^{N_1-2|\ell|} \mathcal{L}_f [Z^\ell f] \big)
  \\
  & \leq {\big\| Z^\ell f \big\|}_{L^{2,N_1-2|\ell|}} 
  {\big\| \big[ Z^\ell, \mathcal{L}_f \big]f \big\|}_{L^{2,N_1-2|\ell|}}
  \\ 
  & + \big(\jk^{N_1-2|\ell|} Z^\ell f, \mathcal{L}_f [\jk^{N_1-2|\ell|} Z^\ell f] \big)
  + {\big\| |{\bf k}|^2 f \big\|}_{L^{\infty,6}}^2 {\big\| Z^\ell f \big\|}_{L^{2,N_1-2|\ell|}}^2.
\end{split}
\end{align}
As before, using Lemma \ref{lemSob} we see that, for $N_1$ large enough, 
the last term above satisfies
\begin{align}
{\big\| |{\bf k}|^2 f \big\|}_{L^{\infty,6}}^2 {\| Z^\ell f \|}_{L^{2,N_1-2|\ell|}}^2
 & \lesssim {\| f \|}_{Z}^2 \cdot \mathcal{E}_{Z,\ell}(t).
\end{align}
We are then left with handling the first term in \eqref{2dpropenpr6},
and the first term on the right-hand side of \eqref{2dpropenpr10} and aim to show the estimates 
\begin{align}
\label{2dpropenpr11}
{\big\| Z^\ell \mathcal{B}_f[f] \big\|}_{L^{2,N_1-2|\ell|}}
  & \lesssim  {\| f \|}_{Z}^3 
\\
\label{2dpropenpr11'}
{\big\| \big[ Z^\ell, \mathcal{L}_f \big]f \big\|}_{L^{2,N_1-2|\ell|}}
  & \lesssim  {\| f \|}_{Z}^3, \qquad 1\leq |\ell| \leq 4.
\end{align}
To compute the expressions to be estimated above 
we iterate the commutator identities from Lemma  \ref{lemcomm} and we will then apply
the Lemmas \ref{lemintL} and \ref{lembdpart} to bound them.

\medskip
{\it Proof of \eqref{2dpropenpr11}.}
Iterating the identity \eqref{lemcomm2} we can estimate
\begin{align}\label{2dpropenpr12}
\begin{split}
{\big\| Z^\ell \mathcal{B}_f[f] \big\|}_{L^{2,N_1-2|\ell|}}  
 \lesssim \sum_{\ell_1+\ell_2+\ell_3 \leq \ell} 
 {\big\| \mathcal{B}_{Z^{\ell_1} f, Z^{\ell_2} f}[Z^{\ell_3} f] \big\|}_{L^{2,N_1-2|\ell|}}.
\end{split}
\end{align}
In what follows all the sums are understood to be 
taken over multi-indexes as above and we will often omit to write this explicitly,
while only writing out additional specific constraints.
Since we are going to apply the estimates \eqref{lembdpartconc}, we define the index
\begin{align}
s_{0,\ell} := \frac{1}{2}(N_1-2|\ell|) + 6. 
\end{align}
If $|\ell_3|\geq 2$, so that in particular $|\ell_1|+|\ell_2|\leq 2$, we use 
\eqref{lembdpartconc1} to obtain
\begin{align}\label{2dpropenpr13}
\begin{split}
& \sum_{|\ell_3| \geq 2} 
 {\big\| \mathcal{B}_{Z^{\ell_1} f, Z^{\ell_2} f}[Z^{\ell_3} f] \big\|}_{L^{2,N_1-2|\ell|}}
 \\
 & \lesssim 
 \sum_{|\ell_3| \geq 2} {\big\| |{\bf k}|^2 Z^{\ell_1} f \big\|}_{L^{\infty,s_{0,\ell}}} 
 {\big\| |{\bf k}|^2 Z^{\ell_2} f \big\|}_{L^{\infty,s_{0,\ell}}}
 {\big\| Z^{\ell_3} f \big\|}_{L^{2,N_1-2|\ell|}}
 \\
 & \lesssim{\| f \|}_{Z}^2 \cdot \sum_{|\ell_3| \geq 2} 
  {\big\| Z^{\ell_3} f \big\|}_{L^{2,N_1-2|\ell|}}
\end{split}
\end{align}
having used again Lemma \ref{lemSob} to estimate the $L^\infty$ based norms.

If instead $|\ell_3| \leq 1$ and, without loss of generality, $|\ell_1| \geq |\ell_2|$
(so that $|\ell_2| \leq 2)$,
we use \eqref{lembdpartconc2} to estimate the function with most vector-fields in $L^2$:
\begin{align}\label{2dpropenpr14}
\begin{split}
& \sum_{|\ell_3|< 2, \, |\ell_1|\geq |\ell_2|} 
 {\big\| \mathcal{B}_{Z^{\ell_1} f, Z^{\ell_2} f}[Z^{\ell_3} f] \big\|}_{L^{2,N_1-2|\ell|}}
 \\
 & \lesssim \sum_{|\ell_3|< 2, \, |\ell_1|\geq |\ell_2|} 
 {\big\| Z^{\ell_1} f \big\|}_{L^{2,N_1-2|\ell|}}
 {\big\| |{\bf k}|^2 Z^{\ell_2} f \big\|}_{L^{\infty,s_{0,\ell}}} 
 {\big\| |{\bf k}|^2 Z^{\ell_3} f \big\|}_{L^{\infty,s_{0,\ell}}}
 \\
 & \lesssim{\| f \|}_{Z}^2 \cdot 
 \sum_{|\ell_1| \leq |\ell|} {\big\| Z^{\ell_1} f \big\|}_{L^{2,N_1-2|\ell|}}
\end{split}
\end{align}
having used again the weighted Sobolev embedding \eqref{lemSobconcZ}. 
This proves \eqref{2dpropenpr11}.

\medskip
{\it Proof of \eqref{2dpropenpr11'}.}
Using repeatedly the commutation identity \eqref{lemcomm1} we can estimate
\begin{align}\label{2dpropenpr20}
\begin{split}
{\big\| \big[ Z^\ell, \mathcal{L}_f \big]f \big\|}_{L^{2,N_1-2|\ell|}}
 \lesssim \sum_{\ell_1+\ell_2+\ell_3 \leq \ell, \,\ell_3 < \ell} 
 {\big\| \mathcal{L}_{Z^{\ell_1} f, Z^{\ell_2} f}[Z^{\ell_3} f] \big\|}_{L^{2,N_1-2|\ell|}};
\end{split}
\end{align}
notice that the high frequency term cannot have all the vector-fields, $\ell_3 < \ell$.
Also, thanks to the symmetry of the bounds \eqref{lemintLconc2} and \eqref{lemintLconc3}, 
without loss of generality, we may assume $|\ell_1| \geq |\ell_2|$, so that, in particular
$|\ell_2|\leq 2$ and we can always estimate $Z^{\ell_2}f$ in the $L^\infty$
based space and apply the embedding \eqref{lemSobconcZ} to control it.

We then look at various cases depending on the number of vector-fields that 
fall on the low frequency input $Z^{\ell_1} f$ since we are forced to estimate it
in $L^2$ when $3$ or $4$ vector-fields hit it, which, in turn, forces 
us to estimate $Z^{\ell_3}f$ in the $L^\infty$-based space with the danger of losing 
too many weights in ${\bf k_3}$.

First, if $|\ell_1|\leq 2$
we use \eqref{lemintLconc1} to estimate the high frequency term in $L^2$ and obtain
\begin{align}\label{2dpropenpr21}
\begin{split}
 & \sum_{
 |\ell_2|\leq |\ell_1|\leq 2, \, \ell_3 < \ell}
 {\big\| \mathcal{L}_{Z^{\ell_1} f, Z^{\ell_2} f}[Z^{\ell_3} f] \big\|}_{L^{2,N_1-2|\ell|}}
 \\
 & \lesssim \sum_{|\ell_2| \leq |\ell_1| \leq 2, \, \ell_3 < \ell}
 {\big\| |{\bf k}|^2 Z^{\ell_1} f \big\|}_{L^{\infty,3}} 
 {\big\| |{\bf k}|^2 Z^{\ell_2} f \big\|}_{L^{\infty,3}}
 {\big\| |{\bf k}|^2 Z^{\ell_3} f \big\|}_{L^{2,N_1-2|\ell|}}
 \\
 & \lesssim{\| f \|}_{Z}^2 \cdot 
 \sum_{\ell_3 < \ell} {\big\| Z^{\ell_3} f \big\|}_{L^{2,N_1-2(|\ell|-1)}};
\end{split}
\end{align}
this is consistent with the desired \eqref{2dpropenpr11'}.

When instead $|\ell_1| \geq 3$ (so that $|\ell_2|,|\ell_3| \leq 1$) 
we use that $|\ell_3| \leq |\ell|-|\ell_1| \leq |\ell|-3$
and the trilinear estimate \eqref{lemintLconc2} and the embeddings \eqref{lemSobconc}-\eqref{lemSobconcZ}
to obtain
\begin{align}\label{2dpropenpr22}
\begin{split}
 & \sum_{|\ell_1|\geq 3, \, |\ell_3| \leq |\ell| - 3}
 {\big\| \mathcal{L}_{Z^{\ell_1} f, Z^{\ell_2} f}[Z^{\ell_3} f] \big\|}_{L^{2,N_1-2|\ell|}}
 \\
 & \lesssim \sum_{|\ell_1|\geq 3, \, |\ell_3| \leq |\ell| - 3}
 {\big\| |{\bf k}|^2 Z^{\ell_1} f \big\|}_{L^{2,4}} 
 {\big\| |{\bf k}|^2 Z^{\ell_2} f \big\|}_{L^{\infty,2}}
 {\big\| |{\bf k}|^2 Z^{\ell_3} f \big\|}_{L^{\infty,N_1-2|\ell|+2}}
 \\
 & \lesssim{\| f \|}_{Z}^2 \cdot 
 \Big( \sum_{|\ell'| \leq |\ell| - 1} 
 {\big\| Z^{\ell'} f \big\|}_{L^{2,N_1-2(|\ell|-1)}}
 + \sum_{|\ell_3| \leq 1} {\big\| Z^{\ell_3} f \big\|}_{L^{2,N_1-2(|\ell|-2)}} \Big);
\end{split}
\end{align}
the above right-hand side is controlled by the right-hand side of 
\eqref{2dpropenpr11'} and this concludes the proof of the proposition.
\end{proof}

\medskip
\subsection{Construction of solutions and proof of Theorem \ref{MainThmg}}\label{sec2dsolutions}
To construct solutions we follow a scheme similar to the standard one used for quasilinear equations,
such as Euler's equations.

\smallskip
{\it Step 1: Viscous regularization.} 
Consider the viscous regularization of \eqref{kingeneralg} given by
\begin{align}\label{2dsolreg}
\begin{split}
\partial_t f(t,{\bf k}) = \mathcal{Q}_{f,f}[f]({\bf k}) - \delta \jk^4 f(t,{\bf k}),
  \qquad f(0) = f_0 \in Z_0,
\end{split}
\end{align}
where $\delta>0$, and $\mathcal{Q}(f) = \mathcal{Q}_{f,f}[f]$ 
denotes the trilinear operator with symbol $\mathcal{M}$.
We first want to construct $f^\delta$ for $\delta>0$ solving \eqref{2dsolreg}
by an iteration scheme.
Recall the definitions of the bounded and dissipative parts \eqref{2dBg} and \eqref{2dLg},
and \eqref{2dLBsplit}.
Define $\{ f_n^\delta \}_{n \in \Z_+}$ solutions of
\begin{align}\label{2dsolregn}
\begin{split}
& \partial_t f_n^\delta(t,{\bf k}) = 
  \mathcal{B}_{f_{n-1}^\delta(t),f_{n-1}^\delta(t)}[f_n^\delta(t)]({\bf k}) 
  + \mathcal{L}_{f_{n-1}^\delta(t),f_{n-1}^\delta(t)}[f_n^\delta(t)]({\bf k}) 
  - \delta \jk^4 f_{n}^\delta(t,{\bf k}), 
  \\
  & f_{n}^\delta(0) = f_0, \qquad n \geq 1,
\end{split}
\end{align}
with $f_0^\delta(t) \equiv f_0$.
For fixed $\delta>0$ and $n\in \Z_+$ we can find solutions $f_n^\delta$,
for example, solving the (linear) equation \eqref{2dsolregn} in its Duhamel representation
\begin{align}
f_n^\delta(t,{\bf k}) =  e^{-\delta \jk^4 t} f_0({\bf k})
 + \int_0^t e^{-\delta \jk^4 (t-s)} \Big( \mathcal{B}_{f_{n-1}^\delta,f_{n-1}^\delta}[f_n^\delta](s,{\bf k}) 
  + \mathcal{L}_{f_{n-1}^\delta,f_{n-1}^\delta}[f_n^\delta](s,{\bf k}) \Big) \, ds
\end{align}
by using the properties of the dissipative semi-group
\begin{align}
{\big\| e^{-\delta \jk^4 t} f \big\|}_{L^2_{\bf k}} \lesssim {\| f \|}_{L^2_{\bf k}},
\qquad
{\big\| \jk^2 e^{-\delta \jk^4 t} f \big\|}_{L^{2,s}} 
  \lesssim (\delta \, t)^{-1/2} {\| f \|}_{L^{2,s}},
\end{align}
to recover the loss of two weights.
This yields solutions $f_n^\delta \in C([0,T_{\delta,R}], L^2(\jk^4 d{\bf k}))$
for some $T_{\delta,R}>0$ by a fixed point argument and, similarly, in the space 
$C([0,T_{\delta,R}], Z)$, after commuting the equation with the vector-fields as done 
in Subsections \ref{Ssec2dComm}-\ref{Ssec2dE}.

Next, we prove that positivity is preserved by \eqref{2dsolregn}, that is, 
if $f_0 \geq 0$ then $f_n^\delta \geq 0$. 
Indeed, from \eqref{2dLBsplit}-\eqref{2dBg} we see that we can alternatively write
the equation in \eqref{2dsolregn} as
\begin{align}\label{2dsolregn>0}
\begin{split}
& \partial_t f_n^\delta(t,{\bf k}) = 
  \mathcal{G}_{f_{n-1}^\delta(t)}[f_n^\delta(t)]({\bf k}) 
  + \mathcal{A}_{f_{n-1}^\delta(t),f_{n-1}^\delta(t)} \cdot f_n^\delta(t,{\bf k}) 
  - \delta \jk^4 f_{n}^\delta(t,{\bf k})
\end{split}
\end{align}
where the operators $\mathcal{G}_g$ and  $\mathcal{A}_{g,g}$ are given by (cf. with \eqref{2dQgh})
\begin{align}
\mathcal{G}_g[h]({\bf k}) & := \int \mathcal{M}({\bf k},{\bf k_1},{\bf k_2},{\bf k_3}) 
  \, \delta\big(\Phi\big) \delta\big(\Sigma\big) (2 \chi_{2<3} + \wt{\chi}_{2,3}) g_1g_2 h_3
  \, d{\bf k_1}d{\bf k_2}d{\bf k_3},
\\
\mathcal{A}_{g,g}({\bf k}) & := \int \mathcal{M}({\bf k},{\bf k_1},{\bf k_2},{\bf k_3}) 
  \, \delta\big(\Phi\big) \delta\big(\Sigma\big) (g_2g_3 -2g_1g_2)
  \, d{\bf k_1}d{\bf k_2}d{\bf k_3}.
\end{align}
Noticing that $\mathcal{G}_g$ has non-negative kernel for $g\geq 0$,
and using the variation of constants formula for \eqref{2dsolregn>0},
we immediately see that $f_{n-1}^\delta\geq 0$ implies  $f_n^\delta\geq 0$.

Next, we claim that there exists an equivalent high order energy $\mathcal{F}_n(t)$ such that
\begin{align}\label{2dsolapest}
\begin{split}
\frac{d}{dt} \mathcal{F}_n(t)
  + \delta {\| \jk^2 f_{n}^\delta(t) \|}_{Z}^2
  \lesssim {\big\| f_{n-1}^\delta(t) \big\|}_{Z}^2 {\| f_{n}^\delta(t) \|}_{Z}^2,
  \qquad \mathcal{F}_n(t) \approx {\| f_{n}^\delta(t) \|}_{Z}^2,
\end{split}
\end{align}
where the implied constants are independent of $\delta \in (0,1)$.
To see the validity of \eqref{2dsolapest} 
we argue as in the proof of the energy estimates from Proposition \ref{2dpropen},
using the trilinear estimates from Lemmas \ref{lemintL} and \ref{lembdpart}.
The only difference for the regularized equation \eqref{2dsolregn} is the presence 
of the viscous term which needs to be commuted with the vector-fields, while the
$L^{2,N_1}$ estimate works directly as before.
Since $[\Omega_{\bf k},\jk^s] = 0$
the only additional terms that we need to account for are commutators with $S$.
Recalling $[S_{\bf k},\jk^s] = s |{\bf k}|^2\jk^{s-2}$,
and using also the fact that $S$ is anti-self-adjoint, we can see that, 
for all $N\in\Z$,
\begin{align}\label{2dsoldeltacomm}
\Big| \big\langle \jk^{N} [S,\delta \jk^4] g, 
  \jk^{N} S g \big\rangle \Big|
  \lesssim \delta {\big\| \jk^{N+2} g \big\|}_{L^2}^2.
\end{align}
Similar inequalities hold for higher powers of $S$, and therefore of $Z$, on the left-hand side:
\begin{align}\label{2dsoldeltacommell}
\Big| \big\langle \jk^{N} [Z^\ell,\delta \jk^4] g, 
  \jk^{N} Z^\ell g \big\rangle \Big|
  \lesssim \delta \sum_{|n| < |\ell|} {\big\| \jk^{N+2} Z^n g \big\|}_{L^2}^2.
\end{align}
Then, we can consider a weighted sum of the components in the energy $\mathcal{E}$
from \eqref{2dpropenE} of the form
\begin{align}\label{2dsolereweigh}
\mathcal{F}(t) := 
  C_0 {\big\| f \big\|}_{L^{2,N_1}}^2 + \sum_{1\leq |\ell| \leq 4} C_{|\ell|}
    {\big\| Z_{\bf k}^\ell f \big\|}_{L^{2,N_1-2|\ell|}}^2
\end{align}
and choose appropriately $C_0 \gg C_1 \gg \dots \gg C_4 > 1$ (independent of $\delta)$
in order to absorb all the lower order terms generated by \eqref{2dsoldeltacomm}-\eqref{2dsoldeltacommell}
into the dissipation term $\delta {\| \jk^2 f_{n}^\delta(t) \|}_{Z}^2$.
The desired \eqref{2dsolapest} follows.
In the remainder of this proof, for ease of notation,
we will sometimes adopt the convention of pretending that the equivalent 
re-weighted energy in \eqref{2dsolapest} (and its variants with fewer weights) 
is given exactly by the $Z$ norm squared (and its respective variants).


The a priori estimates \eqref{2dsolapest} imply that if we assign data 
$f_0 = f^\delta_n(0) \in B^+_R(Z_0)$ 
(recall the definition \eqref{bana3g}) then, there exists $T_R \approx R^{-2}> 0$ such that 
\begin{align}\label{2dsolapriori}
\sup_{t\in[0,T_{R}]} {\| f_n^\delta(t) \|}_{Z}^2
 + \delta \int_0^{T_{R}} {\| \jk^2 f_{n}^\delta(s) \|}_{Z}^2 \, ds\leq CR^2
\end{align}
for some absolute constant $C>0.$


\smallskip
{\it Step 2: Convergence of the iteration.} 
To prove the existence of viscous solutions and then of actual solutions,
let us define a family of weaker norms that have $\rho \in (0,4]$ 
fewer weights than $Z$ (see \eqref{2dZnorm}) by letting
\begin{align}\label{2dZnorm-}
{\| f \|}_{Z_{\rho}} := {\| f \|}_{L^{2,{N_1}-\rho}} + \sum_{1\leq |\ell| \leq 4}
  {\big\| Z_{\bf k}^\ell f \big\|}_{L^{2,N_1-2|\ell|-\rho}}.
\end{align}
Taking differences in the equation \eqref{2dsolregn} we write
\begin{align}\label{2dsolregdiff}
\begin{split}
\partial_t \big( f_{n+1}^\delta(t,{\bf k}) - f_n^\delta(t,{\bf k}) \big) 
  & = \mathcal{L}_{f_{n}^\delta(t),f_{n}^\delta(t)}[f_{n+1}^\delta(t) - f_{n}^\delta(t)]({\bf k})  
  \\
  & - \delta \jk^4 \big[ f_{n+1}^\delta(t,{\bf k}) - f_{n}^\delta(t,{\bf k}) \big]
  + \mathcal{B}_{n} + \mathcal{L}_{n},
\end{split}
\end{align}
where
\begin{align}
\label{2dsolregdiffB}
\mathcal{B}_{n} & := 
  \mathcal{B}_{f_{n}^\delta,f_{n}^\delta}[f_{n+1}^\delta]
  - \mathcal{B}_{f_{n-1}^\delta,f_{n-1}^\delta}[f_{n}^\delta],
\\
\label{2dsolregdiffL}
\mathcal{L}_{n} & := \mathcal{L}_{f_{n}^\delta,f_{n}^\delta}[f_n^\delta]
 - \mathcal{L}_{f_{n-1}^\delta,f_{n-1}^\delta}[f_n^\delta].
\end{align}
Commuting weights as in the previous subsections, 
applying the trilinear estimates in Lemmas \ref{lemintL} and \ref{lembdpart},
and using \eqref{2dsolapriori}, we see that
\begin{align}\label{2dsolregdiffbd}
\begin{split}
{\| \mathcal{B}_{n} \|}_{Z_{\rho}} & \lesssim  
  R^2 \cdot 
  \big( {\| f_{n+1}^\delta(t) - f_{n}^\delta(t) \|}_{Z_{\rho}}
  + {\| f_{n}^\delta(t) - f_{n-1}^\delta(t) \|}_{Z_{\rho}} \big)
\\
{\| \mathcal{L}_{n} \|}_{Z_{\rho}} & \lesssim 
  R \cdot 
  {\| f_{n}^\delta(t) - f_{n-1}^\delta(t) \|}_{Z_{\rho}} \cdot
  {\big\| \jk^2 f_{n}^\delta(t) \big\|}_{Z_{\rho}}.
\end{split}
\end{align}
Then, applying the energy estimate to \eqref{2dsolregdiff},
arguing as in the derivation of \eqref{2dsolapest} to build an equivalent energy 
that accounts for the lower order commutators with the viscous term,
and using ${\|f_{n}^\delta\|}_Z \leq CR$, we see that
\begin{align}\label{2dsolregdiff0}
\begin{split}
{\big\| f_{n+1}^\delta(t) - f_n^\delta(t)  \big\|}_{Z_2}^2
  & + \delta \int_0^t {\big\| \jk^2 \big[ f_{n+1}^\delta(s) - f_{n}^\delta(s) \big] \big\|}_{Z_2}^2\, ds
  \\
  & \lesssim R^2 \int_0^t \big( {\| f_{n+1}^\delta(s) - f_{n}^\delta(s) \|}_{Z_2}^2 
  + {\| f_{n}^\delta(s) - f_{n-1}^\delta(s) \|}_{Z_2}^2 \big) \, ds.
\end{split}
\end{align}
Using \eqref{2dsolregdiff0},
we deduce that there exists a time $T_R \gtrsim R^{-2}$ independent of $\delta$, 
such that 
\begin{align}\label{2dsolCauchy0}
\sup_{t\in[0,T_R]} {\big\| f_{n+1}^\delta(t) - f_n^\delta(t) \big\|}_{Z_2}
  \leq \frac{1}{2} \sup_{t\in[0,T_R]}{\| f_{n}^\delta(t) - f_{n-1}^\delta(t) \|}_{Z_2},
\end{align}
and, therefore $\{f_n^\delta\}_{n\in\Z_+}$ is Cauchy in $C([0,T_R]; Z_2)$.
We then define 
\begin{align}
f^\delta(t) := \lim_{n\rightarrow \infty}f_n^\delta(t), 
  \qquad f^\delta \in C([0,T_R]; Z_2), \quad f^\delta \geq 0,
\end{align}
and have obtained a (non-negative) solution of \eqref{2dsolreg}
(in a space with six fewer weights than $Z$).
Thanks to the a priori bound, convergence also holds in $C([0,T_R]; Z_\rho)$ for $\rho \in (0,4]$
and, by weak-$\ast$ convergence, we also have $f^\delta \in L^\infty([0,T_R]; Z)$.

\smallskip
{\it Step 3: Construction of solutions.} 
Next, observe that by an energy estimate like the one giving \eqref{2dsolapest}, 
we have 
\begin{align}
\begin{split}
& {\| f^\delta(t) \|}_{Z}^2 - {\| f^\delta(0) \|}_{Z}^2
  + \delta \int_0^t {\| \jk^2 f^\delta(s) \|}_{Z}^2 \, ds
  \lesssim \int_0^t {\big\| f^\delta(s) \big\|}_{Z}^4 \, ds, \qquad t\in[0,T_R]
\end{split}
\end{align}
where the implicit constant is independent of $\delta$.
In particular, for some absolute $C>0$
\begin{align}\label{2dsolapest'}
{\| f^\delta(t) \|}_{Z} \leq C \frac{{\| f^\delta(0) \|}_{Z}}{(1-CT_R{\| f^\delta(0) \|}_{Z}^2)^{1/2}}, 
  \qquad t \in [0,T_R]. 
\end{align}
We then consider $f^\delta,f^{\delta'}$, for $0<\delta'<\delta$, and take differences 
in \eqref{2dsolreg} to obtain
\begin{align}\label{2dsoldiffdelta}
\begin{split}
\partial_t \big( f^\delta(t,{\bf k}) - f^{\delta'}(t,{\bf k}) \big) 
  & = \mathcal{L}_{f^\delta(t),f^\delta(t)}[f^\delta(t) - f^{\delta'}(t)]({\bf k})  
  \\
  & - \delta \jk^4 \big[ f^\delta(t,{\bf k}) - f^{\delta'}(t,{\bf k}) \big]
  + (\delta' - \delta)\jk^4 f^{\delta'}(t,{\bf k})
  + \mathcal{B}_{\delta,\delta'} + \mathcal{L}_{\delta,\delta'},
\end{split}
\end{align}
with
\begin{align}
\label{2dsoldiffdeltaB}
\mathcal{B}_{\delta,\delta'} & := 
  \mathcal{B}_{f^\delta,f^\delta}[f^\delta]
  - \mathcal{B}_{f^{\delta'},f^{\delta'}}[f^{\delta'}],
\\
\label{2dsoldiffdeltaL}
\mathcal{L}_{\delta,\delta'} & := \mathcal{L}_{f^\delta,f^\delta}[f^{\delta'}]
 - \mathcal{L}_{f^{\delta'},f^{\delta'}}[f^{\delta'}].
\end{align}
By applying the same energy argument that led to \eqref{2dsolregdiff0} in the space $Z_4$,
together with the trilinear bounds \eqref{2dsolregdiffbd},
and the uniform bound \eqref{2dsolapest'}, we obtain
(abusing notation by omitting the re-weighted energy of the form \eqref{2dsolereweigh})
\begin{align}\label{2dsoldiffdeltabd}
\begin{split}
\frac{d}{dt} & {\big\| f^\delta(t) - f^{\delta'}(t)  \big\|}_{Z_4}^2
  + \delta {\big\| \jk^2 \big[ f^\delta(t) - f^{\delta'}(t) \big] \big\|}_{Z_4}^2
  \\
  & \lesssim R^2 {\| f^\delta(t) - f^{\delta'}(t) \|}_{Z_4}^2
  + |\delta - \delta'| {\| \jk^4 f^{\delta'}(t) \|}_{Z_4} {\| f^\delta(t) - f^{\delta'}(t) \|}_{Z_4}
  \\
  & \lesssim R^2 {\| f^\delta(t) - f^{\delta'}(t) \|}_{Z_4}^2 + |\delta - \delta'|^2.
\end{split}
\end{align}
This shows that $f^\delta$ is Cauchy in $C([0,T_R],Z_4)$. We then define
\begin{align}
f(t) := \lim_{\delta\rightarrow 0}f^\delta(t), 
  \qquad f \in C([0,T_R]; Z_4), \quad f \geq 0.
\end{align}
By interpolating the weighted $L^2$ norms, using the uniform bounds in the $Z$ norm,
we also obtain that $f^\delta$ converges
to $f$ in $C([0,T_R];Z_{\rho})$ for any $\rho \in (0,4]$.

Finally, since 
\begin{align}
{\big\| \mathcal{Q}(f^\delta) - \mathcal{Q}(f) \big\|}_{Z_4} \lesssim R^2 
  {\big\| f^\delta - f \big\|}_{Z_2}, 
  \qquad \delta {\| \jk^4 f^\delta \|}_{Z_4}
  \rightarrow 0,
\end{align}
we see that one can pass to the limit in the equation \eqref{2dsolreg}
and $f$ is indeed a solution of the wave kinetic equation \eqref{kingeneralg} (in $Z_4$)
on $[0,T_R]$ with $T_R \gtrsim R^{-2}$.

Using also weak compactness we see that $f \in L^\infty([0,T_R],Z)$ and is weakly continuous in time.
We have then obtained a solution
\begin{align}\label{2dsolf0}
f \in L^\infty([0,T_R],Z) \cap C([0,T_R];Z_\rho) \cap C_w([0,T_R];Z),
\end{align}
for $\rho \in (0,4]$, with $f \geq 0$ and $\|f\|_Z \leq CR$ for some absolute constant $C>0$.
Finally, for $s<t \in [0,T_R]$ we have
\begin{align}
{\| f(t) - f(s) \|}_{Z_2} \lesssim \int_s^t {\| \mathcal{Q}_{f,f}[f](\tau) \|}_{Z_2} \,d\tau
  \lesssim (t-s) \sup_{\tau \in[0,T_R]} {\| f(\tau) \|}_{Z}^3, 
\end{align}
which shows that $f \in \mathrm{Lip}([0,T_R];Z_2)$.

\smallskip
{\it Step 4: Uniqueness and (restricted) continuity of the flow map.} 
Let $f,g \in B^+_{CR}(Z_0([0,T]))$, for some $0\leq T \leq T_R$, 
be two solutions of \eqref{kingeneralg}. 
Taking differences and arguing as in Step 2 above, using once again the energy argument, 
we can bound
\begin{align}
\begin{split}
\frac{d}{dt} {\big\| f(t) - g(t)  \big\|}_{Z_2}^2 \lesssim
 \big( {\| f(t) \|}_Z + {\| g(t)\|}_Z \big)^2  {\| f(t) - g(t) \|}_{Z_2}^2
\end{split}
\end{align}
which implies 
\begin{align}\label{2dsoldiffbd2}
\begin{split}
\sup_{t\in[0,T]} {\big\| f(t) - g(t)  \big\|}_{Z_2} \lesssim e^{CR^2 T}  {\| f(0) - g(0) \|}_{Z_2}
\end{split}
\end{align}
and immediately gives uniqueness.
The estimate \eqref{2dsoldiffbd2} also shows that the solution map 
\begin{align}
f_0 \in B^+_R(Z_0) \longrightarrow f \in C([0,T_R],Z_2)
\end{align}
is Lipschitz continuous when the topology of the domain is that of $Z_2$.
Interpolating this with the uniform bounds in the highest $Z$ space gives the 
uniform continuity claim of the solution map on $Z_\rho$, $\rho>0.$


\smallskip
{\it Step 5: Continuity in time.}
The restricted continuity statement just proved can be upgraded to full continuity in time
by following a Bona-Smith type argument \cite{BS75}.
Continuity of the flow map on bounded neighbourhoods of $Z$ 
can also be shown by the same arguments; these arguments, however, 
do not give us uniform continuity in the top topology.
Although the arguments are somewhat standard, we provide some details for completeness, 
and since 
this is the first result of this type for a wave kinetic equation derived from a water waves 
model (or any quasilinear equation).

Let us define a space $S_2$ that is better than $Z$ by two weights
\begin{align}\label{BoSmsmootherspace}
{\| f \|}_{S_2} := {\| \jk^2 f \|}_Z \approx {\big\| f \big\|}_{L^{2,N_1+2}} 
	+ \sum_{1\leq |\ell| \leq 4} {\big\| Z_{\bf k}^\ell f \big\|}_{L^{2,N_1+2-2|\ell|}},
\end{align}
and note that for solutions with data $f(0) \in S_2 \cap B_R^+(Z_0)$ 
that exist on $[0,T_R]$,
our energy argument (using the precise estimates in Lemmas \ref{lemintL} and \ref{lembdpart})
gives persistence of regularity, that is
\begin{align}\label{BoSmSpres}
\sup_{t\in[0,T_R]} {\| f(t) \|}_{S_2} \lesssim_R {\| f(0) \|}_{S_2}.
\end{align}

The main ingredient in the Bona-Smith argument is a tame variant 
of the energy estimate from Proposition \ref{2dpropen}.
More precisely we claim that for any two solutions $f,g$ with data $f(0),g(0) \in B_R^+(Z_0)$, 
with the additional assumption that $g(0) \in S_2$, we have 
\begin{align}\label{BoSm1}
\begin{split}
\frac{d}{dt} & {\big\| f(t) - g(t)  \big\|}_{Z}
  \lesssim_{R} {\| f(t) - g(t) \|}_{Z}  +  {\| \jk^2 g \|}_{Z} {\| f(t) - g(t) \|}_{Z_2}
\end{split}
\end{align}
The inequality \eqref{BoSm1} follows by similar 
arguments to the ones used in the previous steps and in Proposition \ref{2dpropen}.
The idea is to write
\begin{align}\label{BoSmQdiff}
\mathcal{Q}_{f,f}[f] - \mathcal{Q}_{g,g}[g]  = \mathcal{Q}_{f,f}[f-g] + \mathcal{Q}_{f-g,f}[g] + \mathcal{Q}_{g,f-g}[g],
\end{align}
and apply the energy argument in Proposition \ref{2dpropen} to $f^\delta$ and $g^\delta$ 
solving regularized equations as in \eqref{2dsolreg}, and then letting $\delta$ go to zero. In what follows we will not make the distinction between $\delta$-regularized solutions
and the corresponding limiting solutions for ease of notation.
To see how \eqref{BoSmQdiff} implies \eqref{BoSm1}
observe that the term  $\mathcal{Q}_{f,f}[f-g]$ (or, better, its regularized version) is handled by commuting vector-fields and using \eqref{lemdippartconcLuu},
while the other two terms are estimated using the bounds in Lemma \ref{lemintL} and Lemma \ref{lembdpart}.
Note that, indeed, in \eqref{BoSm1} we are allowing the loss of two weights
on the `high position' input $g$ in the last two terms.

Next, let us introduce the truncation
\begin{align}\label{BoSmJ}
J_N f({\bf k}) :=  \chi(|{\bf k}|/N) f({\bf k}) 
\end{align}
where $0\leq \chi \leq 1$ is a smooth even function compactly supported 
in a ball of radius $2$, and equal to one in a ball of radius $1$.
Note that $\chi(|{\bf k}|/N)$ essentially commutes with the vector-fields
and that 
\begin{align}\label{BoSm3}
{\| J_N f \|}_Z \lesssim {\| f \|}_Z, \qquad 
  {\| J_N f \|}_{S_2} \lesssim N^2 {\| f \|}_Z, \qquad
  {\| (1-J_N) f \|}_Z \stackrel{N\rightarrow\infty}{\longrightarrow} 0.
\end{align}
Also,
\begin{align}\label{BoSm4}
N^2 {\| (1-J_N) f \|}_{Z_2} \lesssim {\| (1-\wt{J}_N) f \|}_Z 
  \stackrel{N\rightarrow\infty}{\longrightarrow} 0,
\end{align}
where $1-\wt{J}_N$ is a version of $1-J_N$ with a slightly enlarged support.

From the tame energy estimate \eqref{BoSm1} and the 
persistence of regularity \eqref{BoSmSpres} in $S_2$,
we have, for any two ($\delta$-regularized, technically)
solutions $f \in B_{CR}^+(Z([0,T_R]))$ and $g \in S_2([0,T_R])$,
\begin{align}\label{BoSm5}
\frac{d}{dt} & {\big\| f(t) - g(t)  \big\|}_{Z}
  \lesssim_{R} {\| f(t) - g(t) \|}_{Z}  +  {\| g(0) \|}_{S_2} {\| f(t) - g(t) \|}_{Z_2}.
\end{align}
Using that the flow map is Lipschitz on $Z_2$, see \eqref{2dsoldiffbd2}, we have
\begin{align}
\frac{d}{dt} & {\big\| f(t) - g(t)  \big\|}_{Z}
  \lesssim_{R} {\| f(t) - g(t) \|}_{Z}  +  {\| g(0) \|}_{S_2} e^{CR^2t}{\| f(0) - g(0) \|}_{Z_2},
\end{align}
so that Gronwall gives (also as $\delta \rightarrow 0$ by lower semi-continuity)
\begin{align}\label{BoSm6}
\sup_{t\in[0,T_R]} {\big\| f(t) - g(t)  \big\|}_{Z}
  \lesssim_{R} {\| f(0) - g(0) \|}_{Z}  +  {\| g(0) \|}_{S_2} {\| f(0) - g(0) \|}_{Z_2}.
\end{align}

Given $f_0 \in B_{R}^+(Z_0)$, define $f_0^{(N)} := J_Nf_0$, 
let $C=C_J$ be such that $f_0^{(N)} \in B_{C_JR}^+(Z_0)$,
and let $T_R \leq cR^{-2}$ be the time of existence
of the corresponding solution $f^{(N)} \in  B_{CR}^+(Z_0([0,T_R]))$ 
which satisfies, in view of \eqref{BoSmSpres},
\begin{align}\label{BoSm7}
\sup_{[0,T_R]} {\| f^{(N)}(t) \|}_{S_2} \lesssim N^2 {\| f_0 \|}_{Z}.
\end{align}
Applying \eqref{BoSm6} to $g= f^{(N)}$ and using \eqref{BoSm7} gives
\begin{align}\label{BoSm8}
\begin{split}
\sup_{t\in[0,T_R]} {\big\| f(t) - f^{(N)}(t)  \big\|}_{Z}
  & \lesssim_{R} {\| f_0 - J_Nf_0 \|}_{Z}  +  {\| J_Nf_0 \|}_{S_2} {\| f_0 - J_Nf_0 \|}_{Z_2}
  \\
  & \lesssim_{R} {\| f_0 - J_Nf_0 \|}_{Z}  +  N^2 {\| J_Nf_0 \|}_{Z} {\| f_0 - J_Nf_0 \|}_{Z_2}
  \\
  & \lesssim_{R} {\| f_0 - \wt{J}_Nf_0 \|}_{Z},
\end{split} 
\end{align}
having also used \eqref{BoSm4} in the last step.
In particular we see that $f$ is the uniform limit of the sequence 
of functions $f^{(N)}(t)$ that are Lipschitz-in-time with values in $Z$,
since ${\|\mathcal{Q}(f^{(N)})\|}_Z \lesssim_{R,N} 1$, and, therefore $f \in C([0,T_R];Z)$.

\smallskip
{\it Step 6: Continuity of the flow map.}
Consider $f_0 \in B_R^+(Z_0)$ and a sequence $f_{0,i} \rightarrow f_0$ in $Z_0$,
and denote $f$ and $f_i$ the corresponding solutions in $B_{CR}^+(Z([0,T_R]))$.
Consider the truncation of their data $J_N f_0$ and $J_N f_{0,i}$ and denote by
$f^{(N)}$ and $f^{(N)}_i$ their respective solutions. By the triangle inequality
\begin{align}\label{BoSm10}
{\| f - f_i \|}_{C_TZ} \leq 
  {\| f - f^{(N)} \|}_{C_TZ} + {\| f^{(N)} - f^{(N)}_i \|}_{C_TZ}
  + {\| f^{(N)}_i - f_i \|}_{C_TZ}.
\end{align}
First, by \eqref{BoSm8} we know that the first term on the right-hand side 
of \eqref{BoSm10} can be made arbitrarily small by choosing $N$ large enough.
Second, for any fixed $N$, the truncated data converge in $S_2$
\begin{align}\label{BoSm10.1}
{\| f^{(N)}(0) - f^{(N)}_i(0) \|}_{S_2} \lesssim   
  N^2 {\| f^{(N)}(0) - f^{(N)}_i(0) \|}_{Z} \lesssim   N^2 {\| f_0 - f_{0,i} \|}_{Z}  \stackrel{i \rightarrow \infty}{\longrightarrow} 0,
\end{align}
so that, 
using also \eqref{BoSm3}, the estimate \eqref{BoSm6} gives
 \begin{align}\label{BoSm10.2}
{\| f^{(N)} - f^{(N)}_i \|}_{C_TZ} 
	\lesssim_{R} {\| J_Nf_0 - J_Nf_{0,i} \|}_{Z}  +  N^2 {\| J_Nf_{0,i} \|}_{Z} 
		{\| J_Nf_0 - J_Nf_{0,i} \|}_{Z_2} 
		\\
		\lesssim_{R,N} {\| f_0 -  f_{0,i} \|}_{Z}  \stackrel{i \rightarrow \infty}{\longrightarrow} 0.
\end{align}
Third, using again \eqref{BoSm8}, $f_{0,i} \rightarrow f_0$ in $Z_0$,
the equivalence of the $Z$ and $Z_0$ norms, and the boundedness, of $1-\wt{J}_N$, we have
\begin{align}\label{BoSm10.3}
\limsup_{i\rightarrow\infty} {\| f^{(N)}_i - f_i \|}_{C_TZ} 
  \leq C(R) \limsup_{i\rightarrow\infty}{\| (1-\wt{J}_N) f_{0,i} \|}_{Z} 
  = C(R) {\| (1-\wt{J}_N) f_{0} \|}_{Z} 
\end{align}
which can be made arbitrarily small by taking $N$ large enough.
Note,  however, that this property is not uniform in the choice of $f_{0}$,
as one could make the support of $f_0$ shift towards high `frequencies' ${\bf k}$
so that for any chosen $N$ the right-hand side of \eqref{BoSm10.3} would not become small.

In conclusion, by using  \eqref{BoSm10.2} and \eqref{BoSm10.3}
in \eqref{BoSm10}, together with \eqref{BoSm8} we have
\begin{align}\label{BoSm12}
\begin{split}
\limsup_{i\rightarrow \infty} {\| f - f_i \|}_{C_TZ} 
  & \leq 
  {\| f - f^{(N)} \|}_{C_TZ} + \limsup_{i\rightarrow \infty} {\| f^{(N)} - f^{(N)}_i \|}_{C_TZ}
  \\ & \quad + \limsup_{i\rightarrow \infty} {\| f^{(N)}_i - f_i \|}_{C_TZ}
  \lesssim_R {\| (1-\wt{J}_N) f_{0} \|}_{Z}; 
\end{split}
\end{align}
taking the limit as $N\rightarrow \infty$ gives the desired conclusion.
$\hfill \Box$


\appendix

\medskip
\section{Formal derivation of the kinetic equation
}\label{secderivation}

In this section we provide a formal derivation of our main kinetic equation \eqref{kingeneral2}. 
For a one-dimensional interface in a two-dimensional fluid,
the cubic truncation of the main system \eqref{WWE1}--\eqref{defG0} is
\begin{equation}\label{WWE2}
\left\{
\begin{array}{l}
\partial_t h = |\partial_x|\phi+\mathcal{N}_2(h,\phi) + \mathcal{N}_3(h,h,\phi),
\\
\partial_t \phi = -g h  + \sigma \partial_x^2h+\mathcal{Q}_2(\phi,\phi) + \mathcal{Q}_3(h,\phi),
\end{array}
\right.
\end{equation}
where the quadratic terms are 
\begin{equation}\label{WWE3}
\begin{split}
&\mathcal{N}_2(h,\phi) :=-\partial_x(h\partial_x\phi)-|\partial_x|(h|\partial_x|\phi),
\\
&\mathcal{Q}_2(\phi,\phi) :=-\frac{1}{2}(\partial_x\phi)^2 + \frac{1}{2}(|\partial_x|\phi) ^2,
\end{split}
\end{equation}
and the cubic terms are
\begin{align}\label{WWE3cubic}
\begin{split}
\mathcal{N}_3(h,h,\phi) &:= |\partial_x|(h|\partial_x|(h|\partial_x|\phi))
 +\tfrac12|\partial_x|(h^2\partial_x^2\phi)
 +\tfrac12\partial_x^2(h^2|\partial_x|\phi),
 \\
\mathcal{Q}_3(h,\phi) &:= -(|\partial_x|\phi)|\partial_x|(h|\partial_x|\phi)-h(|\partial_x|\phi)\partial_x^2\phi
 -\tfrac{3\sigma}{2} (\partial_xh)^2\partial_x^2h .
\end{split}
\end{align}
As we will see below, only the quadratic terms matter 
for the derivation of the equation \eqref{kingeneral2}.
We use $\T_R=\R/(2\pi R\Z)$, $\Z_R=R^{-1}\Z$, and the Fourier conventions
\[
\widehat f(\xi)=\int_{\T_R}f(x)e^{-ix\xi}\,dx,
\qquad f(x)=\frac1{2\pi R}\sum_{\xi\in\Z_R}\widehat f(\xi)e^{ix\xi}.
\]
We can rewrite the quadratic nonlinear terms in the form
\begin{equation}\label{WWE3.5}
\begin{split}
&\mathcal{F}\{\mathcal{N}_2(h,\phi)\}(\xi)=\frac{1}{2\pi R}\sum_{\eta\in\Z_R}\widehat{h}(\eta)\widehat{\phi}(\xi-\eta)n_2(\eta,\xi-\eta),\\
&n_2(\eta,\rho):=(\eta+\rho)\rho-|\eta+\rho||\rho|,
\end{split}
\end{equation}
and
\begin{equation}\label{WWE3.6}
\begin{split}
&\mathcal{F}\{\mathcal{Q}_2(\phi,\phi)\}(\xi)=\frac{1}{2\pi R}\sum_{\eta\in\Z_R}\widehat{\phi}(\eta)\widehat{\phi}(\xi-\eta)q_2(\eta,\xi-\eta),\\
&q_2(\eta,\rho):=\frac{1}{2}(\eta\rho+|\eta||\rho|).
\end{split}
\end{equation}
Similarly we can write the cubic terms as 
\begin{equation}\label{WWE3.7}
\begin{split}
& \mathcal{F} \{\mathcal{N}_3(h,h,\phi)\}(\xi) = \frac{1}{(2\pi R)^2}
   \sum_{\eta+\sigma+\rho=\xi}
   \widehat{h}(\eta)\widehat{h}(\sigma) \widehat{\phi}(\rho) n_3(\eta,\sigma,\rho),
\end{split}
\end{equation}
and
\begin{align}\label{WWE3.8}
\begin{split}
\mathcal{F} \{\mathcal{Q}_3(h,\phi)\}(\xi)
 & =  \frac{1}{(2\pi R)^2}
   \sum_{\eta+\sigma+\rho=\xi}
   \widehat{h}(\eta)\widehat{\phi}(\sigma) \widehat{\phi}(\rho) q_{3,1}(\eta,\sigma,\rho)
   \\ & + \frac{1}{(2\pi R)^2}
   \sum_{\eta+\sigma+\rho=\xi}
   \widehat{h}(\eta)\widehat{h}(\sigma) \widehat{h}(\rho) q_{3,2}(\eta,\sigma,\rho),
\end{split}
\end{align}
where the cubic symbols $n_3$, $q_{3,1}$ and $q_{3,2}$ are real,
but their exact expressions are unimportant.

We assume $g=1$, $\sigma=1$, and define the main complex variable $u:\T_R\to\mathbb{C}$ by
\begin{equation}\label{WWE4}
u = (I-\partial_x^2)^{1/2}h+i|\partial_x|^{1/2}\phi,
  \qquad h = \frac{1}{2(I-\partial_x^2)^{1/2}}(u+\bar{u}),
  \quad \phi = \frac{1}{2i|\partial_x|^{1/2}}(u-\bar{u}).
\end{equation}
It follows from \eqref{WWE2} that the variable $u$ satisfies the equation
\begin{equation}\label{WWE5}
\begin{split}
&(\partial_t +i\Lambda)u = \mathcal{N}_{++}+\mathcal{N}_{+-}+\mathcal{N}_{--} + 
(I-\partial_x^2)^{1/2} \mathcal{N}_3 + i|\partial_x|^{1/2} \mathcal{Q}_3,
\\
&\widehat{\mathcal{N}_{\iota_1\iota_2}}(\xi):=\frac{1}{2\pi R}\sum_{\eta\in \Z_R}\widehat{u^{\iota_1}}(\eta)\widehat{u^{\iota_2}}(\xi-\eta) m_{\iota_1\iota_2}(\eta,\xi-\eta),\quad (\iota_1\iota_2)\in\{(++), (+-), (--)\},
\end{split}
\end{equation}
where $u^+=u, u^-=\overline{u}$, $\Lambda(\xi)=\sqrt{|\xi|+|\xi|^3}$, and the symbols $m_{\iota_1\iota_2}$ are given by
\begin{equation}\label{WWE6}
\begin{split}
&m_{++}(\eta,\rho):=\frac{\sqrt{1+(\eta+\rho)^2}n_2(\eta,\rho)}{8i\sqrt{1+\eta^2}\sqrt{|\rho|}}+\frac{\sqrt{1+(\eta+\rho)^2}n_2(\rho,\eta)}{8i\sqrt{1+\rho^2}\sqrt{|\eta|}}+\frac{\sqrt{|\eta+\rho|}q_2(\eta,\rho)}{4i\sqrt{|\eta|}\sqrt{|\rho|}},\\
&m_{+-}(\eta,\rho):=-\frac{\sqrt{1+(\eta+\rho)^2}n_2(\eta,\rho)}{4i\sqrt{1+\eta^2}\sqrt{|\rho|}}+\frac{\sqrt{1+(\eta+\rho)^2}n_2(\rho,\eta)}{4i\sqrt{1+\rho^2}\sqrt{|\eta|}}-\frac{\sqrt{|\eta+\rho|}q_2(\eta,\rho)}{2i\sqrt{|\eta|}\sqrt{|\rho|}},\\
&m_{--}(\eta,\rho):=-\frac{\sqrt{1+(\eta+\rho)^2}n_2(\eta,\rho)}{8i\sqrt{1+\eta^2}\sqrt{|\rho|}}-\frac{\sqrt{1+(\eta+\rho)^2}n_2(\rho,\eta)}{8i\sqrt{1+\rho^2}\sqrt{|\eta|}}+\frac{\sqrt{|\eta+\rho|}q_2(\eta,\rho)}{4i\sqrt{|\eta|}\sqrt{|\rho|}},
\end{split}
\end{equation}
after symmetrization, where $n_2,q_2$ are defined in \eqref{WWE3.5}--\eqref{WWE3.6}.
The symbols extend continuously by zero when either input frequency vanishes, and vanish when the output frequency is zero. 
We do not need to write out the symbols of the cubic terms in the equation \eqref{WWE5} as 
trilinear operators acting on $u$ and $\bar{u}$ but only remark
that the associated symbols are purely imaginary
as one can see by using \eqref{WWE4} and counting the copies of $\phi$ 
that are present in \eqref{WWE3.7} and \eqref{WWE3.8}.

We assign random initial data at scale $R$:
\begin{align}\label{KWE0}
b(\xi) = \sqrt{\varphi(\xi)} g(\xi;\omega) \, \varepsilon R^{1/2}
\end{align}
where, for this calculation, $\varphi\geq0$ is smooth and compactly supported away from zero, and the $g(\xi;\omega)$ are independent centered circular complex Gaussians, normalized by
\[
\mathbb E[g(\xi)\overline{g(\eta)}]=\mathbf1_{\xi=\eta},
\qquad \mathbb E[g(\xi)g(\eta)]=0.
\]

We define the frequency profile $a$ by
\begin{equation}\label{Quadprof}
a(t,\xi):=\widehat{u}(t,\xi)e^{it\Lambda(\xi)}.
\end{equation}
The equation \eqref{WWE5} becomes
\begin{align}\label{Quad6}
\begin{split}
a(t,\xi) & = b(\xi) + \int_0^t\frac{1}{2\pi R} \sum_{\ast} \sum_{\eta+\rho=\xi}
  m_{\iota_1\iota_2}(\eta,\rho)
  e^{is\Phi_{\iota_1\iota_2}(\eta,\rho)} a^{\iota_1}(s,\eta)a^{\iota_2}(s,\rho) \,ds
\\
& \, + \int_0^t\frac{1}{(2\pi R)^2} \sum_{\ast\ast} \sum_{\eta+\sigma+\rho=\xi}
  m_{\iota_1\iota_2\iota_3}(\eta,\sigma,\rho)
  e^{is\Phi_{\iota_1\iota_2\iota_3}(\eta,\sigma,\rho)} a^{\iota_1}(s,\eta) 
  a^{\iota_2}(s,\sigma)a^{\iota_3}(s,\rho) \,ds
\end{split}
\end{align}
where $\sum_{\ast}$ denotes the sum over signs 
$(\iota_1\iota_2)\in\{(++), (+-), (--)\}$, 
$\sum_{\ast\ast}$ denotes the sum over signs 
$(\iota_1\iota_2\iota_3)\in\{(+++), (++-), (+--),(---)\}$, 
$a^+(t,\xi) := a(t,\xi)$, $a^-(t,\xi) := \overline{a}(t,-\xi)$, 
and the phases are
\begin{align}\label{Quad7}
\begin{split}
& \Phi_{\iota_1\iota_2}(\eta,\rho) := \Lambda(\eta+\rho)-\iota_1\Lambda(\eta)-\iota_2\Lambda(\rho),
\\
& \Phi_{\iota_1\iota_2\iota_3}(\eta,\sigma,\rho) := 	
  \Lambda(\eta+\sigma+\rho)-\iota_1\Lambda(\eta)-\iota_2\Lambda(\sigma)-\iota_3\Lambda(\rho) .
\end{split}
\end{align}
In view of the prior observation about the cubic symbols, we have that  
$m_{\iota_1\iota_2\iota_3}(\eta,\sigma,\rho) \in i \R$.
This will be relevant to show that the cubic terms in \eqref{Quad6}
do not contribute to the kinetic equation.

Next, we find an approximate solution profile for \eqref{Quad6} up to cubic order in the data
by iterating. The first iterate is $a_1(t,\xi)=b(\xi)$, and the 
quadratic approximation is
\begin{equation}\label{Quad10}
a_2(t,\xi) = b(\xi) + \sum_\ast a_{2,\iota_1\iota_2}(t,\xi) 
\end{equation}
where
\begin{align}\label{Quad11}
\begin{split}
a_{2,\iota_1\iota_2}(t,\xi) & := \frac{1}{2\pi R} \sum_{\eta+\rho=\xi}
  m_{\iota_1\iota_2}(\eta,\rho)
  \frac{e^{it\Phi_{\iota_1\iota_2}(\eta,\rho)} -1}{i\Phi_{\iota_1\iota_2}(\eta,\rho)}b^{\iota_1}(\eta)b^{\iota_2}(\rho).
\end{split}
\end{align}
Substituting $a_2$ back into the Duhamel formula yields the 
cubic approximation: 
\begin{align}\label{Quad20}
 \begin{split}
a_{3}(t,\xi) = a_{2}(t,\xi) + 
  III_1(t,\xi) + III_2(t,\xi)  + III_3(t,\xi)  + III_4(t,\xi)
  \\
  + III_5(t,\xi) + \mathcal R(t,\xi),
\end{split}
\end{align}
where
\begin{itemize}
 
 \item The terms $III_1, \dots, III_4$ on the right-hand side of \eqref{Quad20}
are cubic terms generated by the quadratic nonlinearity 
that will contribute to the second moment when paired with $\overline{b(\xi)}$; 
they have the form $b(\eta)b(\sigma)\overline{b(-\rho)}$, with $\xi=\eta+\sigma+\rho$,
and exact formulas are given below in \eqref{Quad_I}-\eqref{Quad_III}.
 
 \item The term $III_5$ consists of the contributions from the cubic nonlinearity in
 \eqref{Quad6} that 
 can have non-zero Wick contraction when paired with $\overline{b(\xi)}$, that is,
 those of the form $b(\eta)b(\sigma)\overline{b(-\rho)}$;
see the formula \eqref{Quad_III_5} below. We will eventually see that these do not
contribute to the kinetic wave equation.

 \item The remainder $\mathcal R$ in \eqref{Quad20}
consists of all the other cubic terms whose pairing with $\overline{b(\xi)}$ vanishes by 
circular Gaussian symmetry, and terms that are quartic in the initial data.

\end{itemize}

The first term $III_1(t,\xi)$ arises from the $(++)$ interaction crossing with $a_{2,+-}$,
\begin{align}\label{Quad_I}
\begin{split}
III_1(t,\xi) & := 2\int_0^t\frac{1}{2\pi R} \sum_{k_1+k_2=\xi} m_{++}(k_1,k_2) e^{is\Phi_{++}(k_1,k_2)} b(k_1) a_{2,+-}(s,k_2) \,ds \\
& = \frac{2}{(2\pi R)^2} \sum_{\eta+\sigma+\rho=\xi} \frac{m_{++}(\eta,\sigma+\rho)m_{+-}(\sigma,\rho)}{i\Phi_{+-}(\sigma,\rho)} \\
& \qquad \times \Big[\frac{e^{it[\Phi_{++}(\eta,\sigma+\rho)+\Phi_{+-}(\sigma,\rho)]}-1}{i[\Phi_{++}(\eta,\sigma+\rho) +\Phi_{+-}(\sigma,\rho)]} - \frac{e^{it\Phi_{++}(\eta,\sigma+\rho)}-1}{i\Phi_{++}(\eta,\sigma+\rho)}\Big] b(\eta)b(\sigma) \overline{b(-\rho)}.
\end{split}
\end{align}

The second term $III_2(t,\xi)$ arises from the $(+-)$ interaction crossing with $a_{2,++}$ in the first argument,
\begin{align}\label{Quad_II1}
\begin{split}
III_2(t,\xi) & := \int_0^t\frac{1}{2\pi R} \sum_{k_1+k_2=\xi} m_{+-}(k_1,k_2) e^{is\Phi_{+-}(k_1,k_2)} a_{2,++}(s,k_1) \overline{b(-k_2)} \,ds \\
& = \frac{1}{(2\pi R)^2} \sum_{\eta+\sigma+\rho=\xi} \frac{m_{+-}(\eta+\sigma,\rho)m_{++}(\eta,\sigma)}{i\Phi_{++}(\eta,\sigma)} \\
& \qquad \times \Big[\frac{e^{it[\Phi_{+-}(\eta+\sigma,\rho) + \Phi_{++}(\eta,\sigma)]}-1}{i[\Phi_{+-}(\eta+\sigma,\rho)+\Phi_{++}(\eta,\sigma)]} - \frac{e^{it\Phi_{+-}(\eta+\sigma,\rho)}-1}{i\Phi_{+-}(\eta+\sigma,\rho)}\Big] b(\eta)b(\sigma)\overline{b(-\rho)}.
\end{split}
\end{align}

The third term $III_3(t,\xi)$ arises from the $(+-)$ interaction crossing with $a_{2,+-}$ in the second argument,
\begin{align}\label{Quad_II2}
\begin{split}
III_3(t,\xi) & := \int_0^t\frac{1}{2\pi R} \sum_{k_1+k_2=\xi} m_{+-}(k_1,k_2) e^{is\Phi_{+-}(k_1,k_2)} b(k_1) \overline{a_{2,+-}(s,-k_2)} \,ds \\
& = \frac{1}{(2\pi R)^2} \sum_{\eta+\sigma+\rho=\xi} \frac{m_{+-}(\eta,\sigma+\rho) \overline{m_{+-}}(-\rho,-\sigma)}{-i\Phi_{+-}(-\rho,-\sigma)} \\
& \qquad \times \Big[\frac{e^{it[\Phi_{+-}(\eta,\sigma+\rho) - \Phi_{+-}(-\rho,-\sigma)]}-1}{i[\Phi_{+-}(\eta,\sigma+\rho) - \Phi_{+-}(-\rho,-\sigma)]} - \frac{e^{it\Phi_{+-}(\eta,\sigma+\rho)}-1}{i\Phi_{+-}(\eta,\sigma+\rho)}\Big] b(\eta)b(\sigma) \overline{b(-\rho)}.
\end{split}
\end{align}

The fourth term $III_4(t,\xi)$ arises from the $(--)$ interaction crossing with $a_{2,--}$ in the second argument,
\begin{align}\label{Quad_III}
\begin{split}
III_4(t,\xi) & := 2\int_0^t\frac{1}{2\pi R} \sum_{k_1+k_2=\xi} m_{--}(k_1,k_2) e^{is\Phi_{--}(k_1,k_2)} \overline{b(-k_1)}\,\overline{a_{2,--}(s,-k_2)} \,ds \\
& = \frac{2}{(2\pi R)^2} \sum_{\eta+\sigma+\rho=\xi} \frac{m_{--}(\rho,\eta+\sigma) \overline{m_{--}}(-\eta,-\sigma)}{-i\Phi_{--}(-\eta,-\sigma)} \\
& \qquad \times \Big[\frac{e^{it[\Phi_{--}(\rho,\eta+\sigma) - \Phi_{--}(-\eta,-\sigma)]}-1}{i[\Phi_{--}(\rho,\eta+\sigma) - \Phi_{--}(-\eta,-\sigma)]} - \frac{e^{it\Phi_{--}(\rho,\eta+\sigma)}-1}{i\Phi_{--}(\rho,\eta+\sigma)}\Big] b(\eta)b(\sigma) \overline{b(-\rho)}.
\end{split}
\end{align}

The fifth term $III_5(t,\xi)$ arises from the cubic terms \eqref{WWE3.7} and \eqref{WWE3.8}
by first expressing $(h,\phi)$ in terms of $(u,\bar{u})$ via \eqref{WWE4},
then retaining only the interactions of the type $(u,u,\bar{u})$, and finally substituting
the linear solution; they have the form 
\begin{equation}\label{Quad_III_5}
\begin{split}
& III_5(t,\xi) = \frac{i}{(2\pi R)^2}
   \sum_{\eta+\sigma+\rho=\xi} \int_0^t e^{is \Omega} \,
   b(\eta)b(\sigma) \overline{b(-\rho)} \, N(\xi;\eta,\sigma,\rho) \, ds,
\end{split}
\end{equation}
where $N$ is real valued, as already argued above, and $\Omega$ is the cubic phase \eqref{WWEPhase}
given below.

Summing all the terms \eqref{Quad_I}-\eqref{Quad_III_5}
defines the general cubic operator 
\begin{align}\label{Quad40}
 \begin{split}
& (III_1 + III_2 + III_3 + III_4 + III_5)(t,\xi)
 \\ & \quad = \frac{1}{(2\pi R)^2} \sum_{\eta+\sigma+\rho=\xi} 
  \sum_{j=1}^4 (L_j + M_j)(t,\xi;\eta,\sigma,\rho)
  \cdot b(\eta) b(\sigma) \overline{b(-\rho)}
  \\
  & \quad + \frac{i}{(2\pi R)^2}
   \sum_{\eta+\sigma+\rho=\xi}
   N(\xi;\eta,\sigma,\rho) \int_0^t e^{is \Omega} \, ds \cdot b(\eta) b(\sigma) \overline{b(-\rho)}
\end{split}
\end{align}
The symbols $M_j$ are those arising from \eqref{Quad_I}-\eqref{Quad_III} 
that contain the combined (cubic) phase
\begin{align}\label{WWEPhase}
\Omega = \Lambda(\xi)-\Lambda(\eta)-\Lambda(\sigma)+\Lambda(\rho), \qquad \xi=\eta+\sigma+\rho.
\end{align}
Both Wick contractions used below force $\Omega=0$. Since the quadratic symbols are purely imaginary, their products in the cubic terms are real. 
Away from a vanishing inner phase, the $M_j$ contributions are therefore purely imaginary and disappear when taking the real part of the second moment.
At a vanishing inner phase, the separated $L_j$ and $M_j$ expressions must be interpreted through the original Duhamel integral. Indeed, if $\beta$ is the outer phase, the effective inner phase after contraction is $-\beta$, and
\[
\Re\int_0^t e^{is\beta}\int_0^s e^{-is'\beta}\,ds'\,ds
=\frac{1-\cos(t\beta)}{\beta^2},
\]
with continuous value $t^2/2$ at $\beta=0$. 
All subsequent real-part expressions involving $L_j$ use this continuous extension.
The symbols that contribute to the second moment, and 
will therefore determine the dynamics, are the $L_j$'s, which are given by
\begin{align}\label{Quad40sym1}
L_1(t,\xi;\eta,\sigma,\rho) & := 2\frac{m_{++}(\eta,\sigma+\rho)m_{+-}(\sigma,\rho)}{\Phi_{+-}(\sigma,\rho)} \frac{e^{it\Phi_{++}(\eta,\sigma+\rho)}-1}{\Phi_{++}(\eta,\sigma+\rho)},  \\
L_2(t,\xi;\eta,\sigma,\rho) & := \frac{m_{+-}(\eta+\sigma,\rho)m_{++}(\eta,\sigma)}{\Phi_{++}(\eta,\sigma)} \frac{e^{it\Phi_{+-}(\eta+\sigma,\rho)}-1}{\Phi_{+-}(\eta+\sigma,\rho)}, \\
L_3(t,\xi;\eta,\sigma,\rho) & := - \frac{m_{+-}(\eta,\sigma+\rho) \overline{m_{+-}}(-\rho,-\sigma)}{\Phi_{+-}(-\rho,-\sigma)} \frac{e^{it\Phi_{+-}(\eta,\sigma+\rho)}-1}{\Phi_{+-}(\eta,\sigma+\rho)}, \\
L_4(t,\xi;\eta,\sigma,\rho) & := - 2\frac{m_{--}(\rho,\eta+\sigma) \overline{m_{--}}(-\eta,-\sigma)}{\Phi_{--}(-\eta,-\sigma)} \frac{e^{it\Phi_{--}(\rho,\eta+\sigma)}-1}{\Phi_{--}(\rho,\eta+\sigma)}.
\end{align}

\subsection{Wick contraction of the second moment}

Keeping terms through degree four in the initial data, we obtain
\begin{align}\label{a3mom}
\begin{split}
\mathbb{E} \big|a_{3}(t,\xi)\big|^2 &=\mathbb{E}|b(\xi)|^2 +
  \mathbb{E} \big| a_{2,++}(t,\xi) \big|^2 
  + \mathbb{E} \big| a_{2,+-}(t,\xi) \big|^2 + \mathbb{E} \big| a_{2,--}(t,\xi) \big|^2
  \\
  & + 2\Re \mathbb{E} \big[ (III_1+III_2+III_3+III_4)(t,\xi) \cdot \overline{b(\xi)} \big] 
  \\ 
  & + 2\Re \mathbb{E} [ III_5(t,\xi) \cdot \overline{b(\xi)}]
  + O_{R,t,\varphi}(\varepsilon^6)
\end{split}
\end{align}
where the remainder $ O_{R,t,\varphi}(\varepsilon^6)$ is for fixed $R,t,\varphi$
(in particular, no uniform estimate on kinetic time scales is asserted)
and terms of odd degree in $b$ are not present since they vanish by Gaussian symmetry. 
Similarly to before, we observe that the Wick contraction of the cubic terms $III_5$ with 
$\overline{b(\xi)}$ gives a purely imaginary contribution, since the contraction forces $\Omega=0$, and, therefore, $2\Re \mathbb{E} [ III_5(t,\xi) \cdot \overline{b(\xi)}] = 0$.

Using Wick's theorem we now pair the initial random variables in the expression in \eqref{a3mom}. 
The resulting products of covariances, 
with the indicated pairing multiplicities, are valid also when modes coincide.
For the quadratic terms, we substitute $\xi = \eta + \rho$ and use the symmetries of the symbols $m_{++}$ and $m_{--}$ and the vanishing of $m_{+-}(\eta,-\eta)$ to see that
\begin{align}\label{a3a'}
\begin{split}
(2\pi R)^2 \mathbb{E} \big| a_{2,++} \big|^2 & = 2 \sum_{\eta+\rho=\xi} |m_{++}(\eta,\rho)|^2
  \Big| \frac{e^{it\Phi_{++}(\eta,\rho)} -1}{i\Phi_{++}(\eta,\rho)} \Big|^2
  \mathbb{E}|b(\rho)|^2\,\mathbb{E}|b(\eta)|^2, \\
(2\pi R)^2 \mathbb{E} \big| a_{2,+-} \big|^2 & = \sum_{\eta+\rho=\xi} |m_{+-}(\eta,\rho)|^2
  \Big| \frac{e^{it\Phi_{+-}(\eta,\rho)} -1}{i\Phi_{+-}(\eta,\rho)} \Big|^2
  \mathbb{E}|b(\eta)|^2\,\mathbb{E}|b(-\rho)|^2, \\
(2\pi R)^2 \mathbb{E} \big| a_{2,--} \big|^2 & = 2 \sum_{\eta+\rho=\xi} |m_{--}(\eta,\rho)|^2
  \Big| \frac{e^{it\Phi_{--}(\eta,\rho)} -1}{i\Phi_{--}(\eta,\rho)} \Big|^2
  \mathbb{E}|b(-\eta)|^2\,\mathbb{E}|b(-\rho)|^2.
\end{split}
\end{align}

For the cross terms $2\Re \mathbb{E} [\dots]$, we pair one of the positive inputs with $\xi$, and the remaining two inputs with each other. This yields two primary pairing cases: Case A ($\eta = -\rho$, $\sigma = \xi$) and Case B ($\eta = \xi$, $\sigma = -\rho$). 

We evaluate the non-zero contribution from $L_1$ (which comes only from Case A),
\begin{align}\label{a3d1}
\begin{split}
 & 2\Re \mathbb{E} \Big[ \sum_{\eta+\sigma+\rho=\xi} L_1(\xi;\eta,\sigma,\rho) b(\eta) b(\sigma) \overline{b(-\rho)} \cdot \overline{b(\xi)} \Big] \\
 & = - \sum_{\rho} 4\Re \Big[m_{++}(-\rho,\xi+\rho)m_{+-}(\xi,\rho) \frac{e^{it\Phi_{++}(-\rho,\xi+\rho)}-1}{\Phi_{++}^2(-\rho,\xi+\rho)} \Big] \mathbb{E}|b(-\rho)|^2\,\mathbb{E}|b(\xi)|^2.
\end{split}
\end{align}

Similarly, we evaluate the pairing for $L_2$. Due to the symmetry of the symbol $m_{++}$ both Case A and Case B yield identical non-zero contributions, 
\begin{align}\label{a3d2}
\begin{split}
 & 2\Re \mathbb{E} \Big[ \sum_{\eta+\sigma+\rho=\xi} L_2(\xi;\eta,\sigma,\rho) b(\eta) b(\sigma) \overline{b(-\rho)} \cdot \overline{b(\xi)} \Big] \\
 & = - \sum_{\rho} 4\Re \Big[m_{+-}(\xi-\rho,\rho)m_{++}(-\rho,\xi) \frac{e^{it\Phi_{+-}(\xi-\rho,\rho)}-1}{\Phi_{+-}^2(\xi-\rho,\rho)} \Big] \mathbb{E}|b(-\rho)|^2\,\mathbb{E}|b(\xi)|^2.
\end{split}
\end{align}

Similarly, evaluating the pairing for $L_3$ (which is nontrivial only in Case A) we have
\begin{align}\label{a3d3}
\begin{split}
 & 2\Re \mathbb{E} \Big[ \sum_{\eta+\sigma+\rho=\xi} L_3(\xi;\eta,\sigma,\rho) b(\eta) b(\sigma) \overline{b(-\rho)} \cdot \overline{b(\xi)} \Big] \\
 & = \sum_{\rho} 2\Re \Big[ m_{+-}(-\rho,\xi+\rho)m_{+-}(-\rho,-\xi)\frac{e^{it\Phi_{+-}(-\rho,\xi+\rho)}-1}{\Phi_{+-}^2(-\rho,\xi+\rho)} \Big] \mathbb{E}|b(-\rho)|^2\,\mathbb{E}|b(\xi)|^2.
\end{split}
\end{align}

We evaluate the $L_4$ contribution using the symmetry of the symbol $m_{--}$, 
\begin{align}\label{a3d4}
\begin{split}
 & 2\Re \mathbb{E} \Big[ \sum_{\eta+\sigma+\rho=\xi} L_4(\xi;\eta,\sigma,\rho) b(\eta) b(\sigma) \overline{b(-\rho)} \cdot \overline{b(\xi)} \Big] \\
 & =\sum_{\rho} 8\Re \Big[m_{--}(\rho,\xi-\rho) m_{--}(-\xi,\rho)\frac{e^{it\Phi_{--}(\rho,\xi-\rho)}-1}{\Phi_{--}(\rho,\xi-\rho)^2} \Big] \mathbb{E}|b(-\rho)|^2\,\mathbb{E}|b(\xi)|^2.
\end{split}
\end{align}

\subsection{The kinetic limit}\label{sec:kinetic_limit}

We first pass formally to the large-box limit in the degree-four expansion, then extract its term linear in time, and finally make the kinetic closure. Define
\begin{equation}\label{psidef}
\psi(t,\xi):=R^{-1}\mathbb{E} \big|a_{3}(t,\xi)\big|^2
\end{equation}
and let $R\to\infty$ in the retained terms at fixed time, replacing $R^{-1}\sum$ by integration on $\R$. Wick's theorem has already expressed the fourth moments as products of covariances, so no additional decorrelation assumption is needed. Letting $\psi_0:=R^{-1}\mathbb{E}|b|^2=\varepsilon^2\varphi$ and using $\Re(e^{iX}-1)=-2\sin^2(X/2)$ and $|e^{iX}-1|^2=4\sin^2(X/2)$, we obtain the following formal expansion. Here $\simeq$ means that we retain the terms of degrees two and four in $\varepsilon$ and replace the lattice sums by their continuum limits:
\begin{align}\label{kin_intermediate_rho}
\begin{split}
\psi(t,\xi) & \simeq \psi_0(\xi)+ \frac{2}{\pi^2} \int_{\R} |m_{++}(\xi-\rho,\rho)|^2
  \frac{\sin^2\big( t \Phi_{++}(\xi-\rho,\rho)/2 \big)}{\Phi_{++}^2(\xi-\rho,\rho)} 
  \psi_0(\rho)\psi_0(\xi-\rho) \, d\rho
\\
& + \frac{1}{\pi^2} \int_{\R} |m_{+-}(\xi-\rho,\rho)|^2
  \frac{\sin^2\big( t \Phi_{+-}(\xi-\rho,\rho)/2\big)}{\Phi_{+-}^2(\xi-\rho,\rho)} 
  \psi_0(\xi-\rho)\psi_0(-\rho) \, d\rho
\\
& + \frac{2}{\pi^2} \int_{\R} |m_{--}(\xi-\rho,\rho)|^2
  \frac{\sin^2\big( t \Phi_{--}(\xi-\rho,\rho)/2 \big)}{\Phi_{--}^2(\xi-\rho,\rho)} 
  \psi_0(-\rho)\psi_0(\rho-\xi) \, d\rho
\\
& + \frac{2}{\pi^2} \int_{\R} m_{++}(-\rho,\xi+\rho)m_{+-}(\xi,\rho) 
  \frac{\sin^2(t\Phi_{++}(-\rho,\xi+\rho)/2)}{\Phi_{++}^2(-\rho,\xi+\rho)}
 \psi_0(-\rho)\psi_0(\xi) \, d\rho
\\
& + \frac{2}{\pi^2} \int_{\R} m_{+-}(\xi-\rho,\rho)m_{++}(-\rho,\xi) 
  \frac{\sin^2(t\Phi_{+-}(\xi-\rho,\rho)/2)}{\Phi_{+-}^2(\xi-\rho,\rho)} 
 \psi_0(-\rho)\psi_0(\xi) \, d\rho
\\
& - \frac{1}{\pi^2} \int_{\R} m_{+-}(-\rho,\xi+\rho)m_{+-}(-\rho,-\xi) 
  \frac{\sin^2(t\Phi_{+-}(-\rho,\xi+\rho)/2)}{\Phi_{+-}^2(-\rho,\xi+\rho)}
  \psi_0(-\rho)\psi_0(\xi) \, d\rho
\\
& - \frac{4}{\pi^2} \int_{\R} m_{--}(\rho,\xi-\rho) m_{--}(-\xi,\rho) 
  \frac{\sin^2(t\Phi_{--}(\rho,\xi-\rho)/2)}{\Phi_{--}^2(\rho,\xi-\rho)}
  \psi_0(-\rho)\psi_0(\xi) \, d\rho.
\end{split}
\end{align}

To extract the coefficient linear in $t$, we use the distributional identity
\begin{equation*}
\frac1t\frac{\sin^2(ts/2)}{(s/2)^2} \longrightarrow 2\pi\delta_0(s)
\qquad\text{as }t\to\infty.
\end{equation*}
We apply this identity after integration in the interaction variable, at fixed nondegenerate external frequencies $\xi\notin\{0,\pm\sqrt2\}$, where the nontrivial resonant roots are simple. The trivial roots have vanishing amplitudes and give no term linear in $t$. At $\xi=\sqrt2$, the phase $\Phi_{++}(\xi-\rho,\rho)$ has a quadratic zero at $\rho=1/\sqrt2$, and a localized integral with nonzero amplitude there grows like $t^{3/2}$; the reflected statement holds at $\xi=-\sqrt2$. Thus this pointwise delta substitution does not apply at the degenerate frequencies. This does not rule out a weak limit in the external frequency; it is the singularity of the pointwise collision rate that motivates the exceptional set in our solution theory.

The phase $\Phi_{--}(\eta,\rho)=\Lambda(\eta+\rho)+\Lambda(\eta)+\Lambda(\rho)$ has no nontrivial resonances. Its only zero is the all-zero configuration, where the symbols vanish. Hence the third and seventh lines of \eqref{kin_intermediate_rho} give no term linear in $t$.

The coefficient of the term linear in $t$ gives the formal collision rate at the initial density. Replacing $\psi_0$ in this coefficient by the evolving density $\psi(t)$ is the kinetic closure step. For data of size $\psi_0=\varepsilon^2\varphi$, the corresponding slow time is $\tau=\varepsilon^2t$; the rescaled density $n(\tau,\xi)=\varepsilon^{-2}\psi(\varepsilon^{-2}\tau,\xi)$ satisfies the same quadratic kinetic equation. This calculation identifies the collision operator formally and does not establish a kinetic-limit approximation on those time scales. In the unscaled variables, the resulting equation is
\begin{align}\label{kingeneral}
\begin{split}
\partial_t \psi(t,\xi) & =\frac{2}{2\pi} \int_{\R} |m_{++}(\xi-\rho,\rho)|^2\delta_0(\Phi_{++}(\xi-\rho,\rho))
  \psi(t,\rho)\psi(t,\xi-\rho) \, d\rho\\
& + \frac{1}{2\pi} \int_{\R} |m_{+-}(\xi-\rho,\rho)|^2
  \delta_0(\Phi_{+-}(\xi-\rho,\rho))\psi(t,\xi-\rho)\psi(t,-\rho) \, d\rho\\
& + \frac{2}{2\pi} \int_{\R} m_{++}(\rho,\xi-\rho)m_{+-}(\xi,-\rho) 
\delta_0(\Phi_{++}(\rho,\xi-\rho)) \psi(t,\rho)\psi(t,\xi) \, d\rho\\
& + \frac{2}{2\pi} \int_{\R} m_{+-}(\xi+\rho,-\rho)m_{++}(\rho,\xi) 
\delta_0(\Phi_{+-}(\xi+\rho,-\rho))\psi(t,\rho)\psi(t,\xi) \, d\rho\\
& - \frac{1}{2\pi} \int_{\R} m_{+-}(\rho,\xi-\rho)m_{+-}(\rho,-\xi) 
\delta_0(\Phi_{+-}(\rho,\xi-\rho))\psi(t,\rho)\psi(t,\xi) \, d\rho.
\end{split}
\end{align}

We can finally calculate the symbols more precisely on the support of the integrals. Notice that $\delta_0(\Phi_{++}(\xi-\rho,\rho))$ is supported only when $\rho(\xi-\rho)\geq 0$, so either $(\rho\geq 0\text{ and }\xi\geq\rho)$ or $(\rho\leq 0\text{ and }\xi\leq\rho)$. Using the definitions \eqref{WWE6} we evaluate
\begin{equation*}
m_{++}(\xi-\rho,\rho)= \frac{\sqrt{|\xi||\xi-\rho||\rho|}}{4i}\qquad \text{ if } (\rho\geq 0\text{ and }\xi\geq\rho).
\end{equation*}
Similarly, we can evaluate the other symbols on the respective resonant manifolds,
\begin{equation*}
m_{+-}(\xi-\rho,\rho)=\frac{|\xi|\sqrt{|\rho|}\sqrt{1+|\xi|^2}}{2i\sqrt{1+|\xi-\rho|^2}}\qquad\text{ if } (\rho\geq 0\text{ and }\xi\leq 0),
\end{equation*}
\begin{equation*}
m_{++}(\rho,\xi-\rho)= \frac{\sqrt{|\xi||\rho||\xi-\rho|}}{4i}\qquad\text{ if } (\rho\geq 0\text{ and }\xi\geq\rho),
\end{equation*}
\begin{equation*}
m_{+-}(\xi,-\rho)=\frac{|\xi-\rho|\sqrt{|\rho|}\sqrt{1+|\xi-\rho|^2}}{2i\sqrt{1+|\xi|^2}}\qquad\text{ if } (\rho\geq 0\text{ and }\xi\geq\rho),
\end{equation*}
\begin{equation*}
m_{+-}(\xi+\rho,-\rho)= \frac{|\xi|\sqrt{|\rho|}\sqrt{1+|\xi|^2}}{2i\sqrt{1+|\xi+\rho|^2}}\qquad\text{ if } (\rho\geq 0\text{ and }\xi\geq 0),
\end{equation*}
\begin{equation*}
m_{++}(\rho,\xi)=\frac{\sqrt{|\xi+\rho||\rho||\xi|}}{4i}\qquad\text{ if } (\rho\geq 0\text{ and }\xi\geq 0),
\end{equation*}
\begin{equation*}
m_{+-}(\rho,\xi-\rho)= \frac{|\xi|\sqrt{|\xi-\rho|}\sqrt{1+|\xi|^2}}{2i\sqrt{1+|\rho|^2}}\qquad\text{ if } (\xi\geq 0\text{ and }\rho\geq\xi),
\end{equation*}
\begin{equation*}
m_{+-}(\rho,-\xi)=\frac{|\xi-\rho|\sqrt{|\xi|}\sqrt{1+|\xi-\rho|^2}}{2i\sqrt{1+|\rho|^2}}\qquad\text{ if } (\xi\geq 0\text{ and }\rho\geq\xi).
\end{equation*}
The first, third, and fourth lines of \eqref{kingeneral} give $\mathcal Q_1$, $\mathcal Q_3$, and $\mathcal Q_4$ in \eqref{kingeneral2} directly. The changes of variables $\rho_{\rm old}=-\rho$ in the second line and $\rho_{\rm old}=\xi+\rho$ in the fifth line give $\mathcal Q_2$ and $\mathcal Q_5$, respectively. In particular, the products of the purely imaginary symbols give the signs $+,+,-,-,+$, with the common factor $(2\pi)^{-1}$, exactly as in \eqref{kingeneral2}.


\begin{thebibliography}{100}




\bibitem{ABZ2}
T. Alazard, N. Burq, and C. Zuily.
\newblock On the Cauchy problem for gravity water waves.
\newblock {\em Invent. Math.} 198 (2014), 71--163.



\bibitem{ADa}
T. Alazard and J.-M. Delort.
\newblock Global solutions and asymptotic behavior for two dimensional gravity water waves.
\newblock {\em Ann. Sci. \'Ec. Norm. Sup\'er.} (4) 48 (2015), 1149--1238.






\bibitem{AmpColGer}
I. Ampatzoglou, C. Collot, and P. Germain.
\newblock Derivation of the kinetic wave equation for quadratic dispersive problems
in the inhomogeneous setting.
\newblock {\em Amer. J. Math.} 147 (2025), 1053--1158.


\bibitem{Amp}
I. Ampatzoglou.
\newblock Global well-posedness and stability of the inhomogeneous kinetic wave equation near vacuum.
\newblock {\em Kinet. Relat. Models} 17 (2024), 838--854.



\bibitem{AmpLegIll}
I. Ampatzoglou and T. L\'{e}ger.
\newblock On the ill-posedness of kinetic wave equations.
\newblock {\em Nonlinearity} 38 (2025), Paper No. 115004.


\bibitem{AmpLegLr}
I. Ampatzoglou and T. L\'{e}ger.
\newblock On the optimal local well-posedness of the wave kinetic equation in $L^r$.
\newblock {\em J. Differential Equations} 477 (2026), Paper No. 114531.


\bibitem{AmpMPT}
I. Ampatzoglou, J. K. Miller, N. Pavlovi\'c, and M. Taskovi\'c.
\newblock Inhomogeneous wave kinetic equation and its hierarchy in polynomially weighted $L^\infty$ spaces.
\newblock {\em Comm. Partial Differential Equations} 50 (2025), 723--765.




\bibitem{BertiDelort}
M. Berti and J.-M. Delort.
\newblock {\em Almost global solutions of capillary-gravity water waves equations on the circle}.
\newblock Lect. Notes Unione Mat. Ital., vol. 24, Springer, Cham;
Unione Matematica Italiana, Bologna, 2018.


\bibitem{BeFeFr}
M. Berti, R. Feola, and L. Franzoi.
\newblock Quadratic life span of periodic gravity-capillary water waves.
\newblock {\em Water Waves} 3 (2021), 85--115.








\bibitem{BS75}
J. L. Bona and R. Smith.
\newblock The initial-value problem for the Korteweg--de Vries equation.
\newblock {\em Philos. Trans. Roy. Soc. London Ser. A} 278 (1975), 555--601.


\bibitem{BGHS2}
T. Buckmaster, P. Germain, Z. Hani, and J. Shatah.
\newblock Onset of the wave turbulence description of the longtime
behavior of the nonlinear Schr\"odinger equation.
\newblock {\em Invent. Math.} 225 (2021), 787--855.






\bibitem{CL}
D. Christodoulou and H. Lindblad.
\newblock On the motion of the free surface of a liquid.
\newblock {\em Comm. Pure Appl. Math.} 53 (2000), 1536--1602.


\bibitem{ColGer1}
C. Collot and P. Germain.
\newblock On the derivation of the homogeneous kinetic wave equation.
\newblock {\em Comm. Pure Appl. Math.} 78 (2025), 856--909.


\bibitem{ColGer2}
C. Collot and P. Germain.
\newblock Derivation of the homogeneous kinetic wave equation: longer time scales.
\newblock {\em J. Funct. Anal.} 290 (2026), Paper No. 111179.



\bibitem{CS2}
D. Coutand and S. Shkoller.
\newblock Well-posedness of the free-surface incompressible Euler
equations with or without surface tension.
\newblock {\em J. Amer. Math. Soc.} 20 (2007), 829--930.



\bibitem{DeHa1}
Y. Deng and Z. Hani.
\newblock On the derivation of the wave kinetic equation for NLS.
\newblock {\em Forum Math. Pi} 9 (2021), Paper No. e6.


\bibitem{DeHa2}
Y. Deng and Z. Hani.
\newblock Full derivation of the wave kinetic equation.
\newblock {\em Invent. Math.} 233 (2023), 543--724.




\bibitem{DeHa4}
Y. Deng and Z. Hani.
\newblock Derivation of the wave kinetic equation: full range of scaling laws.
\newblock {\em Mem. Amer. Math. Soc.} 320 (2026), no.~1631.






\bibitem{DIPP}
Y. Deng, A. D. Ionescu, B. Pausader, and F. Pusateri.
\newblock Global solutions of the gravity-capillary water-wave system in three dimensions.
\newblock {\em Acta Math.} 219 (2017), 213--402.



\bibitem{DIP}
Y. Deng, A. D. Ionescu, and F. Pusateri.
\newblock On the wave turbulence theory of 2D gravity waves, I: deterministic energy estimates.
\newblock {\em Comm. Pure Appl. Math.} 78 (2025), 211--322.


\bibitem{DIP2}
Y. Deng, A. D. Ionescu, and F. Pusateri.
\newblock On the wave turbulence theory of 2D gravity waves, II:
propagation of randomness.
\newblock Preprint, arXiv:2504.14304 (2025).


\bibitem{GLZ-MMT}
P. Germain, J. La, and K. Z. Zhang.
\newblock Local well-posedness for the kinetic MMT model.
\newblock {\em Comm. Math. Phys.} 406 (2025), Paper No. 18.


\bibitem{GIT}
P. Germain, A. D. Ionescu, and M.-B. Tran.
\newblock Optimal local well-posedness theory for the kinetic wave equation.
\newblock {\em J. Funct. Anal.} 279 (2020), Paper No. 108570.


\bibitem{GMS2}
P. Germain, N. Masmoudi, and J. Shatah.
\newblock Global solutions for the gravity water waves equation in dimension 3.
\newblock {\em Ann. of Math.} (2) 175 (2012), 691--754.


\bibitem{Has1}
K. Hasselmann.
\newblock On the non-linear energy transfer in a gravity-wave spectrum.
Part 1. General theory.
\newblock {\em J. Fluid Mech.} 12 (1962), 481--500.

\bibitem{Has2}
K. Hasselmann.
\newblock On the non-linear energy transfer in a gravity wave spectrum.
Part 2. Conservation theorems; wave-particle analogy; irreversibility.
\newblock {\em J. Fluid Mech.} 15 (1963), 273--281.



\bibitem{IT}
M. Ifrim and D. Tataru.
\newblock Two dimensional water waves in holomorphic coordinates II: global solutions.
\newblock {\em Bull. Soc. Math. France} 144 (2016), 369--394.




\bibitem{IoPu2}
A. D. Ionescu and F. Pusateri.
\newblock Global solutions for the gravity water waves system in 2d.
\newblock {\em Invent. Math.} 199 (2015), 653--804.




\bibitem{IoPu4}
A. D. Ionescu and F. Pusateri.
\newblock Global regularity for 2D water waves with surface tension.
\newblock {\em Mem. Amer. Math. Soc.} 256 (2018), no.~1227.


\bibitem{IoPuRev}
A. D. Ionescu and F. Pusateri.
\newblock Recent advances on the global regularity for irrotational water waves.
\newblock {\em Philos. Trans. Roy. Soc. A} 376 (2018), Paper No. 20170089.


\bibitem{IoPu5}
A. D. Ionescu and F. Pusateri.
\newblock Long-time existence for multi-dimensional periodic water waves.
\newblock {\em Geom. Funct. Anal.} 29 (2019), 811--870.


\bibitem{Lannes}
D. Lannes.
\newblock Well-posedness of the water-waves equations.
\newblock {\em J. Amer. Math. Soc.} 18 (2005), 605--654.


\bibitem{LannesBook}
D. Lannes.
\newblock {\em The water waves problem. Mathematical analysis and asymptotics}.
\newblock Math. Surveys Monogr., vol. 188, American Mathematical Society,
Providence, RI, 2013.


\bibitem{XMa}
X. Ma.
\newblock Almost sharp wave kinetic theory of multidimensional KdV
type equations with $d\geq 3$.
\newblock Preprint, arXiv:2204.06148 (2022).


\bibitem{MMT}
A. J. Majda, D. W. McLaughlin, and E. G. Tabak.
\newblock A one-dimensional model for dispersive wave turbulence.
\newblock {\em J. Nonlinear Sci.} 7 (1997), 9--44.


\bibitem{Naz}
S. Nazarenko.
\newblock {\em Wave turbulence}.
\newblock Lecture Notes in Phys., vol. 825, Springer, Heidelberg, 2011.

\bibitem{NgTr}
T. T. Nguyen and M.-B. Tran.
\newblock On the kinetic equation in Zakharov's wave turbulence theory for capillary waves.
\newblock {\em SIAM J. Math. Anal.} 50 (2018), 2020--2047.


\bibitem{PanWu}
Y. Pan and X. Wu.
\newblock Local-in-time existence of $L^1$ solutions to the gravity water wave kinetic equation.
\newblock Preprint, arXiv:2603.10882 (2026).


\bibitem{Pei}
R. Peierls.
\newblock Zur kinetischen Theorie der W\"armeleitung in Kristallen.
\newblock {\em Ann. Phys.} 395 (1929), 1055--1101.


\bibitem{ShZ3}
J. Shatah and C. Zeng.
\newblock Local well-posedness for fluid interface problems.
\newblock {\em Arch. Ration. Mech. Anal.} 199 (2011), 653--705.



\bibitem{StafTran}
G. Staffilani and M.-B. Tran.
\newblock On the wave turbulence theory for a stochastic KdV type equation.
\newblock Preprint, arXiv:2106.09819 (2021).


\bibitem{SulemBook}
C. Sulem and P.-L. Sulem.
\newblock {\em The nonlinear Schr\"odinger equation. Self-focusing and wave collapse}.
\newblock Appl. Math. Sci., vol. 139, Springer-Verlag, New York, 1999.


\bibitem{Va}
K. D. Vassilev.
\newblock One-dimensional wave kinetic theory.
\newblock {\em Comm. Math. Phys.} 406 (2025), Paper No. 293.







\bibitem{WangCapillary}
J. Wang.
\newblock Local existence of strong solutions to the capillary wave kinetic equation.
\newblock Preprint, arXiv:2607.04263 (2026).



\bibitem{Wu2}
S. Wu.
\newblock Well-posedness in Sobolev spaces of the full water wave problem in 3-D.
\newblock {\em J. Amer. Math. Soc.} 12 (1999), 445--495.




\bibitem{Wu3DWW}
S. Wu.
\newblock Global wellposedness of the 3-D full water wave problem.
\newblock {\em Invent. Math.} 184 (2011), 125--220.




\bibitem{Zak0}
V. E. Zakharov.
\newblock Stability of periodic waves of finite amplitude on the surface of a deep fluid.
\newblock {\em Prikl. Mekh. Tekh. Fiz.} 9 (1968), no.~2, 86--94;
English translation, {\em J. Appl. Mech. Tech. Phys.} 9 (1968), 190--194.





\bibitem{ZhengWW}
F. Zheng.
\newblock Long-term regularity of 3D gravity water waves.
\newblock {\em Comm. Pure Appl. Math.} 75 (2022), 1074--1180.




\end{thebibliography}
\end{document}